\documentclass[10pt]{amsart}
\usepackage[utf8]{inputenc}
\usepackage[a4paper,margin=3cm]{geometry}
\usepackage[pdftex,pdfborder={0 0 0},colorlinks]{hyperref}
\usepackage{latexsym,lmodern,graphicx,enumitem,amssymb,dsfont}

\newtheorem{thm}[equation]{Theorem}
\newtheorem{lem}[equation]{Lemma}
\newtheorem{prop}[equation]{Proposition}
\newtheorem{cor}[equation]{Corollary}
\newtheorem{defi}[equation]{Definition}

\newtheorem{rem}[equation]{Remark}

\numberwithin{equation}{section}

\newcommand{\1}{\mathbf{1}}
\newcommand{\R}{\mathbb{R}}
\newcommand{\X}{\mathcal{X}}
\newcommand{\N}{\mathbb{N}}

\newcommand{\T}{\mathbf{T}}
\newcommand{\E}{\mathbb{E}}
\renewcommand{\P}{\mathbb{P}}

\newcommand{\ent}{\operatorname{Ent}}

\newcommand{\Tw}{\widetilde{\mathcal{T}}}
\newcommand{\Th}{\widehat{\mathcal{T}}}
\newcommand{\Tb}{\overline{\mathcal{T}}}

\begin{document}

\title[Kantorovich duality for general costs and applications]
{Kantorovich duality for general transport costs and applications}

\author[N. Gozlan, C. Roberto, P.-M. Samson, P. Tetali]{Nathael Gozlan, Cyril Roberto, Paul-Marie Samson, Prasad Tetali}

\date{\today}

\thanks{Supported by the grants ANR 2011 BS01 007 01, ANR 10 LABX-58, ANR11-LBX-0023-01; the last author is supported 
by the NSF grants DMS 1101447 and 1407657, and is also grateful for the hospitality of Universit\'e Paris Est 
Marne La Vall\'ee. All authors acknowledge the kind support of the American Institute of Mathematics (AIM, California).}

\address{Universit\'e Paris Est Marne la Vall\'ee - Laboratoire d'Analyse et de Math\'e\-matiques Appliqu\'ees (UMR CNRS 8050), 5 bd Descartes, 77454 Marne la Vall\'ee Cedex 2, France}
\address{Universit\'e Paris Ouest Nanterre La D\'efense - Modal'X, 200 avenue de la R\'epublique 92000 Nanterre, France}
\address{School of Mathematics \& School of Computer Science, Georgia Institute of Technology,
Atlanta, GA 30332}
\email{nathael.gozlan@univ-mlv.fr, croberto@math.cnrs.fr, paul-marie.samson@univ-mlv.fr,\linebreak tetali@math.gatech.edu}
\keywords{Duality, Transport inequalities, logarithmic-Sobolev inequalities, metric spaces}
\subjclass{60E15, 32F32 and 26D10}

\begin{abstract}
We introduce a general notion of  transport cost that encompasses many costs used in the literature (including the classical one and weak transport costs introduced by Talagrand and Marton in the 90's), and prove a Kantorovich type duality theorem. As a by-product we obtain various applications in different directions: we give a short proof of a result by Strassen on the existence of a martingale with given marginals, we characterize the associated transport-entropy inequalities together with the log-Sobolev inequality restricted to convex/concave functions. Some explicit examples of discrete measures satisfying weak transport-entropy inequalities are also given.
\end{abstract}

\maketitle


\section{Introduction}

Concentration of measure phenomenon was introduced in the seventies by V. Milman \cite{Mil71} in his study of asymptotic geometry of Banach spaces. It was then studied in depth by many authors including Gromov \cite{GM83,Gro99}, Talagrand \cite{Tal95}, Maurey \cite{Mau91}, Ledoux \cite{Led95,BL97}, Bobkov \cite{Bob97,BL00} and many others and played a decisive role in analysis, probability and statistics in high dimensions. We refer to the monographs \cite{Led01} and \cite{BLM13} for an overview of the field. 

One classical example of such phenomenon can be observed for the standard Gaussian measure $\gamma_m$ on $\R^m.$ It follows from the well known Sudakov-Tsirelson-Borell isoperimetric result in Gauss space \cite{ST74,Bor75} that if $X_1,\ldots,X_n$ are $n$ i.i.d random vectors with law $\gamma_m$ and $f : (\R^m)^n \to \R$ is a $1$-Lipschitz function (with respect to the Euclidean norm), then
\begin{equation}\label{eq:Conc-Gauss}
\P(f(X_1,\ldots,X_n) > m+ t) \leq e^{-(t-t_o)^2/(2a)},\qquad \forall t\ge t_o,
\end{equation}
with $a=1$ and $t_o=0$, and where $m$ denotes the median of the random variable $f(X_1,\ldots,X_n)$.
The remarkable feature of this inequality is that it does not depend on the sample size $n$. This property was used in numerous applications \cite{Led01}.

The standard Gaussian measure is far from being the only example of a probability distribution satisfying such a bound. In this introduction, we will say that a probability $\mu$ on some metric space $(X,d)$ satisfies the Gaussian dimension-free concentration of measure phenomenon if \eqref{eq:Conc-Gauss} holds true, with a constant $a$ independent on $n$, when the $X_i$'s are distributed according to $\mu$ and $f$ is a function which is $1$-Lipschitz with respect to the distance $d_2$ defined on $X^n$ by
\[
d_2(x,y) = \left[\sum_{i=1}^n d(x_i,y_i)^2\right]^{1/2},\qquad x,y \in X^n.
\]
This is also equivalent to the following property: for all positive integers $n$, and all Borel set $A \subset X^n$ such that $\mu^n(A) \geq 1/2$, it holds
\begin{equation}\label{eq:Conc-Gauss2}
\mu^n(A_t) \geq 1-e^{-(t-t_o)^2/(2a)},\qquad \forall t\geq t_o,
\end{equation}
where $A_t=\{y \in X^n : \exists x \in A, d_2(x,y) \leq t\}$.

For instance if $\mu$ is a probability measure on $\R^m$, or even more generally on a smooth Riemannian manifold $M$ equipped with its geodesic distance $d$ and has a density of the form $e^{-V}$, where $V$ is some smooth function on $M$ such that the so-called Bakry-\'Emery curvature condition holds
\begin{equation}\label{eq:BE}
\mathrm{Ric}+\mathrm{Hess}\, V \geq K\mathrm{Id}\,,
\end{equation}
for some $K>0$, then the Gaussian dimension-free concentration of measure phenomenon holds with the constant $a=K$ (a direct proof can be found in \cite{Led01}).

Another very classical sufficient condition for the Gaussian concentration of measure property \eqref{eq:Conc-Gauss} is the Logarithmic Sobolev inequality introduced by Gross \cite{G75} (see also Stam \cite{Sta59} and Federbush \cite{Fed69}). If, for some $C>0$, $\mu$ satisfies 
\begin{equation}\label{eq:LSI}
\ent_\mu(f^2) \leq 2C \int |\nabla f|^2\,d\mu,
\end{equation}
for all smooth functions $f: M \to \R$, then it satisfies \eqref{eq:Conc-Gauss} with $a = C$ (a proof of this classical result due to Herbst can be found in \cite{Led01}). 
We recall that the entropy functional of a positive function $g$ is defined by $\ent_\mu(g) = \int g\log \left(\frac{g}{\int g\,d\mu}\right)\,d\mu$. Condition \eqref{eq:LSI} - denoted $\mathbf{LSI}(C)$ in the sequel - is less restrictive since, according to the famous Bakry-\'Emery criterion \eqref{eq:BE} implies \eqref{eq:LSI}.

It turns out that Condition \eqref{eq:LSI} can be further relaxed. Indeed, in \cite{Tal96a}, Talagrand introduced another remarkable functional inequality involving the Wasserstein distance $W_2$ defined, for all probability mesures $\mu,\nu$ on $M$ by
\begin{equation}\label{eq:Tr2}
W_2^2(\nu,\mu) = \inf_{(X,Y)} \E[d^2(X,Y)],
\end{equation}
where the infimum runs over all pairs of random variables $(X,Y)$, with $X$ distributed according to $\mu$ and $Y$ according to $\nu.$ A probability measure $\mu$ satisfies Talagrand's transport inequality $\T_2(D)$ for some $D>0$, if
\begin{equation}\label{eq:T2}
W_2^2(\nu,\mu) \leq 2D H(\nu|\mu),
\end{equation}
for all probability measure $\nu$ on $M$, where $H(\nu|\mu)$ denotes the relative entropy defined by $H(\nu|\mu) = \ent_\mu (h)$, $h = d\nu/d\mu$ if $\nu \ll \mu$ (i.e., $\nu$ absolutely continuous with respect to $\mu$) and $+\infty$ otherwise.
A nice argument first discovered by Marton \cite{Mar86} shows that \eqref{eq:T2} is a sufficient condition for the Gaussian dimension-free concentration property \eqref{eq:Conc-Gauss} with $a=D$. 
One crucial ingredient to derive dimension-free concentration from $\T_2$ is the tensorization property enjoyed by this inequality (the same property holds for $\mathbf{LSI}$): if $\mu$ satisfies $\T_2(C)$, then for any positive integer $n$, the product measure $\mu^n$ also satisfies $\T_2(C).$ 
Condition \eqref{eq:T2} is again an improvement upon Condition \eqref{eq:LSI} since it was proved by Otto and Villani \cite{OV00} (see \cite{BGL01} for an alternative proof and \cite{GRS14} and the references therein for extensions to more general spaces) that \eqref{eq:LSI} implies \eqref{eq:T2} with 
$D=C$. It was then shown by the first author \cite{Goz09} that Condition \eqref{eq:T2} was not only sufficient but also necessary for Gaussian dimension-free concentration. More precisely, if \eqref{eq:Conc-Gauss} holds true with some $a$ (and all $n$), then $\mu$ satisfies $\T_2(a)$.

One of the main motivations behind this work, and a few satellite papers by the same authors and Y. Shu \cite{GRST14,Shu, GRSST15, GRST16}, is to understand what can replace each term in the chain of implications:
\[
\eqref{eq:BE}\Rightarrow \eqref{eq:LSI}\Rightarrow \eqref{eq:T2} \Leftrightarrow \eqref{eq:Conc-Gauss}
\] 
in a \emph{discrete} setting (for instance, when the space is a graph, finite or otherwise).

While several useful variants of the logarithmic Sobolev inequality are well identified in discrete (involving different natural discrete gradients, see \textit{e.g.} \cite{SC97, BT06}), the other terms are far from being understood.

After the works by Lott-Villani  \cite{LV09} and Sturm \cite{Stu06} extending \eqref{eq:BE} to non-smooth geodesic spaces through convexity properties of the entropy functional on the space of probability measures equipped with the Wasserstein distance $W_2$, the question of generalizing the Bakry-\'Emery condition in a discrete setting attracted in recent years a lot of attention. We refer to the works by Ollivier \cite{Oll09}, Bonciocat-Sturm \cite{BS09}, Ollivier-Villani \cite{OV12}, Erbar-Mass \cite{EM12}, Hillion \cite{Hil12} and the work \cite{GRST14} by the authors for different attempts to give a meaning to the notion of ``discrete curvature''.

In the present paper, the focus is put on the rightmost terms of our chain of implications: namely, our purpose is to find out what type of dimension-free concentration results we can hope for in a discrete setting and what type of transport inequalities can be related to it. At this stage, it is worth noting that, unfortunately, Talagrand's inequality is \emph{never} satisfied in discrete (except of course by a Dirac mass). For instance, it is proven in full generality in \cite{GRS14}, that if $\mu$ is a probability measure on a metric space $(X,d)$ which satisfies $\T_2$, then its support is connected. It follows from the equivalence \eqref{eq:T2} $\Leftrightarrow$ \eqref{eq:Conc-Gauss}, that Gaussian dimension-free concentration is also never true in discrete.

One thus looks for a transport-cost sufficiently weaker than $W_2^2$, to allow discrete measures to satisfy the related transport inequality, but sufficiently strong to make the transport inequality stable under tensor products. A natural candidate would be the $W_1$ distance: $W_1(\nu,\mu) := \inf \{ \E[d(X,Y)] : \mathrm{Law}(X)=\mu, \mathrm{Law}(Y)=\nu\}.$ Although transport inequalities involving  $W_1^2$ instead of $W_2^2$ (the so called $\T_1$ inequalities) make perfectly sense in discrete (see Bobkov-G\"otze \cite{BG99}, Djelout-Guillin-Wu \cite{DGW04}, Bobkov-Houdr\'e-Tetali \cite{BHT06}), these inequalities tensorize only with a constant depending on the dimension! So the $W_1$ distance does not fulfill the second requirement. 

The present paper is devoted to the study of a family of weak transport costs, one typical element of which is the following weak version of the cost $W_2^2$ defined as follows. If $\mu$ and $\nu$ are probability measures on a metric space $(X,d)$, one defines the weak cost $\Tw_2(\nu|\mu)$ as follows
\[
\Tw_2(\nu|\mu) = \inf_{(X,Y)}\E \left[\,\E[d(X,Y) | X]^2\,\right]
\]
where again the infimum runs over all pairs ($X,Y$) of random variables such that $X$ follows the law $\mu$ and $Y$ the law $\nu.$
Jensen inequality immediately shows that 
\[
W_1^2(\nu,\mu) \leq \Tw_2(\nu|\mu) \leq W_2^2(\nu,\mu).
\]
Two weak versions of Talagrand's inequality are naturally associated to this cost: a probability $\mu$ is said to satisfy $\widetilde{\T}_2^-(C)$ for some $C>0$ if
\[
\Tw_2(\mu|\nu) \leq CH(\nu|\mu),\qquad \forall \nu 
\]
and to satisfy $\widetilde{\T}_2^+(C)$ if 
\[
\Tw_2(\nu|\mu) \leq CH(\nu|\mu),\qquad \forall \nu.
\]
Since $\Tw_2$ is not symmetric these two inequalities are not equivalent in general. Both are of course implied by the usual $\T_2(C)$ inequality. As we shall see in Theorem~\ref{gozlan}, which is one of our main results, a probability measure $\mu$ on $X$ satisfies the two inequalities $\widetilde{\T}^{\pm}_2(C)$ for some $C>0$ if and only if it satisfies the following dimension-free concentration of measure property: for all positive integer $n$ and all set $A \subset X^n$ such that $\mu^n(A)>0$, it holds
\begin{equation}\label{eq:Conc-tilde}
\mu^n(\widetilde{A}_t) \geq 1-\frac{1}{\mu^n(A)}e^{-t^2/D},\qquad \forall t\geq 0,
\end{equation}
for some $D$ related to $C$.  In this concentration inequality the enlargement $\widetilde{A}_t$ of $A$ is defined as follows 
\begin{align*}
\widetilde{A}_t = \{  y \in X^n : \exists p \in \mathcal{P}(X^n)   \text{ with }  p(A)& =1  \text{ such that }  \sum_{i=1}^n\left(\int d(x_i,y_i)\,p(dx)\right)^2 \leq t^2 \},
\end{align*}
where $\mathcal{P}(X^n)$ is the set of all Borel probability measures on $X^n.$ Taking $p = \delta_x$ with $x \in A$, we see immediately that $A_t \subset \widetilde{A}_t$ and therefore \eqref{eq:Conc-tilde} is less demanding than \eqref{eq:Conc-Gauss2}.

Before going further into the presentation of our results, let us make some bibliographical comments on these weak transport costs and on the concentration property \eqref{eq:Conc-tilde}. First of all, this way of enlarging sets first appeared in the papers \cite{Tal95,Tal96b} by Talagrand, in the particular case where $d(x,y) = \1_{x\neq y}$ is the Hamming distance (see \cite[Theorem 4.1.1]{Tal95} and \cite[(1.2)]{Tal96b}). It was shown by Talagrand that \emph{any} probability measure $\mu$ on a polish space $X$ satisfies the concentration inequality \eqref{eq:Conc-tilde} with some universal constant $D$ (and with the Hamming distance). This deep result known as Talagrand's convex hull concentration inequality has had a lot of interesting applications in probability theory and combinatorics \cite{Tal95,Led01,AS08}. The result was given another proof by Marton in \cite{Mar96a}, where she introduced (again with $d$ being the Hamming distance) the weak transport cost $\Tw_2$ (denoted $\bar{d}_2$ in her work) and proved that any probability measure $\mu$ satisfies 
\[
\Tw_2(\nu_1|\nu_2)^{1/2} \leq \left(2H(\nu_1|\mu)\right)^{1/2} + \left(2H(\nu_2 |\mu)\right)^{1/2},
\]
for all probability measures $\nu_1,\nu_2.$
Then she proved the tensorization property for this transport inequality and derived from it, using an argument that will be recalled in Section \ref{sec:concentration}, Talagrand's concentration result. A similar strategy was then developed by Dembo in \cite{Dem97} in order to recover the sharp form of other concentration results by Talagrand involving a control by $q$ points. Finally the third named author extended in \cite{Sam00} the tensorization technique of Marton to some classes of dependent random variables. In \cite{Sam07} he improved Marton's transport inequality to recover yet another sharp concentration inequality by Talagrand (discovered in \cite{Tal96b}) related to deviation inequalities for empirical processes. Besides the Hamming case, almost nothing is known on the inequalities $\widetilde{\T}^\pm_2$. Note that Marton's result shows at least that any probability measure on a bounded metric space $(X,d)$ (for instance a finite graph equipped with graph distance) satisfies $\widetilde{\T}_2^{\pm}(C)$ for $C = 2 \mathrm{Diam}(X)$, but the optimal constant $C$ can  be much smaller. Note also that if our primary motivation was to consider these inequalities on discrete spaces, the continuous case is also of interest since the inequalities $\widetilde{\T}^\pm_2$ could be a good substitute for probability measures not satisfying the usual $\T_2$. The aim of the paper is thus to provide different tools that can be useful in the study of these weak transport cost inequalities and to exhibit some new examples of such inequalities (mostly on unbounded spaces).

The main new tool we introduce is a version of the Kantorovich duality theorem suitable for the weak transport cost $\Tw_2$. Actually this duality result holds for a large family of transport costs that we shall now describe. To each cost function $c:X \times \mathcal{P}(X) \to [0,\infty]$, (where $\mathcal{P}(X)$ is the set of Borel probability measures on $X$) we associate the optimal transport cost $\mathcal{T}_c$ defined, for all probability measures $\mu,\nu$ on $X$, by
\begin{equation}\label{eq:introOT}
\mathcal{T}_c(\nu|\mu) = \inf_p \int c(x,p_x)\,\mu(dx),
\end{equation}
where the infimum runs over the set of all probability kernels $p : X \to \mathcal{P}(X): x \mapsto p_x(\,\cdot\,)$ such that $\mu p =\nu$. Note that the usual cost $W_2^2$ corresponds to $c(x,p) = \int d^2(x,y)\,p(dy)$ and the weak cost $\Tw_2$ to $c(x,p)=\left(\int d(x,y)\,p(dy)\right)^2$. Under some easily satisfiable technical assumptions on $c$ (an important one being that $c$ be \emph{convex} with respect to its second variable $p$), we prove in Theorem \ref{Kantorovich general} that 
\[
\mathcal{T}_c(\nu|\mu) = \sup\left\{\int R_c\varphi (x)\,\mu(dx) - \int \varphi(y)\,\nu(dy) \right\},
\]
where the supremum runs over the set of bounded continuous functions, and 
\[
R_c\varphi(x) = \inf_{p \in \mathcal{P}(X)} \left\{ \int \varphi(y)\,p(dy) + c(x,p)\right\},\qquad x\in X.
\]
Note that, when $c(x,p)=\int d^2(x,y)p(dy)$, then $R_c\varphi(x) = \inf_{y\in X}\{\varphi(y) + d^2(x,y)\}$ and the result reduces to the classical Kantorovich duality for $W_2^2$ (see \textit{e.g.} \cite{Vil03,Vil09}). Up to our best knowledge, this class of cost functionals has not been considered before in the literature on Optimal Transport but we think that it may find interesting applications in this field. For example, denoting by $\overline{\mathcal{T}}_p$ the weak cost associated to the cost function $c(x,p) = \|x - \int y\,p(dy)\|^p$ defined on $\R^m \times \mathcal{P}(\R^m)$ where $\|\,\cdot\,\|$ is some norm on $\R^m$, it turns out that the duality formula for $\overline{\mathcal{T}}_1$ immediately gives back a well known result by Strassen \cite{Str65} about the existence of martingales with given marginals. This is detailed in Section \ref{Sec:Strassen}.

The paper is organized as follows.

Section 2 introduces a general definition of optimal transport costs and presents in detail three particular families of costs (all variants of Marton's costs $\Tw_2$ defined above) which will play a role in the rest of the paper. In particular, we state a Kantorovich duality formula for each of these transport costs.

Section 3 is dedicated to the proof of Strassen's theorem on the existence of martingales with given marginals.

Section 4 introduces the general definition of transport-entropy inequalities (involving general transport costs of the form \eqref{eq:introOT}) and presents their basic properties such as their dual formulation and their tensorization.

Section 5 deals with the links between concentration of measure and transport-entropy inequalities. We recall in particular the argument due to Marton that enables to deduce concentration estimates from transport-entropy inequalities. We also extend to this general framework a result by the first author and show that in great generality dimension-free concentration gives back transport-entropy inequalities. In particular, we give a characterization (in terms of a transport-entropy inequality involving the cost $\Tb_2$ defined above) of dimension-free Gaussian concentration \eqref{eq:Conc-Gauss} restricted to Lipschitz convex (or concave) functions.

In Section 6 we recall the universal transport-entropy inequalities developed by Marton \cite{Mar96a, Mar96b}, Dembo \cite{Dem97} and Samson \cite{Sam03,Sam07} in order to recover some of Talagrand's concentration inequalities for product measures. We take advantage of our duality theorems to revisit and bring some simplifications in the proof of \cite{Sam07}.

In Section 7 we present the examples of Bernoulli, Binomial and Poisson laws  for which  some sharp transport-entropy are proved in \cite{GRST16}.

In Section 8 we show the equivalence between transport-entropy inequalities involving the transport cost $\Tb_2$ and the logarithmic-Sobolev inequality restricted to the class of log-convex or log-concave functions. This enables us to get other examples for these transport inequalities.

Finally, Section 9 contains the proof of our general Kantorovich duality result, Theorem~\ref{Kantorovich general}, for a transport cost of the form \eqref{eq:introOT}.

\tableofcontents


\section{Optimal transport costs and duality}
In this section, we introduce a general class of optimal transport costs and describe an associated Kantorovich type duality formula.
 
\subsection{Notations}\label{Notation}
Throughout the paper $(X,d)$ is a complete separable metric space. The Borel $\sigma$-field will be denoted by $\mathcal{B}.$
The space of all Borel probability measures on $X$ is denoted by $\mathcal{P}(X)$. 

If $\gamma:\R_+ \to \R_+$ is a lower-semicontinuous function satisfying 
\begin{equation}\label{gamma}
\gamma(0)=0 \quad \text{ and } \quad \gamma(u+v)\leq C(\gamma(u)+\gamma(v)), \qquad u,v \in \mathbb{R}_+,
\end{equation}
for some constant $C$, then we set
\[
\mathcal{P}_\gamma(X):= \left\{\mu \in \mathcal{P}(X) ; \int \gamma(d(x,x_o))\,\mu(dx)<\infty\right\}
\] 
for some (hence all) $x_o \in X$. In the specific cases where $\gamma_r(u):=u^r$, $u\geq 0$, $r>0$, we use the simpler notation  $\mathcal{P}_r(X) := \mathcal{P}_{\gamma_r}(X)$. We shall also consider the limit case $\gamma_0(u):=\1_{u\neq  0}$, $u \geq 0$, for which $\mathcal{P}_{\gamma_0}(X) = \mathcal{P}(X).$

We also denote by $\Phi_\gamma(X)$ (resp. $\Phi_{\gamma,b}(X)$) the set of continuous  (resp. continuous and bounded from below) functions $\varphi:X\to \R$ satisfying the growth condition 
\begin{equation}\label{eq:growth condition}
|\varphi(x)|\leq a + b \gamma(d(x,x_o)),\qquad \forall x \in X\,,
\end{equation}
for some $a,b \geq0$ and some (hence all) $x_o\in X$. 

The spaces $\Phi_\gamma(X\times X)$ and $\mathcal{P}_\gamma (X\times X)$ are defined accordingly, with $X \times X$ equipped with (say) the $\ell_1$ product metric. 

The space $\mathcal{P}_\gamma(X)$ will always be equipped with the $\sigma$-field $\mathcal{F}_\gamma$ generated by the maps 
\[
\mathcal{P}_\gamma(X) \to [0,1] : \nu \mapsto \nu(A),
\]
where $A$ is a Borel set of $X.$
In particular, one says that $p: X \to \mathcal{P}_\gamma(X) : x \mapsto p_x$ is a kernel if it is measurable with respect to the Borel $\sigma$-field $\mathcal{B}$ on $X$ and the $\sigma$-field $\mathcal{F}_\gamma$ on $\mathcal{P}_\gamma(X)$. This amounts to requiring that, for all $A \in \mathcal{B}$, the map $X \to [0,1] : x \mapsto p_x(A)$ be Borel measurable.

%
%
%
%
%

%
%

\subsection{Costs functions, couplings and weak optimal transport costs}

In this paper, a \emph{cost function} will be a measurable function $c \colon X \times \mathcal{P}_\gamma(X) \to [0,\infty]$,
for some fixed $\gamma$ satisfying \eqref{gamma}. For all $\pi \in \mathcal{P}_\gamma(X\times X)$, we set
\[
I_c[\pi] = \int c(x,p_x)\,\pi_1(dx),
\]
where $\pi_1$ is the first marginal of $\pi$ and $x\mapsto p_x$ the ($\pi_1$-almost everywhere, uniquely determined) probability kernel such that 
\[
\pi(dxdy)= \pi_1(dx)p_x(dy).
\]
Note that if $\pi \in \mathcal{P}_\gamma(X\times X)$, then $p_x \in \mathcal{P}_\gamma(X)$ for $\pi_1$ almost all $x\in X$ and thus the preceding definition makes sense.

Given two probability measures $\mu$ and $\nu$ on $X$, we denote by
\[
\Pi(\mu,\nu)=\{\pi\in \mathcal{P}(X\times X) ; \pi(dx\times X)=\mu(dx)\text{ and } \pi(X\times dy)=\nu(dy)\}
\]
the set of all \emph{couplings} $\pi$ whose first marginal is $\mu$ and whose second marginal is $\nu$. Note also that if, $\mu,\nu \in \mathcal{P}_\gamma(X)$, then $\Pi(\mu,\nu) \subset \mathcal{P}_\gamma(X\times X).$

Using the above notations, we introduce an extension of the well-known Monge-Kantorovich optimal transport costs as follows.
\begin{defi}\label{OTnew}
Let $c \colon X \times \mathcal{P}_\gamma(X)\to [0,\infty]$ and $\mu,\nu \in \mathcal{P}_\gamma(X)$. The optimal transport cost $\mathcal{T}_c(\nu|\mu)$ between $\mu$ and $\nu$ is defined by
\[
\mathcal{T}_c(\nu|\mu):=\inf_{\pi \in \Pi(\mu,\nu)} I_c[\pi]= \inf_{\pi \in \Pi(\mu,\nu)} \int c(x,p_x)\,\mu(dx) .
\]
\end{defi}

Let us first remark that optimal transport costs in the classical sense (see e.g. \cite{Vil03,Vil09}) enter the framework of this definition. Namely, if $\omega : X \times X \to [0,\infty]$ is some measurable cost function, and $c(x,p) = \int \omega(x,y)\,p(dy)$, for all $x \in X$ and $p \in \mathcal{P}(X)$, then it is clear that
\[
\mathcal{T}_c(\nu |\mu) = \inf \left\{\iint \omega(x,y)\,\pi(dxdy) : \pi \in \Pi(\mu,\nu)\right\},
\]
which is the usual optimal transport cost related to the cost function $\omega.$ In the sequel, we will denote by $\mathcal{T}_\omega(\nu,\mu)$ the usual Monge-Kantorovich optimal transport cost, defined by the right hand side above.
One sees that while in the usual definition every elementary transport of mass from $\mu$ to $\nu$ represented by $p_x$ is penalized by its mean cost $\int \omega (x,y)\,p_x(dy)$, our definition allows other types of penalization. See Section \ref{sec:particular cases} below for some examples.

\subsection{A Kantorovich type duality}
If $\omega : X \times X \to [0,\infty]$ is lower semi-continuous, then according to the well known Kantorovich duality theorem (see for instance \cite[Theorem 5.10]{Vil09}), it holds
\[
\mathcal{T}_\omega(\nu,\mu) = \sup\left\{ \int \psi(x)\,\mu(dx) - \int \varphi(y)\,\nu(dy)\right\},
\]
where the supremum runs over the class of pairs $(\psi,\varphi)$ of bounded continuous functions  on $X$ such that
\[
\psi(x) - \varphi(y) \leq \omega(x,y),\qquad \forall x,y \in X.
\]
A classical and simple argument shows that one can always replace $\psi$ by the function $Q_\omega \varphi$ defined by
\[
Q_\omega \varphi(x) = \inf_{y \in X}\{ \varphi(y) +\omega(x,y)\},\qquad x\in X.
\]
Therefore, the duality formula above can be restated as follows
\[
\mathcal{T}_\omega(\nu,\mu) = \sup\left\{ \int Q_\omega \varphi(x)\,\mu(dx) - \int \varphi(y)\,\nu(dy)\right\},
\]
where the supremum runs over the class of bounded continuous functions $\varphi$. In case the function 
$Q_\omega \varphi$ is not measurable, then we understand $\int Q_\omega\varphi(x)\mu(dx)$ as the integral with respect to the inner measure $\mu_*$ induced by $\mu$. Recall that if $g: X \to \R$ is a function bounded from below, then
\[
\int g(x)\,\mu_*(dx) = \sup \int f(x)\,\mu(dx),
\]
where the supremum runs over the set of bounded \emph{measurable} functions $f$ such that $f \leq g.$

Under some semi-continuity and convexity assumptions on the cost function $c$, this duality formula generalizes to our optimal transport costs in the sense of Definition \ref{OTnew}. This duality property is described in the following 
definition.
\begin{defi}\label{defi:duality}
Let $\gamma : \R_+\to\R_+$ satisfy \eqref{gamma} and $c: X \times \mathcal{P}_\gamma(X) \to [0,+\infty]$ be a measurable cost function. One says that \emph{duality holds for the cost function $c$}, if for all $\mu,\nu \in \mathcal{P}_\gamma(X)$, it holds
\[
\mathcal{T}_c(\nu|\mu) = \sup_{\varphi \in \Phi_{\gamma,b}(X)}\left\{ \int R_c\varphi(x)\,\mu_*(dx) - \int \varphi(y)\,\nu(dy)\right\},
\]
where
\[
R_c\varphi(x):=\inf_{p\in \mathcal{P}_\gamma(X)}\left\{ \int \varphi(y)\,p(dy) + c(x,p)\right\},\quad  x\in X,\quad  \varphi\in \Phi_{\gamma,b}(X).
\]
\end{defi}
Section \ref{sec:duality} is devoted to the proof of a general result showing that duality holds under mild regularity conditions on $c$. Among these conditions, the main requirement is that $c$ is convex with respect to the $p$ variable.
We refer to Theorem \ref{Kantorovich general} for a precise statement. Since we do not know whether the conditions of Theorem \ref{Kantorovich general} are minimal, we prefer to postpone its statement to Section \ref{sec:duality} and to focus on particular families of cost functions (which are especially relevant for the applications we have in mind) for which the duality holds.

\subsection{Particular cases}\label{sec:particular cases}
As we already observed, if $\omega : X \times X \to [0,\infty]$ is a measurable function and $c(x,p)=\int \omega(x,y)\,p(dy)$, $x \in X,p \in \mathcal{P}(X)$, then the associated optimal transport cost corresponds the usual Monge-Kantorovich optimal transport cost $\mathcal{T}_\omega$ defined by
\[
\mathcal{T}_\omega(\nu,\mu) = \inf_{\pi \in \Pi(\mu,\nu)} \iint \omega(x,y)\,\pi(dxdy).
\] 
Among these costs a popular choice consists of taking, for $x,y \in X$, $\omega(x,y) = \alpha(d(x,y))$,  where $\alpha : \R^+ \to \R^+$ is a convex function.

The simple idea that leads from this classical family of cost functions to the family of cost functions described below, is to weaken $c$ by applying Jensen inequality:
\[
c(x,p)=\int \alpha\left(d(x,y)\right)\,p(dy) \geq \alpha\left(\int d(x,y)\,p(dy)\right):= \tilde{c}(x,p).
\]
Cost functions of the form $\tilde{c}$ as above appeared (in the particular case of the Hamming distance) in papers by Marton \cite{Mar96a, Mar96b}, Dembo \cite{Dem97}, Samson \cite{Sam00,Sam03,Sam07} in their studies of transport-type inequalities related to Talagrand's universal concentration inequalities for independent random variables. See Section~\ref{sectionhamming} for more information on the topic.

\subsubsection{Marton's cost functions}
Fix a function $\gamma :\R_+ \to \R_+$ satisfying \eqref{gamma} and a convex function $\alpha: \R_+ \to [0,+\infty]$.
The optimal transport cost associated to the cost function
\begin{equation}\label{eq:Marton-cost}
c(x,p)= \alpha\left(\int \gamma(d(x,y))\,p(dy)\right),\qquad  x\in X,\qquad  p \in \mathcal{P}_\gamma(X),
\end{equation}
will be denoted by $\Tw_\alpha$ and is defined by
\begin{equation}\label{eq:Marton-Ocost}
\Tw_\alpha (\nu |\mu)= \inf_{\pi \in \Pi(\mu,\nu)}\int \alpha\left(\int \gamma(d(x,y))\,p_x(dy)\right) \mu(dx)
\end{equation}
where $x\mapsto p_x$ is the probability kernel defined as usual by $\pi(dxdy)=\mu(dx)p_x(dy)$.
We will refer to this family of cost functions / optimal transport costs as \emph{Marton's costs} since they were first considered in \cite{Mar96a} for $\gamma=\gamma_0$ and $\alpha$ being the quadratic function, and therefore $c(x,p)= \left(\int\1_{x\neq y}\,p(dy)\right)^2 = p(X\setminus\{x\})^2$. 

Note that, in general, $\Tw_\alpha$ is not symmetric in $\mu,\nu$. Moreover, as we already observed above, if $\omega(x,y) = \alpha(\gamma(d(x,y)))$, then by Jensen's inequality, 
\[
\Tw_\alpha(\nu|\mu) \leq \mathcal{T}_\omega(\nu,\mu).
\] 
Finally, using probabilistic notations, one has
\[
\Tw_\alpha(\nu |\mu) = \inf_{(X,Y)} \E\Big[\, \alpha\big(\,\E\left[\gamma(d(X,Y)) \,|\, Y\right]\,\big)\,\Big],
\]
where the infimum runs over the set of all pairs of random variables $(X,Y)$ where $X$ has law $\mu$ and $Y$ has law $\nu.$
The following result gives sufficient conditions for duality for Marton's costs.
\begin{thm}\label{Kantorovich T tilda} 
Assume either that  
\begin{itemize}
\item $(X,d)$ is a complete separable metric space,
$\alpha:\R_+\to\R_+$ is a convex continuous function with $\alpha(0)=0$ and  $\gamma:\R_+\to\R_+$ is continuous, 
\item or $(X,d)$ is either a compact space or a countable set of isolated points, $\alpha:\R_+\to[0,+\infty]$ is a convex lower-semicontinuous function with $\alpha(0)=0$ and $\gamma:\R_+\to\R_+$ is lower-semicont\-inuous. 
\end{itemize}
Then, duality holds for the cost function $c$ defined in \eqref{eq:Marton-cost}. More precisely, 
\begin{equation} \label{kantotilde}
\Tw_\alpha(\nu|\mu)=\sup_{\varphi\in \Phi_{\gamma,b}(X)} \left\{\int \widetilde{Q}_\alpha \varphi(x)\,\mu(dx) - \int \varphi(y)\,\nu(dy)\right\},\;  \mu,\nu \in \mathcal{P}_\gamma(X),
\end{equation}
where
\[
\widetilde{Q}_\alpha \varphi(x)=\inf_{p\in \mathcal{P}_\gamma(X)} \left\{ \int \varphi(y)\,p(dy) + \alpha\left(\int \gamma(d(x,y))\,p(dy)\right)\right\}, 
\]
for $x\in X$, $\varphi\in \Phi_{\gamma,b}(X)$.
\end{thm}
We observe that, anticipating  the present paper, the duality formula \eqref{kantotilde} 
was already put to use in \cite{GRST14}, in connection with displacement convexity of the relative entropy functional on graphs.

\subsubsection{A barycentric variant of Marton's cost functions}
When $X\subset \R^m$ (equipped with an arbitrary norm $\|\,\cdot\,\|$) is a closed set, a variant of Marton's costs functions is obtained by choosing 
\begin{equation}\label{eq:Bar-cost}
c(x,p)=\theta\left(x- \int y\,p(dy)\right), \qquad x \in X,\qquad  p \in\mathcal{P}_1(X),
\end{equation}
where $\theta:\R^m\to [0,\infty]$ is a lower-semicontinuous convex function.
The corresponding transport cost is denoted by $\Tb_{\theta}$ and defined by
\begin{equation}\label{eq:Bar-Ocost}
\Tb_{\theta}(\nu|\mu) = \inf_{\pi \in \Pi(\mu,\nu)} \int \theta\left( x - \int y\,p_x(dy)\right)\,\mu(dx).
\end{equation}
We use the notation $\Tb_\theta$ with a \emph{bar} in reference to the \emph{barycenter} entering its definition.

Using probabilistic notations, we have the following alternative definition
\[
\Tb_\theta(\nu|\mu) = \inf_{(X,Y)} \E\big[\,\theta\left(X - \E[Y|X]\right)\,\big],
\]
where the infimum runs over the set of all pairs of random variables $(X,Y)$, with $X$ having law $\mu$ and $Y$ having law $\nu.$ Moreover, if $\omega(x,y) = \alpha(\|x-y\|)$, $x,y\in \R^m$, where $\alpha:\R_+ \to \R_+$ is convex, and $\theta(u) = \alpha(\|u\|)$, $u \in \R^m$, then the following holds:
\[
\mathcal{T}_\omega(\nu,\mu) \geq \Tw_\alpha(\nu|\mu) \geq \Tb_\theta(\nu|\mu).
\]

As we shall see below, this family of transport costs has strong connections with convex functions, and convex ordering of probability measures. In particular, the transport cost corresponding to $\theta(x)=|x|$, $x\in \R$, will be involved in a new proof of a result by Strassen on the existence of a martingale with given marginals (see Section \ref{Sec:Strassen}).

Duality for this family of costs functions is established in the following result. Note that for the ``bar'' transport cost, the duality formula for $\Tb_\theta$ can be expressed using only convex functions. This fact will repeatedly be used in the applications.
\begin{thm}\label{Kantorovich T bar}
Let $X\subset \R^m$ be a closed subset of $\R^m$ equipped with a norm $\|\cdot \|$ and  $\theta \colon \R^m \to \R_+$ be a convex function such that $\theta(x)\geq a\|x\|+b$, for all $x\in \R^m$ and for some $a>0$ and $b \in  \R$. Then duality holds for the cost function defined in \eqref{eq:Bar-cost}. More precisely:
\begin{enumerate}
\item The following duality identity holds
\[
\Tb_{\theta}(\nu|\mu)=\sup_{\varphi\in \Phi_{1,b}(X)}\left\{\int \overline{Q}_{\theta} \varphi(x)\,\mu(dx) - \int \varphi(y)\,\nu(dy)\right\},\;\;  \mu,\nu \in \mathcal{P}_1(X),
\]
where for all $x\in \R^m$ and all $\varphi\in \Phi_{1,b}(X)$,
\[
\overline{Q}_{\theta} \varphi(x)=\inf_{p\in \mathcal{P}_1(X)}\left\{\int \varphi(y) \,p(dy) + \theta\left(x-\int y\,p(dy)\right) \right\}.
\]
Since $\mathcal{P}_1(X)\subset \mathcal{P}_1(\R^m)$, the same conclusion holds replacing  
$\Phi_{1,b}(X)$ by $\Phi_{1,b}(\R^m)$ in the dual expression of $\Tb_{\theta}(\nu|\mu)$, and $\mathcal{P}_1(X)$ by $\mathcal{P}_1(\R^m)$ in the definition of $ \overline{Q}_{\theta} \varphi$.
\item For all $\varphi \in \Phi_{1,b}(\R^m)$ and all $x\in \R^m$,  it holds
\[
\overline{Q}_{\theta} \varphi(x):=\inf_{p\in \mathcal{P}_1(\R^m)}\left\{\int \varphi(y) \,p(dy) + \theta\left(x-\int y\,p(dy)\right) \right\}= Q_\theta \overline{\varphi}(x),
\]
where $\overline{\varphi}$ denotes the greatest  convex function $h:\R^m \to \R$ such that $h\leq \varphi$, and we recall that $Q_\theta g(x)=\inf_{y\in \R^m} \{g(y)+\theta(x-y)\}$,  $g\in \Phi_{1,b}(\R^m)$, $x\in \R^m$.
\item For all $\mu,\nu \in \mathcal{P}_1(X)$, it holds
\begin{align*}
\Tb_\theta(\nu|\mu)= \sup\Big\{ \int Q_{\theta} \varphi\,d\mu - \int  \varphi \,  d\nu &; \varphi:\R^m  \to \R, \text{ convex}, \text{Lipschitz}, \text{bounded from below} \Big\}.
\end{align*}
\end{enumerate}
The  results $\it{(1),(2),(3)}$ also hold when $\theta: \R^m \to [0,+\infty]$ is a lower semi-continuous convex function and $X$ is either  compact  or  a countable set of isolated points. 
\end{thm}

\subsubsection{Samson's cost functions}
Let $\beta:\R_+\to [0,+\infty]$ be a lower-semicontinuous convex function and  $\mu_0$ be a reference probability measure on $X$. The choice 
\begin{equation}\label{eq:Hat-cost}
c(x,p)=\int \beta\left(\gamma(d(x,y)) \frac{dp}{d\mu_0}(y)\right)\,\mu_0(dy), \qquad x \in X, 
\end{equation}
if $p\in\mathcal{P}$ is absolutely continuous with respect to $\mu_0$ on $X\setminus\{x\}$, and $c(x,p)=+\infty$ otherwise,
yields the family of weak transport $\Th_\beta$ defined by 
\begin{equation}\label{eq:Hat-Ocost}
\Th_{\beta}(\nu|\mu) = \inf_{\pi \in \Pi(\mu,\nu)} \iint \beta\left( \gamma(d(x,y)) \frac{dp_x}{d\mu_0}(dy)\right)\,\mu_0(dy)\,\mu(dx),
\end{equation}
for all measures $\mu,\nu\in\mathcal{P}_1(X)$, absolutely continuous with respect to $\mu_0$. Cost functions of this type were introduced by the third named author in \cite{Sam07}.

Again, if $\beta=\alpha$ is convex, then Jensen inequality gives
\[
\Tw_\beta (\nu|\mu) \leq \Th_\beta(\nu|\mu),
\]
but there is no clear comparison between $\Th_\beta(\nu|\mu)$ and $\mathcal{T}_\omega(\nu|\mu)$ with $\omega(x,y) = \alpha(d(x,y))$, $x,y \in X.$

Finally we state a duality theorem for the ``hat" transport cost.

\begin{thm}\label{Kantorovich T hat} Let $(X,d)$ be a compact metric space or a countable set of isolated points.  Let  $\beta:\R_+\to [0,+\infty]$ be a  lower-semicontinuous  convex function with $ \beta(0)=0$ and $\lim_{x\to \infty}\beta(x)/x=+\infty$. Assume that $\gamma:\R_+\to \R_+$ is lower-semicontinuous with $\gamma(0)=0$ and $\gamma(u)>0$ for all $u>0$. Let $\mu_0$ be a reference probability measure on $X$.
Then duality holds for the cost function defined in \eqref{eq:Hat-cost}. More precisely, for all $\mu,\nu \in \mathcal{P}_\gamma(X)$ absolutely continuous with respect to $\mu_0$, it holds
\[
\Th_\beta(\nu|\mu)=\sup_{\varphi\in \Phi_{\gamma,b}(X)} \left\{\int \widehat{Q}_\beta \varphi(x)\,\mu(dx) - \int \varphi(y)\,\nu(dy)\right\},
\]
where for $x\in X$ and $\varphi\in \Phi_{\gamma,b}(X)$,
\begin{align*}
\widehat{Q}_\beta \varphi(x):=\inf_{p\in \mathcal{P}_\gamma(X),\ p\ll\mu_0 \,on\, X\setminus\{x\} }
\Big\{ & \int \varphi(y)\,p(dy) + \int \beta\left( \gamma(d(x,y))\,\frac{dp}{d\mu_0}(y)\right)d\mu_0(y)\Big\}.
\end{align*}
\end{thm}

\subsubsection{Notation}\label{Notation T_p}
We end this section by introducing notations for the optimal transport costs related to power functions.
When $\alpha(x)=x^p$, $x\geq0$, $p>0$, we will use the notation $\mathcal{T}_p$ and $\Tw_p$ to denote the costs above. Accordingly, if $X=\R^m$ is equipped with a norm $\|\,\cdot\,\|$ and $\theta(x)=\|x\|^p$, we will denote the third transport cost by $\Tb_{p}$.

\medskip



\section{Proof of a result by Strassen}\label{Sec:Strassen}
In this short section, we show that the transport cost $\Tb_\theta$ can be used to recover an old result by Strassen \cite{Str65} about the existence of a martingale with given marginals.

In the sequel, we equip $\R^m$ with an arbitrary norm $\|\cdot\|$.
Let $\mu,\nu \in \mathcal{P}_1(\R^m)$;  one says that $\mu$ is dominated by $\nu$ in the convex order sense, and one writes $\mu\preceq_C \nu$, if 
\[
\int f\,d\mu \leq \int f\,d\nu,
\]
for all convex\footnote{Note that since $\mu,\nu \in \mathcal{P}_1(\R^m)$, any affine map is integrable with respect to $\mu$ and $\nu$. Since a convex function is always positive up to the addition of some affine map, we see that the integral of convex functions with respect to $\mu$ and $\nu$ makes sense.}
 $f:\R^m \to \R$. Note that, in particular, this implies that $\int f\,d\mu = \int f\,d\nu$ for all affine maps $f:\R^m \to \R$. 
 
 It is not difficult to check that $\mu\preceq_C \nu$ if and only if $\int f\,d\mu \leq \int f\,d\nu$ for all $1$-Lipschitz and convex $f:\R^m \to \R$ bounded from below\footnote{One possible way to prove this is to use the fact that if $f \colon \R^m\to \R$ is convex, then the classical inf-convolution operator 
 $Q_tf(x):=\inf_{y\in \R^m}\{f(y) + \frac{1}{t} \|x-y\|\}$ is convex, $1/t$-Lipschitz and $Q_t f(x) \uparrow f(x)$ when $t \to 0$ for all $x\in \R^m$.}.
 
 The following result goes back at least to the work of Strassen \cite{Str65}.
 \begin{thm}\label{thm:Strassen}
 Let $\mu,\nu \in \mathcal{P}(\R^m)$; there exists a martingale $(X,Y)$, where $X$ follows the law $\mu$ and $Y$ the law $\nu$ if and only if $\mu \preceq_C\nu.$ 
 \end{thm}
 
Below we obtain Strassen's theorem as a consequence of the duality formula for the cost $\Tb_1$ given in the following proposition.
\begin{prop}
For all $\mu,\nu \in \mathcal{P}_1(\R^m)$,
\begin{align*}
\Tb_1(\nu|\mu) = \sup\Big\{ \int \varphi\,d\mu - \int \varphi\,d\nu;  \varphi  \text{ convex}, & \text{ 1-Lipschitz},\text{bounded from below}\Big\} .
\end{align*}
\end{prop}

\proof
We already know from Point (3) of Theorem \ref{Kantorovich T hat} that for all $\mu,\nu \in \mathcal{P}_1(\R^m)$ it holds
\begin{align*}
\Tb_1(\nu|\mu) = \sup\Big\{ \int Q_1 \varphi\,d\mu - \int \varphi\,d\nu;  \varphi  \text{ convex}, & \text{ Lipschitz},\text{bounded from below}\Big\} .
\end{align*}
with $Q_1\varphi(x) = \inf_{y\in \R^m}\{\varphi(y) + \|x-y\|\}$, $x\in \R^m.$ 
It is easy to check that if $\varphi:\R^m\to\R$ is convex and bounded from below, so is 
$Q_1\varphi \colon \R^m \to \R$. 
Being an infimum of $1$-Lipschitz functions, $Q_1\varphi$ is itself $1$-Lipschitz. Moreover, if $\psi:\R^m \to\R$ is some $1$-Lipschitz convex function, then $Q_1\psi=\psi$; namely, for all $x\in\R^m$, one has 
\[
0\geq Q_1\psi(x)-\psi(x)\geq \inf_{y\in \R^m}\{\psi(y)-\psi(x) + \|x-y\|\}\geq 0.
\]
From these considerations, we conclude that
\begin{align*}
\Tb_1(\nu|\mu) &= \sup \left\{\int Q_1 \varphi \,d\mu - \int \varphi\,d\nu ; \varphi \text{ convex Lipschitz bounded below}\right\}\\
&\leq \sup \left\{\int \psi \,d\mu - \int \psi\,d\nu ; \psi \text{ convex, 1-Lipschitz, bounded below}\right\}\\
& = \sup \left\{\int Q_1\psi \,d\mu - \int \psi\,d\nu ; \psi \text{ convex, 1-Lipschitz, bounded below}\right\}\\
& \leq  \sup \left\{\int Q_1\varphi \,d\mu - \int \varphi\,d\nu ; \varphi \text{ convex, Lipschitz, bounded below}\right\}.
\end{align*}
This concludes the proof.
\endproof

\proof[Proof of Theorem \ref{thm:Strassen}]
If $\pi \in \Pi(\mu,\nu)$ denotes the law of $(X,Y)$, the condition that $(X,Y)$ is a martingale is expressed by
 \begin{equation}\label{martingale}
 \int y\,p_x(dy) = x,\qquad \text{ for } \mu \text{ almost every } x \in \R^m.
 \end{equation}
Recall that $\Tb_1(\nu|\mu) = \inf_{\pi \in \Pi(\mu,\nu)} \int \|x - \int y\,p_x(dy)\| \mu(dx).$ Therefore, there exists some $\pi \in \Pi(\mu,\nu)$ satisfying \eqref{martingale} if and only if $\Tb_1(\nu|\mu)=0$.
Since, by Corollary \ref{Kantorovich T bar},
\begin{align*}
\Tb_1(\nu|\mu) = \sup \Big\{ \int f\,d\mu - \int f\,d\nu ;  f \colon  \R^m & \to\R, 1- \text{Lipschitz}, \text{convex and bounded below} \Big \},
\end{align*}
the expected result follows.
\endproof

\begin{rem}
Let us note that we obtained in fact the following slightly more general result: Let $\varepsilon>0$ ; two probability measures $\mu,\nu \in \mathcal{P}_1(\R^m)$ satisfy $\int f\,d\mu \leq \int f\,d\nu + \varepsilon$, for all $1$-Lipschitz convex functions $f:\R^m\to \R$, if and only if there exists a pair $(X,Y)$ of random variables, with $X$ of law $\mu$ and $Y$ of law $\nu$, such that
\[
\E[ \|X - \E[Y|X]\|] \leq \varepsilon.
\]
\end{rem}


\section{Transport-entropy inequalities: definitions, tensorization, and dual formulation}

In this section, we introduce a general notion of transport-entropy inequalities of Talagrand-type and investigate them.

\subsection{Definitions}
We recall that if $\mu,\nu$ are two probability measures on some space $X$, the relative entropy of $\nu$ with respect to $\mu$ is defined by
\[
H(\nu|\mu) = \int \log \left(\frac{d\nu}{d\mu}\right)\,d\nu\in \R_+\cup\{+\infty\}\,,
\]
if $\nu\ll \mu$. Otherwise, ones sets $H(\nu|\mu)=+\infty.$
\begin{defi}[Transport-entropy inequalities $\mathbf{T}_c(a_1,a_2)$ and $\mathbf{T}_c(b)$] \label{def:te}\ \\
Let  $c \colon X\times \mathcal{P}_\gamma(X)\to [0,\infty]$ be a measurable cost function associated to some lower-semicontinuous function $\gamma \colon \R_+\to \R_+$ satisfying \eqref{gamma}, and $\mu \in \mathcal{P}_\gamma(X).$
\begin{itemize}
\item The probability measure $\mu$ is said to satisfy $\mathbf{T}_c(a_1,a_2)$, for some $a_1,a_2>0$ if 
\begin{equation} \label{eq:te}
\mathcal{T}_c(\nu_1 | \nu_2) \leq a_1 H(\nu_1 | \mu) + a_2 H(\nu_2 | \mu),\qquad \forall \nu_1, \nu_2 \in \mathcal{P}_\gamma(X).
\end{equation}
\item The probability measure $\mu$ is said to satisfy $\mathbf{T}^+_c(b)$ for some $b>0$,
if 
\begin{equation} \label{eq:te+}
\mathcal{T}_c(\nu | \mu) \leq b H(\nu | \mu),\qquad \forall \nu \in \mathcal{P}_\gamma(X).
\end{equation}
\item The probability measure $\mu$ is said to satisfy $\mathbf{T}^-_c(b)$ for some $b>0$, if
\begin{equation} \label{eq:te-}
\mathcal{T}_c(\mu | \nu) \leq b H(\nu | \mu),\qquad \forall \nu \in \mathcal{P}_\gamma(X).
\end{equation}
\end{itemize}
For the specific transport costs $\Tw_p$ and $\Tb_p$ introduced in Section \ref{Notation T_p} we may use the corresponding notations $\mathbf{\widetilde{T}}_p(a_1,a_2)$, $\mathbf{\widetilde{T}}^\pm_p(b)$,
respectively $\mathbf{\overline{T}}_p(a_1,a_2)$, $\mathbf{\overline{T}}^\pm_p(b)$.
\end{defi}

Let us comment on this definition. 
First we note, that when $c(x,p) = \int \omega (x,y)\,p(dy)$, \eqref{eq:te+} and $\eqref{eq:te-}$ give back the usual transport-entropy inequalities of Talagrand type (see \cite{Led01}, \cite{Vil09} or \cite{GL10} for a general introduction on the subject).
Also, we observe that $\mathbf{T}_c(a_1,0)$ or $\mathbf{T}_c(a_2,0)$ (which are not considered in the above definition, since $a_1,a_2>0$) has no meaning. Indeed, if $\mathbf{T}_c(a_1,0)$ holds, then $\mathcal{T}_c(\nu_1 | \nu_2) \leq a_1 H(\nu_1 | \mu)$ for all $\nu_1, \nu_2$ which in turn implies $\mathcal{T}_c(\mu | \nu_2) = 0$ for all $\nu_2$ which is impossible.
Finally, using the convention that $0 \cdot \infty = 0$, we observe that $\mathbf{T}^+_c(b)$ is formally equivalent to $\mathbf{T}_c(b,\infty)$, and $\mathbf{T}^-_c(b)$ is equivalent to $\mathbf{T}^-_c(\infty,b)$.

As for the classical inequality, $\mathbf{T}_c(a_1,a_2)$ does enjoy the tensorization property.
Moreover, if duality holds for the cost function $c$ (in the sense of Definition \ref{defi:duality}), we can state a dual characterization of $\mathbf{T}_c(a_1,a_2)$ in the spirit of Bobkov-G\"otze dual formulation \cite{BG99}.

We now state these properties and characterizations.

\subsection{Bobkov-G\"otze dual characterization}

The following characterization extends, thanks to the dual formulation of the transport cost \cite{BG99}; see also \cite{GL10}.

\begin{prop}[Dual formulation] \label{prop:bg}
Let  $c \colon X\times \mathcal{P}_\gamma(X)\to [0,\infty]$  be a measurable cost function associated to some lower-semicontinuous function $\gamma \colon \R_+\to \R_+$ satisfying \eqref{gamma}. Assume that 
$c(x,\delta_x)=0$ for all $x\in X$ and that duality holds for the cost function $c$.
For $\mu \in \mathcal{P}_\gamma(X)$ and $a_1, a_2, b>0$, Items $(i)$'s and $(ii)$'s are equivalent:
\begin{itemize}
\item 
\begin{itemize}
\item[$(i)$] $\T_c(a_1,a_2)$ holds;
\item[$(ii)$] for all $ \varphi \in \Phi_{\gamma,b}(X)$ (resp. for all non-negative $\varphi\in \Phi_\gamma$), it holds
\begin{equation} \label{eq:bg}
\left( \int \exp\left\{\frac{R_c \varphi}{a_2}\right\}d\mu \right)^{a_2}
\left( \int \exp\left\{-\frac{\varphi}{a_1}\right\} d\mu \right)^{a_1} 
 \leq 1;
\end{equation}
\end{itemize}
\item 
\begin{itemize}
\item[$(i')$] $\T^+_c(b)$ holds;
\item[$(ii')$] for all $ \varphi \in \Phi_{\gamma,b}(X)$ (resp. for all non-negative $\varphi\in \Phi_\gamma$), it holds
\begin{equation} \label{eq:bg+}
\exp \left\{ \int R_c \varphi\,d\mu \right\}  
\left( \int \exp \left\{\frac{ -\varphi}{b} \right\} d\mu \right)^{b}
\leq 1 ;
\end{equation}
\end{itemize}
\item
\begin{itemize}
\item[$(i'')$] $\T^-_c(b)$ holds;
\item[$(ii'')$] for all $ \varphi \in \Phi_{\gamma,b}(X)$ (resp. for all non-negative $\varphi\in \Phi_\gamma$), it holds
\begin{equation} \label{eq:bg-}
\left(\int \exp\left\{\frac{R_c \varphi}{b}\right\} d\mu \right)^{b}
\exp\left\{-\int \varphi \,d\mu  \right\}
\leq 1\,,
\end{equation}
\end{itemize}
\end{itemize}
where we recall that $R_c \varphi(x)=\inf_{p \in \mathcal{P}_\gamma(X)}\left\{ \int \varphi(y)\,p(dy) + c(x,p) \right\}$, $x \in X$.

Moreover, specializing to the ``bar" cost $\Tb_\theta$, one can replace, in $(ii)$, $(ii')$ and $(ii'')$,
$R_c \varphi$ by $Q_\theta \varphi:= \inf_{y \in \mathbb{R}^m} \left\{ \varphi(y) + \theta(\,\cdot\,-y)\right\}$
and restrict to the set of functions $\varphi$ that are convex, Lipschitz and bounded from below.

\end{prop}

\begin{rem}\ 
\begin{itemize}
\item The preceding result thus applies to the cost functions defined in Section \ref{sec:particular cases} under the assumptions of Theorems \ref{Kantorovich T tilda}, \ref{Kantorovich T bar} and \ref{Kantorovich T hat} and more generally to all the cost functions satisfying the assumptions of our general duality Theorem \ref{Kantorovich general}.
\item In the result above, we implicitly assumed that functions $R_c\varphi$ were measurable. If it is not the case, then integrals of $R_c\varphi$ with respect to $\mu$ have to be replaced by integrals with respect to the inner measure $\mu_*$.
\item When $c(x,p)=\theta\left(x - \int y\,p(dy)\right)$, $x\in \R^m$, $p\in \mathcal{P}_1(\R^m)$, for some convex function $\theta:\R^m\to\R_+$, the inequality $\T_c(a_1,a_2)$ is thus equivalent to the following exponential type inequality first introduced by Maurey \cite{Mau91} (the so-called convex  $(\tau)$-property):
\[
\left(\int e^{\frac{Q_\theta\varphi}{a_2}}\, d\mu\right)^{a_2} \left(\int e^{-\frac{\varphi}{a_1}}\,d\mu\right)^{a_1}\leq 1,\qquad \forall \varphi:\R^m\to\R_+ \text{ convex.}
\]
\end{itemize}
\end{rem}

\begin{proof}
By duality (\textit{i.e.}\ using Definition \ref{defi:duality}), $\T_c(a_1,a_2)$ is equivalent to have
\[
a_2\left( \int \frac{R_c \varphi}{a_2}\,d\nu_2 - H(\nu_2|\mu) \right) 
+
a_1\left( \int - \frac{\varphi}{a_1}\,d\nu_1 - H(\nu_1|\mu) \right) 
\leq 0\,,
\]
for all $\varphi \in \Phi_{\gamma,b}(X)$ and all $\nu_1, \nu_2 \in \mathcal{P}_\gamma(X)$ with finite relative entropy with respect to $\mu.$
The expected result follows by taking the (two independent) suprema, on the left hand side, over $\nu_1$ and $\nu_2$, and by using Lemma \ref{lem-dual} below. Note that since $c(x,\delta_x)=0$ for all $x\in X$, one always has $R_c\varphi \leq \varphi$, for all $\varphi \in \Phi_{\gamma,b}(X)$ and so the function $\psi = R_c\varphi /a_2$ satisfies the assumption of Lemma \ref{lem-dual}. This completes the proof of the equivalence $(i) \Leftrightarrow (ii)$. 

Note that \eqref{eq:bg} is invariant under translations $\varphi \mapsto \varphi+a$ and so the functions $\varphi$ can be assumed non-negative.

The two last equivalences follow the same line (and the details are left to the reader). Similarly, the specialization to the ``bar" cost is identical, one just needs to apply Item $(3)$ of Theorem \ref{Kantorovich T bar}.
\end{proof}

\begin{lem} \label{lem-dual}
Let $\mu \in \mathcal{P}_\gamma(X)$ for some lower-semicontinuous function $\gamma:\R_+\to\R_+$ satisfying \eqref{gamma}; for all measurable function $\psi:\X \to \R$ such that $\psi \leq \varphi$ for some $\varphi \in \Phi_\gamma(X)$, it holds
\[
\sup_{\nu \in \mathcal{P}_\gamma(X)}\left\{ \int \psi\, d\nu - H(\nu|\mu) \right\}
=
\log \int e^{\psi}\,d\mu .
\]
\end{lem}

\proof[Proof of Lemma \ref{lem-dual}]
Consider the function $U(x)=x\log(x)$, $x>0$. A simple calculation shows that $U^*(t):=\sup_{x>0}\{tx-U(x)\}=e^{t-1}$, $t\in \R$. Since $\psi \leq \varphi$, for some $\varphi \in \Phi_\gamma(X)$, one concludes that $\int [\psi]_+\,d\nu$ is finite for all $\nu \in \mathcal{P}_\gamma(X)$, and thus $\int \psi\,d\nu$ is well-defined in $\R\cup\{-\infty\}$. Let $\nu\ll \mu$ ; applying Young's inequality $xy \leq U(x) + U^*(y)$, $x>0,y\in \R$, one gets
\[
\int \psi\,d\nu \leq\int U^*(\psi)\,d\mu  + \int U\left(\frac{d\nu}{d\mu }\right)\,d\mu  = \int e^{\psi-1}\,d\mu  + H(\nu|\mu ).
\]
Applying this inequality to $\psi+u$, where $u\in \R$, we get
\[
\int \psi\,d\nu - H(\nu| \mu)\leq e^{u-1}\int e^\psi\,d\mu -u,
\]
and this inequality is still true, even if $\nu$ is not absolutely continuous with respect to $\mu.$
Optimizing over $u\in \R$ and over $\nu \in \mathcal{P}_\gamma(X)$ yields:
\[
\sup_{\nu \in \mathcal{P}_\gamma(X)}\left\{\int \psi\,d\nu - H(\nu|\mu )\right\}\leq \log \int e^\psi\,d\mu .
\]
To get the converse inequality, consider $A_k=\{ x \in X ;  \psi(x)\leq k\}$, for $k\geq 0$ large enough, $\nu_k(dx)=\frac{e^{\psi(x)}}{\int e^{\psi}\1_{A_k}\,d\mu}\1_{A_k}(x)\,\mu(dx)$\,. Since $\mu$ belongs to $\mathcal{P}_\gamma(X)$ and $\nu_k$ has a bounded density with respect to $\mu$, $\nu_k$ also belongs to $\mathcal{P}_\gamma(X).$
Furthermore
\[
\int \psi\,d\nu_k-H(\nu_k|\mu)=\log\left(\int e^\psi \1_{A_k}\,d\mu\right)\to \log\left( \int e^{\psi}\,d\mu\right),
\]
when $k\to\infty.$ This completes the proof.
\endproof

\subsection{Tensorization}

In this section, we collect two important properties which will allow us to deal with one-dimensional measures in applications.

\begin{thm}[Tensoring property]\label{tensorization}
Let  $\gamma \colon \R_+\to \R_+$ be a lower-semicontinuous function satisfying \eqref{gamma},
 $(X_1,d_1),\ldots,(X_n,d_n)$ be complete separable metric spaces equipped with measurable cost functions $c_i:X_i \times \mathcal{P}_\gamma(X_i) \to [0,\infty]$, $i\in\{1,\ldots,n\}$ such that $c_i(x_i,\delta_{x_i})=0$  and $p_i \mapsto c_i(x_i,p_i)$ is convex for all $x_i\in X_i$.
For all $i\in\{1,\ldots,n\}$, let $\mu_i \in \mathcal{P}_\gamma(X_i)$ satisfying the transport inequality $\T_{c_i}(a_1^{(i)},a_2^{(i)})$ for some $a_1^{(i)},a_2^{(i)}>0.$ Then the product probability measure $\mu_1\otimes \cdots\otimes \mu_n$ satisfies the transport inequality $\T_{c}(a_1,a_2)$,
with $a_1:=\max_i a_1^{(i)}$, $a_2:= \max_i a_2^{(i)}$, for the cost function $c \colon X_1 \times \cdots\times X_n \times \mathcal{P}_\gamma(X_1\times \cdots \times X_n)\to [0,\infty)$ defined by
\[
c(x,p) = c_1(x_1,p_1) +\cdots + c_n(x_n,p_n),
\]
for all $x=(x_1,\ldots,x_n) \in X_1 \times\cdots \times X_n,$ and for all $p \in \mathcal{P}_\gamma(X_1 \times\cdots \times X_n)$, where $p_i$ denotes the $i$-th marginal distribution of $p$.
\end{thm}

The following is an immediate corollary of Theorem \ref{tensorization}.
\begin{cor} \label{cor:tensorization}
Let  $\gamma \colon \R_+\to \R_+$ be a lower-semicontinuous function satisfying \eqref{gamma} and assume that $\mu \in \mathcal{P}_\gamma(X)$ satisfies the transport inequality $\T_c(a_1,a_2)$ for some $a_1,a_2>0$ and some cost function $c:X \times \mathcal{P}_\gamma(X) \to [0,\infty]$ that satisfies $c(x,\delta_x) =0$ and $p \mapsto c(x,p)$ convex for all $x \in X$. Then for all positive integers $n$, the product probability measure $\mu^n \in \mathcal{P}_\gamma(X^n)$ satisfies the inequality $\T_{c^n}(a_1,a_2)$, where $c^n:X^n \times \mathcal{P}_r(X^n)\to [0,\infty)$ is the cost function defined by
\[
c^n(x,p):=\sum_{i=1}^n c(x_i,p_i),\qquad  x=(x_1,\dots,x_n) \in X^n, \qquad  p \in \mathcal{P}_\gamma(X^n),
\] 
where  $p_i$ denotes the $i$-th marginal distribution of $p$.
\end{cor}

The proof of Theorem \ref{tensorization} is postponed to Appendix \ref{appendix-tensorization}.


\section{Transport-entropy inequalities : link with dimension-free concentration}\label{sec:concentration}
In this section, extending \cite{Goz09}, we characterize the transport-entropy inequality $\T_c(a_1,a_2)$
in terms of a dimension-free concentration property. We recall first (and introduce) some notation.

Let  $\gamma \colon \R_+\to \R_+$ be a lower-semicontinuous function satisfying \eqref{gamma} and
$c : X \times \mathcal{P}_\gamma(X) \to [0,\infty)$ such that $c(x,\delta_x)=0$ for all $x\in X$.
Recall from Corollary \ref{cor:tensorization} that
for all integers $n \geq 1$,
\[
c^n(x,p):=\sum_{i=1}^n c(x_i,p_i),\qquad  x=(x_1,\dots,x_n) \in X^n,  \qquad p \in \mathcal{P}_r(X^n),
\] 
where  $p_i$ denotes the $i$-th marginal distribution of $p$. 
For all $\varphi \in \Phi_\gamma(X^n)$, define as before
\[
R_{c^n}\varphi (x) = \inf_{p\in \mathcal{P}_\gamma(X^n)} \left\{ \int \varphi\,dp + c^n(x,p)\right\},
\qquad  x\in \X^n.
\] 
Finally for all Borel sets $A \subset X^n$, let
\[
c^n_A(x):=\inf_{p \in \mathcal{P}_\gamma(X^n): p(A)=1} c^n(x,p), \qquad x \in X^n,
\]
and, for $t \geq 0$, 
\[
A^n_t : = \left\{x \in X^n : c^n_A(x) \leq t \right\} .
\]

\subsection{A general equivalence}
We are now in a position to state our theorem.

\begin{thm} \label{gozlan}
Let  $\gamma \colon \R_+\to \R_+$ be a lower-semicontinuous function satisfying \eqref{gamma} and $c : X \times \mathcal{P}_\gamma(X) \to [0,\infty)$ a measurable cost function such that $c(x,\delta_x)=0$ for all $x\in X$,
and for which duality holds in the sense of Definition \ref{defi:duality}.
For $\mu \in \mathcal{P}_\gamma(X)$ and $a_1, a_2 >0$, the following are equivalent:
\begin{itemize}
\item[$(i)$] $\mu$ satisfies $\T_c(a_1,a_2)$;
\item[$(ii)$] there exists a numerical constant $K$ such that  for all integers $n \geq 1$, for all Borel sets $A \subset X^n$, it holds
\begin{equation} \label{eq:concentration}
\mu^n(X^n\setminus A^n_t)^{a_2} \mu^n(A)^{a_1} \leq Ke^{-t} \qquad \forall t \geq 0 .
\end{equation}
\item[$(iii)$] there exists a numerical constant $K$ such that for all integers $n \geq 1$, for all non-negative $\varphi \in \Phi_\gamma(X^n)$, it holds
\[
\mu^n(R_{c^n} \varphi > u)^{a_2} \mu^n(\varphi \leq v)^{a_1} \leq K e^{-u+v}
 \qquad \forall u, v \in \R .
\]
\end{itemize}
\end{thm}

\begin{rem}\ 
\begin{itemize}
\item The implication $(i) \Rightarrow (ii)$ was first discovered by Marton \cite{Mar86, Mar96a, Mar96b}. This nice observation is at the origin of the interest in transport-entropy inequalities.
\item The implications $(i) \Rightarrow (ii)$ and $(ii)\Rightarrow (iii)$ are in fact true solely under the assumptions $c(x,\delta_x)=0$ for all $x\in X$ and $p\mapsto c(x,p)$ is convex, as the proof indicates.
\end{itemize}
\end{rem}

\begin{proof}
First we prove that $(i)$ implies $(ii)$. Since $\mu$ satisfies $\T_c(a_1,a_2)$, by the tensorization property, for all positive integers $n$, it holds
\[
\mathcal{T}_{c^n}(\nu_1 | \nu_2) \leq a_1H(\nu_1 | \mu^n) + a_2 H(\nu_2 | \mu^n),
\]
for all $\nu_1,\nu_2 \in \mathcal{P}_\gamma(X^n).$
Let $A\subset X^n$ be a Borel set and define $\nu_1(dx) = \frac{\1_A(x)}{\mu^n(A)}\,\mu^n(dx)$ and $\nu_2(dx) =\frac{\1_B(x)}{\mu^n(B)}\,\mu^n(dx)$, where $B = X^n\setminus A^n_t$, for some $t>0.$
Then $H(\nu_1|\mu^n) = -\log \mu^n(A)$ and $H(\nu_2 | \mu^n) = -\log \mu^n(B)$. Furthermore, if $\pi \in \Pi(\nu_2,\nu_1)$ with disintegration kernel $(p_x)_{x\in X^n}$, then for $\nu_2$ almost all $x \in X^n$, $p_x(A)=1$. Therefore,
\[
\int c(x,p_x)\,\nu_2(dx) \geq \int c_A^n(x)\,\frac{\1_B(x)}{\mu^n(B)}\,\mu^n(dx) \geq t,
\]
where the last inequality comes from the fact that $c_A^n(x) > t$ for all $x\in B=\{x \in X^n : c^n_A(x)>t\}.$ Taking the infimum over all $\pi \in \Pi(\nu_2,\nu_1)$ finally yields
\[
t\leq \mathcal{T}_{c^n}(\nu_1 |\mu_2) \leq -a_1\log (\mu^n(A)) - a_2\log \mu^n (X^n \setminus A^n_t),
\]
which proves $(ii).$
%

Now we prove that $(ii)$ implies $(iii)$. Fix $n \geq 1$, $m \in \mathbb{R}$, $t \geq 0$ and a non-negative $\varphi \in \Phi_\gamma(X^n)$. We will prove that $\{R_{c^n} \varphi > m+t\} \subset \{c^n_A > t\}$ with $A:=\{\varphi \leq m\}$. To that aim consider $x \in \{R_{c^n} \varphi > m+t\}$. Then, for all $p \in \mathcal{P}_\gamma(X^n)$ with $p(A)=1$, we have
$\int \varphi \,dp \leq m$ so that, by definition of $R_{c^n}$, it holds
\[
m+t < \int \varphi\,dp + c^n(x,p) \leq m + c^n(x,p)  .
\]
Hence, taking the infimum over all $p$ with $p(A)=1$ leads to $c^n_A(x) > t$, which is the desired result. Point $(iii)$ then immediately follows applying Point $(ii)$ to $A$.

Finally we prove that $(iii)$ implies $(i)$, following \cite{GRS11}.
Fix $\varepsilon\in (0,1)$.
Given $f \in \Phi_\gamma(X)$, non-negative, let $\varphi(x) = f(x_1)+f(x_2) + \dots + f(x_n)$, $x\in X^n.$ Then, $\varphi \in \Phi_\gamma(X^n)$ is also non-negative and
$R_{c^n} \varphi (x) = \sum_{i=1}^n R_c f (x_i)$, so that, using the product structure of $\mu^n$,
\begin{align} \label{eurostar}
\left( \int e^{\frac{R_{c} f}{(1+\varepsilon) a_2}} d\mu \right)^{a_2}
\left( \int e^{-\frac{f}{(1-\varepsilon)a_1}} d\mu \right)^{a_1} =
\left( \int e^{\frac{R_{c^n} \varphi}{(1+\varepsilon) a_2}} d\mu^n \right)^{\frac{a_2}{n}}
\left( \int e^{-\frac{\varphi}{(1-\varepsilon)a_1}} d\mu^n \right)^{\frac{a_1}{n}} .
\end{align}
Our aim is to prove that the right hand side, to the power $n$, is bounded. 
Thanks to Point$ (iii)$, for any $v \in \mathbb{R}$ it holds
\begin{align*}
\int e^{\frac{R_{c^n} \varphi}{(1+\varepsilon) a_2}}\, d\mu^n
& =
1 +  \int_0^\infty e^u \mu^n\left(\frac{R_{c^n} \varphi}{(1+\varepsilon) a_2} > u\right)\,du \\
& \leq 
1 + \mu^n\left(\frac{\varphi}{(1-\varepsilon)a_1} \leq v\right)^{-\frac{a_1}{a_2}} K^{\frac1{a_2}}e^\frac{(1-\varepsilon)a_1 v}{a_2} \int_0^\infty e^{-\varepsilon u}\,du \\
& =
1 + \frac{1}{\varepsilon} \mu^n\left(\frac{\varphi}{(1-\varepsilon)a_1} \leq v\right)^{-\frac{a_1}{a_2}} K^{\frac1{a_2}} e^\frac{(1-\varepsilon)a_1 v}{a_2} .
\end{align*}
In particular, for all $v \in \mathbb{R}$,
\[
\left(-1+ \int e^{\frac{R_{c^n} \varphi}{(1+\varepsilon) a_2}} \,d\mu^n \right)^\frac{a_2}{a_1}
e^{-v} \mu^n\left(\frac{\varphi}{(1-\varepsilon)a_1} \leq v\right) \leq K^{\frac1{a_1}} \frac{e^{-\varepsilon v}}{\varepsilon^\frac{a_2}{a_1}} .
\]
Since $\int e^{-\frac{\varphi}{(1-\varepsilon)a_1}} \,d\mu^n =
\int_0^\infty e^{-v} \mu^n\left(\frac{\varphi}{(1-\varepsilon)a_1} \leq v\right)\, dv$, 
integrating the latter implies that
$$
\left(-1+ \int e^{\frac{R_{c^n} \varphi}{(1+\varepsilon) a_2}} \,d\mu^n \right)^\frac{a_2}{a_1}
\int e^{-\frac{\varphi}{(1-\varepsilon)a_1}}\,d\mu^n
\leq 
\frac{K^{\frac1{a_1}}}{\varepsilon^{1+\frac{a_2}{a_1}}} .
$$
This in turn implies, by simple algebra that 
\begin{align*}
& \left(\int e^{\frac{R_{c^n} \varphi}{(1+\varepsilon) a_2}} \,d\mu^n \right)^{a_2} 
\left( \int e^{-\frac{\varphi}{(1-\varepsilon)a_1}}\,d\mu^n\right)^{a_1} \\
& \qquad \qquad  \leq 
\left(1 + \frac{1}{\varepsilon} \left( \varepsilon \int e^{-\frac{\varphi}{(1-\varepsilon)a_1}}\,d\mu^n\right)^{-\frac{a_1}{a_2}} \right)^{a_2} 
\left( \int e^{-\frac{\varphi}{(1-\varepsilon)a_1}}\,d\mu^n\right)^{a_1} \\
& \qquad  \qquad =  
\left( 
\left( \int e^{-\frac{\varphi}{(1-\varepsilon)a_1}} \,d\mu^n   \right)^\frac{a_1}{a_2} + 
\frac{1}{\varepsilon^{\1+\frac{a_1 }{a_2}}}
\right)^{a_2}  \\
& \qquad \qquad  \leq
\left( 1 + 
\frac{1}{\varepsilon^{\1+\frac{a_1 }{a_2}}}
\right)^{a_2} \,,
\end{align*}
where in the last line we used that $\varphi$ is a non-negative function.

Plugging this bound into \eqref{eurostar} leads, in the limit $n \to \infty$, to
\[
\left( \int e^{\frac{R_{c} f}{(1+\varepsilon) a_2}}\,d\mu \right)^{a_2}
\left( \int e^{-\frac{f}{(1-\varepsilon)a_1}} \,d\mu \right)^{a_1} 
\leq 1 .
\]
Taking $\varepsilon$ to $0$ gives $\T_c(a_1,a_2)$, thanks to Proposition \ref{prop:bg}.
%
%
%
\end{proof}

\subsection{Particular cases} In this section we focus on concentration inequalities related to the usual Monge-Kantorovich transport-cost and to barycentric transport-costs.

\subsubsection{Usual costs}
Note that when $c(x,p) = \int \omega(x,y)\,p(dy)$, for some measurable $\omega : X \times X \to[0,\infty)$, the enlargement $A_t^n$ of some set $A\subset X$ reduces to 
\[
A_t^n = \{x \in X^n ; \exists y\in A \text{ s.t. } \sum_{i=1}^n\omega(x_i,y_i)\leq t \}.
\]
In particular, when $X=\R^m$ and $\omega(x,y)=\|x-y\|^r$, $r\geq2$, where $\|\,\cdot\,\|$ is a given norm on $\R^m$, then denoting by 
\begin{equation}\label{Bpn}
B_{r}^n = \left\{x \in \left(\R^m\right)^n ; \sum_{i=1}^n \|x_i\|^r \leq 1\right\}\,,
\end{equation}
 it holds
\[
A_t^n = A + t^{1/r}B_r^n.
\]

Concentration of measure inequalities are usually stated for enlargements of sets of measure bigger than $1/2$ as in \eqref{eq:Conc-Gauss2} (see \cite{Led01}).
In what follows we connect \eqref{eq:concentration} to the usual definition for some families of cost functionals. 
\begin{lem}\label{equiv conc 1/2}
Consider a cost function $c$ of the form 
\[
c(x,p) =  \int \gamma (d(x,y))\,p(dy),\qquad x\in X,\quad p\in \mathcal{P}_\gamma(X)
\] 
with $\gamma : \R_+ \to \R_+$ an increasing convex function such that $\gamma(0)=\gamma'(0)=0$ and suppose that $\gamma$ satisfies \eqref{gamma}. Suppose also that, for a given $n\in \N^*$, a probability measure $\mu$ on $X$ satisfies, for some constants $a>0, b\geq 1$, the following concentration property : 
\begin{equation}\label{conc 1/2}
\mu^n(X^n\setminus A_t^n) \leq b e^{-t/a},\qquad \forall t\geq0,
\end{equation}
for all $A \subset X^n$ such that $\mu^n(A)\geq 1/2.$\\
Then $\mu$ satisfies the following property : for all $s\in (0,1)$ and for all $A \subset X^n$\,,
\begin{equation}\label{conc gal}
\mu^n(X^n\setminus A_t^n)^{1/(1-s)^{r-1}} \mu^n(A)^{1/s^{r-1}}\leq {b}e^{-t/a},\qquad \forall t\geq0,
\end{equation}
where the exponent $r$ is defined by $r=\sup_{x>0} x\gamma_+'(x)/\gamma(x) \in(1,\infty)$\,, (here $\gamma'_+$ stands for the right derivative).

Conversely,  if the concentration property \eqref{conc gal} holds, then one has (by optimizing over all $s\in (0,1)$),  for all $A \subset X^n$ such that $\mu^n(A)\geq 1/2,$ for all $t > \max(a\log (2b),0)$,
\begin{equation*}
\mu^n(X^n\setminus A_t^n) \leq \inf_{s\in (0,1)} \left(b^{(1-s)^{r-1}} 2^{\frac{(1-s)^{r-1}}{s^{r-1}}}e^{-t (1-s)^{r-1}/a}\right)= be^{-t(1-\varepsilon(t))^r/a},
\end{equation*}
with $\varepsilon(t)=\left(\frac{\log 2}{\frac ta-\log b}\right)^{1/r}$.
\end{lem}

\proof
The fact that $1<r<\infty$ follows from \eqref{gamma} and the convexity inequalities $\gamma(2x) \geq \gamma(x) + x\gamma'(x)$ and $\gamma(x)/x <\gamma'(x)$, $x>0$.

To clarify the notations, we will omit some of the dependencies in $n$  in this proof.
The fact  that \eqref{conc 1/2} implies \eqref{conc gal} is a consequence of 
the following set inclusions (that are justified at the end of the proof):

$\displaystyle
(a)\quad A \subset X^n \setminus \left(\left(X^n\setminus A_{u}\right)_u\right),\quad  \forall  u\geq 0\,,$

and for all $s\in (0,1)$, 

$\displaystyle
(b)\quad \left(A_{u}\right)_v\subset A_{\left(u^{1/r}+v^{1/r}\right)^r}\subset A_{\frac{u}{s^{r-1}} +\frac v{(1-s)^{r-1}}},\quad \forall u,v\geq 0.
$

\noindent
The last inclusion above follows from the identity, 
\begin{equation}\label{infiden}
\left(u^{1/r}+v^{1/r}\right)^r=\inf_{s\in(0,1)}\left\{\frac{u}{s^{r-1}} +\frac v{(1-s)^{r-1}}\right\}.
\end{equation}
Let $t\geq 0$, $s\in (0,1)$  and $A\subset X^n$ and let us consider the set 
$B=A_{s^{r-1}t}$. \\If $\mu(B) \geq 1/2$ then by applying first (b) for $u=s^{r-1}t$ and $v= (1-s)^{r-1} t$, and then the concentration property   \eqref{conc 1/2},
we get 
\[\mu(X^n\setminus A_t)\leq \mu\left(X^n\setminus B_{(1-s)^{r-1} t}\right)\leq be^{-(1-s)^{r-1} t/a}.\]
\\If $\mu(B) < 1/2$ then $\mu(X^n\setminus B)\geq 1/2$. Therefore  by applying first (a) for $u=s^{r-1}t$ and then the concentration property   \eqref{conc 1/2}, we get 
\[\mu( A)\leq \mu\left(X^n\setminus ((X^n\setminus B)_{s^{r-1} t}\right)\leq be^{-s^{r-1} t/a}.\]
As a consequence in any case the concentration property \eqref{conc gal} holds.


Now let us justify the inclusion properties (a) and (b).\\
To prove (a) let us show that $A \cap \left(X^n\setminus A_{u}\right)_u =\emptyset$. Suppose on the contrary that there is some $x \in A \cap \left(X^n\setminus A_{u}\right)_u$, then there is some $y\in X^{n}\setminus A_u$ such that $\sum_{i=1}^n\gamma(d(x_i,y_i))\leq u$. But, since $y \in X^n\setminus A_u$, it holds $\sum_{i=1}^n\gamma(d(y_i,z_i))>u$ for all $z \in A$.
In particular, taking $z=x$, one gets  a contradiction. \\
Finally, let us show (b). According to e.g. \cite[Lemma 4.7]{GRS13}, the function $x \mapsto \gamma^{1/r}(x)$ is subbadditive. It follows easily that $(x,y) \mapsto \left(\sum_{i=1}^n \gamma(d(x_i,y_i))\right)^{1/r}$ defines a distance on $X^n$. Point (b) then follows immediately from the triangle inequality.
\endproof

For the next corollary, recall the definition of $B_r^n$ given in \eqref{Bpn}.

\begin{cor}\label{gozlan-bis}
Let $r\geq 2$ and consider the cost $c(x,p) = \int \|x-y\|^r\,p(dy)$, $x\in \R^m$, $p\in \mathcal{P}_1(\R^m)$, where $\|\,\cdot\,\|$ is a norm on $\R^m.$ 
For a probability measure $\mu \in \mathcal{P}_r(\R^m)$, the following propositions are equivalent :
\begin{enumerate}
\item There exist $a_1,b_1>0$ such that, $\forall n\in \N^*$, 
\[
\mu^n(A + t^{1/r}B_r^n) \geq 1 - b_1e^{-t/a_1},\qquad \forall t\geq 0,
\]
for all sets $A$ such that $\mu^n(A)\geq 1/2.$
\item There exist $a_2,b_2>0$ such that, $\forall n\in \N^*$, 
\[
\mu^n (f > \mathrm{med}\,(f) + r) \leq b_2e^{-t^r/a_2},\qquad \forall t\geq 0,
\]
for all $f:\left(\R^m\right)^n \to \R$ which are $1$-Lipschitz with respect to the norm $\|\,\cdot\,\|_r^n$ defined on $\left(\R^m\right)^n$ by 
\[
\| x\|_r^n = \left(\sum_{i=1}^n \|x_i\|^r\right)^{1/r},\qquad x\in\left(\R^m\right)^n.
\]
\item There exist $a_3,b_3>0$ such that, $\forall n\in \N^*$, $\forall s\in (0,1)$, and $\forall A \subset \left(\R^m\right)^n$\,,
\[
\mu^n(\left(\R^m\right)^n\setminus A_t^n)^{1/(1-s)^{r-1}} \mu^n(A)^{1/s^{r-1}} \leq b_3e^{-t/a_3},\qquad \forall t\geq 0,
\]
where $A_t^n = \{x \in \left(\R^m\right)^n ; c_A^n(x)\leq t\} =  A + t^{1/r} B_r^n$.
\item $\exists a_4>0$ such that $\forall s\in (0,1)$, $\mu$ satisfies \[\mathbf{T}_r(a_4/s^{r-1},a_4/(1-s)^{r-1})\,.\]
\item   $\exists a_5>0$ such that $\mu$ satisfies $\mathbf{T}^+_r(a_5)$ (which is  equivalent to $\mathbf{T}^-_r(a_5)$ for that cost).
\end{enumerate}
Moreover $(1) \Leftrightarrow (2)$ with $a_2=a_1$ and $b_2=b_1$, $(3) \Rightarrow (4)$ with $a_4=a_3$, $(4) \Rightarrow (3)$ with  $a_3=a_4$ and $b_3=1$, $(4)\Leftrightarrow (5)$ with $a_4=a_5$, $(1) \Rightarrow (3)$ with $a_3=a_1$ and $b_3=b_1$,   $(3)\Rightarrow (1)$ with $b_1=b_3^{(1-s)^{r-1}} 2^{\frac{(1-s)^{r-1}}{s^{r-1}}}$ and $a_1 = \frac{a_3}{(1-s)^{r-1}}$ for any $s\in (0,1)$. 
\end{cor}
Note that this result is not as general as possible; see \cite[Theorem 1.3]{Goz09} for a similar statement involving more general cost functions.
\proof
The equivalence $(1) \Leftrightarrow (2)$ is very classical (see e.g \cite[Proposition 1.3]{Led01}).\\
The implications $(1) \Rightarrow (3)$ and $(3) \Rightarrow (1)$ are given in Lemma \ref{equiv conc 1/2}. $(3) \Rightarrow (4)$ and $(4) \Rightarrow (3)$ are consequences of 
Theorem \ref{gozlan}. \\
If the property (4) holds,  then for all $\nu_1\in \mathcal{P}_r$, 
\begin{equation*} 
\mathcal{T}_r(\nu_1 , \mu)=\mathcal{T}_c(\nu_1 | \mu) \leq \frac{a_4}{s^{r-1}} H(\nu_1 | \mu)\qquad\forall s\in (0,1).
\end{equation*}
As $s$ goes to 1, we get (5), $\mu$ satisfies   $\mathbf{T}^+_r(a_4)$ or equivalently $\mathbf{T}^-_r(a_4)$.\\
Conversely assume that (5) holds. By the triangular inequality, we get for all $\nu_1,\nu_2\in \mathcal{P}_r$,
\begin{align*}
\mathcal{T}_c(\nu_1 | \nu_2)&=\mathcal{T}_r(\nu_1 , \nu_2)\leq\left(\mathcal{T}_r(\nu_1 , \mu)^{1/r}+ \mathcal{T}_r(\mu , \nu_2)^{1/r}\right)^r\\
&\leq \left((a_5 H(\nu_1 | \mu))^{1/r}+ (a_5 H(\nu_2 | \mu))^{1/r}\right)^r.
\end{align*}
The property (4) with $a_4=a_5$ then follows from the identity \eqref{infiden}.
\endproof

\subsubsection{Barycentric costs}
When $c(x,p) = \|x - \int y\,p(dy) \|^r$, $x\in \R^m,p\in \mathcal{P}_1(\R^m)$, for some norm $\|\,\cdot\,\|$ on $\R^m$, then the enlargement of a set $A \subset \left(\R^m\right)^n$ reduces to
\[
A_t^n=\overline{\mathrm{conv}}(A) + t^{1/r}B_r^n,
\]
denoting by $\overline{\mathrm{conv}}(A)$ the closed convex hull of $A$ and $B_r^n$ as defined in \eqref{Bpn}. Indeed, denoting $\|\,\cdot\,\|_r^n$, for the norm defined on $\left(\R^m\right)^n$ by $\|x\|_r^n = \left(\sum_{i=1}^n \|x_i\|^r\right)^{1/r}$, then,  for all $x\in  \left(\R^m\right)^n$, it holds 
$c_A(x) = \inf_{y\in C}\{\|x-y\|_r^n\} = \inf_{y\in \overline{C}}\{\|x-y\|\}$, with $C=\{ \int y\,p(dy) ; p\in \mathcal{P}_1(A)\}.$ It is well known that $\overline{C} = \overline{\mathrm{conv}}\,(A)$, which proves the claim.

The result below shows in particular that inequalities $\T_2^\pm$ are responsible for Gaussian dimension-free concentration for convex and concave Lipschitz functions.
\begin{cor}\label{gozlan-ter}
Let $r\geq 2$ and consider the cost $c(x,p) = \|x - \int y\,p(dy)\|^r$, $x \in \R^m$, $p\in \mathcal{P}_1(\R^m)$. 
For $\mu \in \mathcal{P}_1(\R^m)$, the following propositions are equivalent :
\begin{enumerate}
\item There exist $a_1,b_1>0$ such that, $\forall n\in \N^*$, 
\begin{equation*}
\mu^n(A + t^{1/r}B_r^n) \geq 1 - b_1e^{-t/a_1},\qquad \forall t\geq 0,
\end{equation*}
for any set $A$ which is either \emph{convex} or the \emph{complement of a convex set} and such that $\mu^n(A)\geq 1/2.$
\item There exist $a_2,b_2>0$ such that, $\forall n\in \N^*$, 
\[
\mu^n (f > \mathrm{med}\,(f) + t) \leq b_2e^{-t^r/a_2},\qquad \forall t\geq 0,
\]
for all $f:\left(\R^m\right)^n \to \R$ which is either \emph{convex} or \emph{concave} and $1$-Lipschitz with respect to the norm $\|\,\cdot\,\|_r^n$ defined on $\left(\R^m\right)^n$ by 
\[
\| x\|_r^n = \left(\sum_{i=1}^n \|x_i\|^r\right)^{1/r},\qquad x\in\left(\R^m\right)^n.
\]
\item There exist $a_3,b_3>0$ such that, $\forall n\in \N^*$, $\forall s\in (0,1)$, and $\forall A \subset \left(\R^m\right)^n\,,$
\[
\mu^n(\left(\R^m\right)^n\setminus A_t^n)^{1/(1-s)^{r-1}} \mu^n(A)^{1/s^{r-1}} \leq b_3e^{-t/a_3},\qquad \forall t\geq 0,
\]
where $A_t^n = \{x \in \left(\R^m\right)^n ; c_A^n(x)\leq t\} = \overline{\mathrm{conv}}\, A + t^{1/r} B_r^n$.
\item There exists $a_4>0$ such that   $\mu$ satisfies $\overline{\mathbf{T}}_r(\frac{a_4}{s^{r-1}},\frac{a_4}{(1-s)^{r-1}})$ $\forall s\in (0,1)$.
\item There exists $a_5>0$ such that $\mu$ satisfies $\overline{\mathbf{T}}^+_r(a_5)$ and $\mu $ satisfies  $\overline{\mathbf{T}}^-_r(a_5)$.
\end{enumerate}
Moreover $(1) \Leftrightarrow (2)$ with $a_2=a_1$ and $b_2=b_1$, $(3) \Rightarrow (4)$ with $a_4=a_3$, $(4) \Rightarrow (3)$ with  $a_3=a_4$ and $b_3=1$, $(4)\Leftrightarrow (5)$ with $a_4=a_5$, $(1) \Rightarrow (3)$ with $a_3=a_1$ and $b_3=b_1$,   $(3)\Rightarrow (1)$ with $b_1=b_3^{(1-s)^{r-1}} 2^{\frac{(1-s)^{r-1}}{s^{r-1}}}$ and $a_1 = \frac{a_3}{(1-s)^{r-1}}$ for any $s\in (0,1)$. 
\end{cor}
\proof
Adapting \cite[Proposition 1.3]{Led01}, one sees easily that $(1)\Leftrightarrow (2),$ and, according to Theorem \ref{gozlan}, $(3) \Leftrightarrow (4)$. 

Let us show that $(3)$ implies $(1)$. Let $A$ be a convex subset. As in Lemma \ref{equiv conc 1/2},  if $\mu^n(A)\geq 1/2$, then, by applying $(3)$ to $A$  and since $A_t=\overline{A}+ t^{1/r} B_r^n$, we get (1) for convex sets  with $b_1=b_3^{(1-s)^{r-1}} 2^{\frac{(1-s)^{r-1}}{s^{r-1}}}$ and $a_1 = \frac{a_3}{(1-s)^{r-1}}$ for $s\in (0,1)$. Let $D=(\R^m)^n\setminus A$ and assume that $\mu(D)\geq 1/2$. For all $t>0$, the set $C= (\R^m)^n\setminus(D+t^{1/r} B_r^n)$ is convex and satisfies for all $t'<t$, 
\[ C_{t'}=(\overline{C}+t'^{1/r} B_r^n)\subset (\R^m)^n\setminus D.\]
Since $\mu^n(D)\geq 1/2$, it follows that $\mu^n((\R^m)^n\setminus C_{t'})\geq 1/2$.
As a consequence, applying $(3)$ to the set  $C$, we obtain for all $t>t'>0$, for all $s\in(0,1)$\,,
 \[\mu^n((\R^m)^n\setminus(D+t^{1/r} B_r^n))=\mu^n(C)\leq b_3^{s^{r-1}} 2^{\frac{s^{r-1}}{(1-s)^{r-1}}} e^{-\frac{s^{r-1}t'}{a_3}}.\]
As $t'$ goes to $t$, this implies  the concentration property (1) for complement of convex sets.

We adapt  the proof of Lemma \ref{equiv conc 1/2} to get $(1) \Rightarrow (3)$.
The property (a) is replaced by the following, for all subset $A$,
\[(a') \quad A\subset \overline{\mathrm{conv}}\, A \subset (\R^m)^n\setminus[(X\setminus A_u)+u^{1/r} B_r^n],\qquad u\geq 0.\]
Since $A_u=\overline{\mathrm{conv}}\, A+u^{1/r} B_r^n$, this property $(a')$ is a simple consequence of  the property $(a)$ applied to the set $\overline{\mathrm{conv}}\, A$.
For the same reason, the property $(b)$ still holds.
 Then following  the proof of Lemma \ref{equiv conc 1/2},  by using  $(a') $ and $(b)$, with  the set $B=A_{s^{r-1}t}$, $s\in(0,1)$,   and applying the concentration property $(1)$ to the convex set $B$ or to it's  complement  $(\R^m)^n\setminus B$, we get  $(1)\implies (3)$ with $a_3=a_1$ and $b_3=b_1$.

The equivalence between (3) and (4) is a consequence of Theorem \ref{gozlan}.

If the property (4) holds,  then for all $\nu_1,\nu_2\in \mathcal{P}_r$ and for all $\forall s\in (0,1)$,
\begin{equation*} 
\overline{\mathcal{T}}_r(\nu_1 | \mu) \leq \frac{a_4}{s^{r-1}} H(\nu_1 | \mu), \quad \mbox{and} \quad \overline{\mathcal{T}}_r(\mu | \nu_2) \leq \frac{a_4}{(1-s)^{r-1}} H(\nu_2 | \mu) .
\end{equation*}
As $s$ goes to 1 or to 0, we get (5) -- that $\mu$ satisfies   $\overline{\mathbf{T}}^+_r(a_4)$ and $\overline{\mathbf{T}}^-_r(a_4)$.\\
Conversely assume that (5) holds, then (4) follows with $a_4=a_5$ by the following triangular inequality, for all $\nu_1,\nu_2\in \mathcal{P}_r$,
\begin{equation*}
\overline{\mathcal{T}}_r(\nu_1 | \nu_2)^{1/r}\leq\overline{\mathcal{T}}_r(\nu_1 | \mu)^{1/r}+ \overline{\mathcal{T}}_r(\mu | \nu_2)^{1/r}.
\end{equation*}

\endproof

\section{Universal transport cost inequalities with respect to Hamming distance and Talagrand's concentration of measure inequalities}\label{sectionhamming}
This section is devoted to universal transport-entropy inequalities associated to the weak transport costs $\Tw$ and $\Th$ with respect to the Hamming distance.

\subsection{Transport inequalities for Marton's costs} 
In this section, we recall a transport-entropy inequality  obtained by  Dembo \cite{Dem97}, improving upon preceding works by Marton \cite{Mar96b} and used in its dual form by the third named author  \cite{Sam07} to obtain optimal concentration bounds for supremum of empirical processes.

Let us introduce some notation. For $t \in (0,1)$, define $\alpha_t$ by
\[
\alpha_t (u)= 
\begin{cases} 
\frac{t(1-u)\log(1-u)-(1-t u)\log(1-t u)}{t(1-t)} & \text{if } 0\leq u\leq 1 \\
+\infty & \text{otherwise}
\end{cases}
\]
and also  set $\alpha_0(u)= (1-u)\log(1-u)+u$ and $\alpha_1(u) = -u-\log(1-u)$ when $u \in (0,1)$ (and $+\infty$ otherwise). Let us consider the cost of the form $\Tw$ associated to $\alpha_t$:
\[
\Tw_{\alpha_t}(\nu_1|\nu_2) = \inf \int \alpha_t\left(\int \1_{x\neq y}\,p_x(dy)\right)\,\nu_2(dx)\,,
\]
where the infimum runs over the set of kernels $p$ such that $\nu_2p=\nu_1$.

\begin{thm}\label{hamming1}
Let $(X,d)$ be a polish space, $t \in (0,1)$ and $\mu \in \mathcal{P}(X)$. Then, for all probability measures $\nu_1,\nu_2$ on $X$, it holds
\begin{equation}\label{dembo}
\Tw_{\alpha_{t }}(\nu_1|\nu_2)\leq \frac 1{1-t} H(\nu_1|\mu) +\frac 1t  H(\nu_2|\mu) .
\end{equation}
For $t=0$, it also holds
\[
\Tw_{\alpha_{0}}(\nu_1|\mu)\leq   H(\nu_1|\mu),
\]
and for $t=1$,
\[
\Tw_{\alpha_{1 }}(\mu|\nu_2)\leq H(\nu_2|\mu) .
\]
\end{thm}
The transport inequality \eqref{dembo} is due to Dembo \cite[Theorem 1.(i)]{Dem97}. A short proof of this theorem is given in \cite{Sam07} (see Lemma 2.1.)
As  shown in \cite{Sam07}, the behavior  of the family of cost functions $\alpha_t$ allows to capture optimal bounds for the deviations of suprema of empirical bounded processes. 

Let us just recall simple and useful corollaries of Theorem \ref{hamming1}. 
First observing that $\alpha_t(u) \geq u^2/2$, we immediately recover using Theorem \ref{gozlan} (implication (i) $\Rightarrow$ (ii)) the following celebrated concentration result by Talagrand (see \cite[Theorem 4.1.1]{Tal95}).
\begin{cor}
For any probability measure $\mu$ on $X$, it holds
\[
\mu^n(X^n\setminus A_t^n) \leq \frac{1}{\mu^n(A)^{s/(1-s)}}e^{-st/2},\qquad \forall t>0,\forall s \in (0,1),
\]
for all $A \subset X^n$ and $n\in \N^*$, where 
\begin{align*}
 A_t^n = \Big\{ y \in X^n : \exists p \in \mathcal{P}(X^n)  \text{ with }  p(A) =1  \text{ such that } 
  \sum_{i=1}^n \left(\int \1_{x_i\neq y_i}\,p(dx)\right)^2 \leq t\Big\}.
\end{align*}
\end{cor}
We refer to \cite{Tal95,Led01,AS08,DP09,Pau14} for applications of this concentration inequality, under the so-called convex hull distance .

\begin{cor}
Suppose that $\mu$ is a probability on $\R^m$ (equipped with some arbitrary norm $\|\,\cdot\,\|$) such that the diameter of $\mathrm{supp}(\mu)$ is bounded by $M>0.$
Then $\mu$ satisfies the inequality $\widetilde{\T}_2(4M^2,4M^2)$ and thus $\overline{\T}_2(4M^2,4M^2).$
\end{cor} 
\proof
Observe that $\tilde{\alpha}_t(u) \geq u^2/2$, for all $u\in [0,1]$ and $t=1/2.$ Furthermore, if $\nu_1,\nu_2$ are absolutely continuous with respect to $\mu$ then $\mathrm{supp}(\nu_i) \subset \mathrm{supp} (\mu).$ Therefore, if $\pi(dxdy) = \nu_1(dx) p_x(dy)$ is a coupling between $\nu_1$ and $\nu_2$, then $\int \|x-y\|\,p_x(dy) \leq M \int \1_{\{x\neq y\}} \,p_x(dy)$, for $\nu_1$-almost all $x$, and so 
\begin{align*}
 \frac{1}{2M^2}\int \left( \int \|x-y\|\,p_x(dy)\right)^2 \nu_1(dx) & \leq 
 \int \tilde{\alpha}_t\left( \frac{1}{M}\int \|x-y\|\,p_x(dy)\right)\,\nu_1(dx) \\
& \leq \int \tilde{\alpha}_t\left( \int \1_{\{x\neq y\}} \,p_x(dy)\right) \,\nu_1(dx).
\end{align*}
Optimizing over all $\pi$, and then using Theorem \ref{hamming1} for $t=1/2$, completes the proof.
\endproof
We recover from the preceding result, and Corollary \ref{gozlan-ter}, the well-known fact that any probability measure with a bounded support satisfies dimension-free Gaussian type concentration for convex/concave Lipschitz functions.

\subsection{Transport inequalities for Samson's costs} Now we consider a stronger variant of Theorem \ref{hamming1} involving costs of the form $\Th$. 
To state this result we need to introduce additional notation. For $t\in (0,1)$, one sets 
\[
\beta_t (u):=\sup_{s\in \R}\left\{su-\beta_t ^*(s)\right\}, \qquad u\in \R.
\]
where $\beta_t ^*$ is defined by 
\[
\beta_t^*(s):=\frac{t  e^{(1-t )s}+(1-t ) e^{-t  s} -1}{t (1-t )},\qquad s\in \R.
\]
We extend the definition for $t\in \{0,1\}$ by setting
\[
\beta^*_0(s) = e^{s} -s -1\qquad \text{and}\qquad \beta^*_1(s) = e^{-s} +s -1,\qquad \forall s\in \R.
\]
In general, $\beta_t$ does not have an explicit expression, but for $t\in \{0,1\}$ an easy calculation shows that
\begin{align*}
\beta_0(u)&=(1+u)\log(1+u)-u,\quad u\geq-1\\
\beta_1(u)&=\beta_0(-u)=(1-u)\log(1-u)+u,\quad u\leq1.
\end{align*}
Finally, consider the cost of the form $\Th$ associated to these functions:
\[
\Th_{\beta_t}(\nu_1|\nu_2) = \inf \iint \beta_t\left(\1_{x\neq y}\frac{dp_x}{d\mu}(dy)\right)\,\mu(dy)\nu_2(dx),
\]
where the infimum runs over the set of kernels $p$ such that, in addition, $p_x \ll \mu$ for $\nu_2$-almost all $x \in X$.

\begin{thm}\label{hamming2}
Let $(X,d)$ be a compact metric space or a countable set of isolated points. Let $t \in (0,1)$
and  $\mu \in \mathcal{P}(X)$. Then, for all probability measures $\nu_1,\nu_2$ on $X$, it holds
\begin{equation}\label{eq:hamming2}
\Th_{\beta_{t }}(\nu_1|\nu_2)\leq \frac 1{1-t} H(\nu_1|\mu) +\frac 1t  H(\nu_2|\mu).
\end{equation}
For $t=0$, it also holds
\[
\Th_{\beta_{0}}(\nu_1|\mu)\leq   H(\nu_1|\mu),
\]
and for $t=1$,
\[
\Th_{\beta_{1 }}(\mu|\nu_2)\leq H(\nu_2|\mu) .
\]
\end{thm}
By Proposition \ref{prop:bg} and Theorem \ref{Kantorovich T hat}, one sees that Theorem \ref{hamming2} is exactly the dual form of Theorem 1.1 of \cite{Sam07} (for $n=1$). This new expected formulation of Theorem 1.1 in \cite{Sam07} is therefore a direct consequence of the generalization of the Kantorovich theorem (Theorem \ref{Kantorovich general}).

A direct consequence of Theorem \ref{hamming2} and implication (i) $\Rightarrow$ (ii) of Theorem \ref{gozlan} is the following deep concentration result that improves the  one by Talagrand \cite[Theorem 4.2.]{Tal96b}. 
\begin{cor}
For any probability measure $\mu$ on $X$, it holds
\[
\mu^n(X^n\setminus A_{s,t}^n) \leq \frac{1}{\mu^n(A)^{s/(1-s)}}e^{-st},\qquad \forall t>0,\forall s \in (0,1),
\]
for all $A \subset X^n$ and $n\in \N^*$, where 
\begin{multline*}
A_{s,t}^n = \left\{ y \in X^n : \exists p \in \mathcal{P}(X^n)\text{ with } p(A)=1 \text{ and } p_i\ll \mu, \forall i\vphantom{\sum_{i=1}^n \int \beta_s\left(\1_{x_i\neq y_i}\,\frac{dp_i}{d\mu}(x_i)\right)\,\mu(dx)}\right.\\
\text{ such that } \left. \sum_{i=1}^n \int \beta_s\left(\1_{x_i\neq y_i}\,\frac{dp_i}{d\mu}(x_i)\right)\,\mu(dx) \leq t\right\},
\end{multline*}
where we recall that $p_i$ denotes the $i$-th marginal of $p.$
\end{cor}

In Talagrand's paper \cite{Tal96b}, this kind of concentration result is the main ingredient to get deviation inequalities of Bernstein\-type  for sup\-rema of centered bounded empirical processes. Starting from the optimal  transport inequality of Theorem \ref{hamming2}, the third-named author has obtained optimal constants in the Bernstein bounds for the deviations under and above the mean \cite{Sam07}. This transportation method is an alternative of the entropy method introduced by Ledoux \cite{Led95}, and then developed by many authors. We refer to the book by Boucheron, Lugosi and Massart \cite{BLM13} for more development in this field. 

Below, we sketch the proof of Theorem \ref{hamming2}, by revisiting and to some extent simplifying some of the arguments given in \cite{Sam07} with the help of the duality results developed in the present paper and in \cite{GRS11}. 
The first of these duality formulas is Kantorovich duality for the cost $\Th$ given in Theorem \ref{Kantorovich T hat}. The second formula is more classical and is recalled in the following proposition.
\begin{prop}\label{prop:Duality GRS11}
Let $\beta : [0,\infty) \to \R$ be a lower semi-continuous strictly convex and super-linear function (\text{i.e.} $\beta(x)/x \to +\infty$ as $x \to \infty$). Let $\mu$ be a probability measure on a polish space $X$ and denote by $U_\beta$ the function defined on $\mathcal{P}(X)$ by
\[
U_\beta(\nu) = \int \beta\left(\frac{d\nu}{d\mu}\right)\,d\mu,
\]
if $\nu$ is absolutely continuous with respect to $\mu$ and $+\infty$ otherwise.
Then, for any bounded continuous function $\varphi$ on $X$, it holds
\[
\sup_{\nu \in \mathcal{P}(X)} \left\{ \int \varphi(x)\,p(dx) - U_\beta(\nu)\right\} = \inf_{t\in \R} \left\{ \int \beta^\circledast(\varphi(x) + t)\,\mu(dx) - t\right\},
\]
where $\beta^\circledast$ denotes the \emph{monotone conjugate} of $\beta$, defined by $\beta^\circledast(x)=\sup_{y\geq0}\{xy-\beta(y)\}$, $x \in \R.$
\end{prop}
We refer to \cite[Proposition 2.9]{GRS11} for a short and elementary proof of this result.

We begin with an elementary lemma connecting monotone and usual conjugates of our functions $\beta_t$. The proof is left to the reader.
\begin{lem}\label{Lem conj monotone}
For all $u\in \R$, $\beta_t^\circledast (u) = \beta_t^*([u]_+)$.
\end{lem}
The next lemma gives an expression of $\widehat{Q}_{\beta_t}\varphi$, which will be crucial in order to establish the dual form of the transport inequality.

\begin{lem}\label{Lem Q chapeau}
Let $(X,d)$ be a compact metric space or a countable set of isolated points and $t \in [0,1]$.
For all bounded continuous function $\varphi : X \to \R$, there exists a function $v:X \to \R$ such that $v(x) \leq \varphi(x)$ for all $x \in X$ and such that 
\[
\widehat{Q}_{\beta_t}\varphi (x) = v(x) - \int \beta_t^*\left([v(x)-v(y)]_+\right)\,\mu(dy).
\]

\end{lem}
\proof
Fix $x \in X$ and recall that
\[
\widehat{Q}_{\beta_t}\varphi (x) = \inf\left\{ \int \varphi(y)\,p(dy) + \int \beta_t\left( \1_{x\neq y} \frac{dp}{d\mu}(y)\right)\,\mu(dy) \right\}, 
\]
where the infimum runs over the set of probability measures $p\ll\mu$ on $X\setminus\{x\}.$ 

A probability $p$ of this set can be written $p= \alpha \delta_x + (1-\alpha) q$, with $\alpha = p(\{x\})$ and where $q$ is another probability such that $q \ll \mu$ and $q(\{x\})=0.$ 
So 
\begin{align}\label{eq:inf-chapeau}
\widehat{Q}_{\beta_t}\varphi (x)- \varphi(x)  = \inf_{\alpha \in [0,1]} & \inf_{q \ll \mu, q(\{x\})=0}
 \Big\{\int (1-\alpha)(\varphi(y)-\varphi(x))\,q(dy) + \int \beta_t\left( (1-\alpha)\1_{x\neq y} \frac{dq}{d\mu}(y)\right)\,\mu(dy) \Big\}.
\end{align}

Consider the probability measure $\mu_x$ with the following density with respect to $\mu$: 
$\frac{d\mu_x}{d\mu}(y) = \lambda^{-1}\mathbf{1}_{x\neq y},$ where $\lambda = \mu(X \setminus \{x\})>0$ (we assume of course that $\mu$ is not the Dirac mass at point $x$), then $q(\{x\})=0$ and $q\ll \mu$ if and only if $q \ll \mu_x$ and in this case $\frac{dq}{d\mu_x} = \lambda \frac{dq}{d\mu}$, $\mu_x$-almost everywhere. Therefore, \eqref{eq:inf-chapeau} becomes
\begin{align*}
\widehat{Q}_{\beta_t}\varphi (x)- \varphi(x)  =  \inf_{\alpha \in [0,1]} \inf_{q \ll \mu_x} 
\Big\{ \int  (1-\alpha)(\varphi(y)-\varphi(x))\,q(dy) + \lambda\int \beta_t\left( \frac{(1-\alpha)}{\lambda}\frac{dq}{d\mu_x}(y)\right)\,\mu_x(dy) \Big\}.
\end{align*}
So it holds
\begin{align*}
\widehat{Q}_{\beta_t}\varphi (x)-\varphi(x) &= 
\inf_{\alpha \in [0,1)}\inf_{q \ll \mu_x}\Big\{ \int (1-\alpha)(\varphi(y)-\varphi(x))\,q(dy) + \lambda\int \beta_t\left( \frac{(1-\alpha)}{\lambda} \frac{dq}{d\mu_x}(y)\right)\,\mu_x(dy) \Big\}\\
&= \inf_{\alpha \in [0,1)}-\inf_{r\in \R} \left\{ \lambda\int \beta_t^\circledast\left( \frac{(1-\alpha)(\varphi(x)-\varphi(y)) + r}{(1-\alpha)}\right)\,\mu_x(dy) - r\right\}\\
& = \inf_{\alpha \in [0,1)}-\inf_{v\in \R} \left\{ \int \beta_t^*\left([v-\varphi(y)]_+ \right)\mathbf{1}_{x\neq y}\,\mu(dy) - (1-\alpha)(v-\varphi(x))\right\}\\
& = \inf_{\alpha \in [0,1]}\sup_{v\in \R} \left\{  (1-\alpha)(v-\varphi(x)) - \int \beta_t^*\left([v-\varphi(y)]_+ \right)\1_{x\neq y}\,\mu(dy) \right\}\\
& = \sup_{v\in \R}\inf_{\alpha \in [0,1]} \left\{  (1-\alpha)(v-\varphi(x)) - \int \beta_t^*\left([v-\varphi(y)]_+ \right)\1_{x\neq y}\,\mu(dy) \right\}\\
& = \sup_{v\in \R}\left\{  -[v-\varphi(x)]_- - \int \beta_t^*\left([v-\varphi(y)]_+ \right)\1_{x\neq y}\,\mu(dy) \right\},
\end{align*}
where the second equality comes from Proposition \eqref{prop:Duality GRS11} and Lemma \ref{Lem conj monotone}, and the last one from (an elementary version of) the Min-Max theorem.
In particular,
\begin{align*}
\widehat{Q}_{\beta_t}\varphi(x)& = \varphi(x)- [v(x)-\varphi(x)]_- - \int \beta_t^*\left([v(x)-\varphi(y)]_+\right)\1_{x\neq y}\,\mu(dy),\\
& =  \min(v(x) , \varphi(x)) -  \int \beta_t^*\left([v(x)-\varphi(y)]_+\right)\1_{x\neq y}\,\mu(dy).
\end{align*}
for some function $v$ (realizing the supremum in the last identity). 

For a fixed $x \in X$, consider the function $F(v) = -[v-\varphi(x)]_- - \int \beta_t^*\left([v-\varphi(y)]_+ \right)\1_{x\neq y}\,\mu(dy)$, $v\in \R.$ Since $\beta_t^*$ is increasing on $[0,\infty)$, the function $F$ is clearly non-increasing on $[\varphi(x),+\infty)$. Therefore $F$ reaches its supremum on $(-\infty,\varphi(x)]$. On $(-\infty,\varphi(x))$, the function $F$ is differentiable and it holds 
\begin{align*}
F'(v) 
&=  
1 - \int e^{(1-t) [v(x)-\varphi(y)]_+}\1_{v(x)>\varphi(y)}\1_{x\neq y}\,\mu(dy) 
+\int e^{-t [v(x)-\varphi(y)]_+}\1_{v(x) >\varphi(y)}\1_{x\neq y}\,\mu(dy)\\
& = 1 - \int e^{(1-t) [v(x)-\varphi(y)]_+}\1_{x\neq y}\,\mu(dy) +\int e^{-t [v(x)-\varphi(y)]_+}\1_{x\neq y}\,\mu(dy).
\end{align*}
It is not difficult to prove the existence of a point $\bar{v}$ (independent of $x$) such that
\begin{equation}\label{eq:v-opt}
\int e^{(1-t) [\bar{v}-\varphi(y)]_+}\1_{x\neq y}\,\mu(dy) = 1+\int e^{-t [\bar{v}-\varphi(y)]_+}\1_{x\neq y}\,\mu(dy)
\end{equation}
and to check that the function $F$ reaches its supremum at $v(x) := \min(\bar{v},\varphi(x))$. 

Finally, note that $[v(x)-\varphi(y)]_+ = [v(x) - v(y)]_+$, which completes the proof.
\endproof

The next result is Lemma 2.2 of \cite{Sam07}.
\begin{lem}\label{Lem Sam07}Let $\mu$ be some probability on a measurable space $X$.
For every bounded function $v : X \to \R$, it holds for all $t \in [0,1]$\,,
\[
\left(\int e^{t v(x) - t \int \beta_t^*([v(x)-v(y)]_+)\mu(dy)}\mu(dx)\right)^{1/t} 
\left( \int e^{-(1-t)v(x)}\mu(dx)\right)^{1/(1-t)} \leq 1.
\]
\end{lem}
With these lemmas in hand, we are now in a position to prove Theorem \ref{hamming2}.
\proof[Proof Theorem \ref{hamming2}]
Fix $t \in (0,1)$ ; according to Proposition \ref{prop:bg}, the transport inequality \eqref{eq:hamming2} is equivalent to proving that
\begin{equation}\label{eq:proof hamming}
\left(\int e^{t \widehat{Q}_{\beta_t}\varphi(x)}\,\mu(dx)\right)^{1/t} \left( \int e^{-(1-t)\varphi(x)}\,\mu(dx)\right)^{1/(1-t)} \leq 1,
\end{equation}
for all bounded continuous function $\varphi: X \to \R.$ But according to Lemma \ref{Lem Q chapeau}, 
\[
\widehat{Q}_{\beta_t}\varphi(x) = v(x) - \int \beta_t^*([v(x)-v(y)]_+)\,\mu(dy),
\] 
for some function $v \leq \varphi$ (possibly depending on $t$ and on $\mu$). According to Lemma \ref{Lem Sam07}, it holds
\[
\left(\int e^{t \widehat{Q}_{\beta_t}\varphi(x)}\,\mu(dx)\right)^{1/t} \left( \int e^{-(1-t)v(x)}\,\mu(dx)\right)^{1/(1-t)} \leq 1.
\]
Since $v \leq \varphi$, this gives \eqref{eq:proof hamming} and completes the proof. 
\endproof

Now for the sake of completeness, we give a quick proof of Lemma \ref{Lem Sam07} in the particular case $t=1$.
The general case is more tricky and the interested reader is referred to \cite{Sam07}.

\proof[Proof of Lemma \ref{Lem Sam07} for $t=1$]
In this case, the conclusion of the lem\-ma amounts to proving that for all bounded measurable $v:X \to \mathbb{R}$, it holds
\begin{equation}\label{cocotte}
\int e^{H(v(x))} d\mu(x) \leq 1,
\end{equation}
where
\[H(v(x))= v(x)-\int v(y) d\mu(y)-D(v(x)),\]
with 
$D(v(x))=\int \beta_1^*([v(x)-v(y)]_+) d\mu(y)$.
Replacing everywhere $v$ by $\lambda v$, $\lambda\geq 0$, it is equivalent to showing that for all $\lambda\geq 0$, 
\[
\phi(\lambda)=\int e^{H(\lambda v(x))} d\mu(x) \leq 1.
\]
Since $\phi(0)=1$, it is sufficient to get that $\phi'(\lambda)\leq 0$ for all $\lambda\geq 0$.
Let us first observe that since $ \beta_1^*(h)=e^{-h}+h-1$, 
\[H(\lambda v(x))=\int \left(1-e^{-\lambda[v(x)-v(y)]_+}\right)d\mu(y)- \int \lambda[ v(y)- v(x)]_+ d\mu(y).\]
It follows that for $\lambda\geq 0$, 
\begin{align*}
\phi'(\lambda) &= \int \Big(\int [v(x)-v(y)]_+ e^{-\lambda[v(x)-v(y)]_+} d\mu(y)  -\int [ v(y)- v(x)]_+ d\mu(y)\Big) e^{H(\lambda v(x))} d\mu(x)\\
&=\iint [v(x)-v(y)]_+\left(e^{-\lambda[v(x)-v(y)]_++H(\lambda v(x))}-e^{H(\lambda v(y))} \right)d\mu(x)d\mu(y)
\end{align*}
For $v(x)\geq v(y)$ one has 
\[-\lambda[v(x)-v(y)]_++H(\lambda v(x))-H(\lambda v(y))=D(\lambda v(y))-D(\lambda v(x))\leq 0,\]
and therefore $\phi'(\lambda)\leq 0$ for $\lambda\geq 0$. This ends the proof of \eqref{cocotte}.
\endproof

\section{Discrete examples : Bernoulli, Binomial and Poisson laws}\label{sec:example}

In this section, we gather some basic examples of probability measures satisfying  weak transport inequalities. These examples are studied in detail in \cite{GRST16}. We start with the Bernoulli measure, from which we derive weak transport  inequalities for the binomial law and the Poisson distribution. Let us mention that \cite{GRST16} also contains transport-entropy inequalities for the uniform measure on the symmetric group that give back the concentration results obtained by Talagrand in \cite{Tal95}.

We first consider some results for the Bernoulli measure, derived in  \cite{Sam03}, and as such introduce some notations from there.  

Given $\rho  \in (0,1)$, define
$u_{\rho,0}  \colon \mathbb{R} \to \mathbb{R}_+ \cup \{+\infty\}$ as 
\begin{equation*}
 u_{\rho,0} (h) 
 =
\begin{cases}
\frac{1-\rho (1-h)}{\rho } \log \frac{1-\rho (1-h)}{1-\rho } + (1-h)\log(1-h)\,, & \text{if } -\frac{1-\rho }{\rho } \leq h \leq 1 \\
+ \infty & \text{otherwise}
\end{cases}
\end{equation*} 
and define
$u_{\rho,1}  \colon \mathbb{R} \to \mathbb{R}_+ \cup \{+\infty\}$ as 
\begin{equation*}
 u_{\rho,1} (h) 
  =
\begin{cases}
\frac 1\rho \left[(1-\rho -h) \log \frac{1-\rho -h}{1-\rho }-(1-h) \log(1-h)\right]\,, & \text{if } h \leq 1-\rho \\
+ \infty & \text{otherwise}
\end{cases}
\end{equation*} 

Finally we define, for $t\in\{0,1\}$, 
\[
{\theta}_{\rho ,t}(h) = \begin{cases}
u_{\rho,t} (h) & \text{if } h \geq 0 \\
u_{1-\rho,t }(-h) & \text{if } h < 0 \,.
\end{cases}
\]

In \cite{GRST16}, using a result from \cite{Sam03} and the duality results proved in Section  \ref{sec:particular cases},  the following weak transport inequalities are obtained  for the non symmetric Bernoulli measure.
Set $\mu_\rho := (1-\rho ) \delta_{0} +\rho  \delta_{1}$, $\rho  \in [0,1]$.

\begin{prop} \label{Bernoulli}
For all $ \rho  \in (0,1)$, it holds
\begin{equation}\label{Tb}
\Tb_{\theta_{\rho ,1}}(\mu_\rho |\nu)\leq H(\nu|\mu_\rho )
\quad \text{and}
\quad
 \Tb_{\theta_{\rho ,0}}(\nu | \mu_\rho )\leq H(\nu|\mu_\rho ) 
 \quad \forall \nu \in \mathcal{P}(\{0,1\}) .
 \end{equation}
\end{prop}

In \cite{GRST16}, a family of transport inequalities is given    that interpolates between the two transport inequalities of \eqref{Tb} for $t=0$ and $t=1$, as in Theorems \ref{hamming1} and \ref{hamming2}. Moreover in this paper, other related weak transport inequalities with  cost types $\Tw_{ \theta}$ are given. 

As explained in \cite{GRST16}, the cost functions ${\theta}_{\rho ,0 }$ and ${\theta}_{\rho ,1 }$ correspond to the optimal choice in the transport inequalities  \eqref{Tb}.

By  Theorem  \ref{tensorization},  the weak transport inequalities for the Bernoulli measure $\mu_\rho $ given in Proposition \ref{Bernoulli} tensorize. Hence, the  product of Bernoulli measures $\mu_\rho ^n:=\mu_\rho \otimes\cdots \otimes \mu_\rho $ on the hypercube $\{0,1\}^n$ satisfies  the following $n$-dimensional version of the $\Tb$-transport-entropy inequalities.  Recall  that  the corresponding  $n$-dimensional costs are  defined, for all $x=(x_1,\dots,x_n)\in \{0,1\}^n$
and all $p\in  \mathcal{P}(\{0,1\}^n)$,  respectively by
$$
\bar{c}^{(n)}_{\rho ,t }(x,q):=\sum_{i=1}^n \theta_{\rho ,t }\left(x_i-\int_{\{0,1\}} y_i \,q_i(dy_i)\right).
$$
where $q_i \in \mathcal{P}(\{0,1\})$ is the $i$-th marginal of $q$, and $t\in\{0,1\}$.
We denote by 
$\Tb_{\bar{c}^{(n)}_{\rho ,t }}$ 
 the corresponding transport cost.
Applying Theorem  \ref{tensorization}, we immediately get, from Proposition~\ref{Bernoulli}, the following 
weak transport-entropy inequalities for product of Bernoulli measures.

\begin{cor} \label{product-Bernoulli}
For all  $\rho  \in (0,1)$ and all integers  $n\geq 1$, it holds
\begin{equation*}
\Tb_{\bar{c}^{(n)}_{\rho ,1}}(\mu_\rho ^n|\nu)\leq H(\nu|\mu_\rho ^n)
\quad \text{and}
\quad
\Tb_{\bar{c}^{(n)}_{\rho ,0}}(\nu | \mu_\rho ^n)\leq H(\nu|\mu_\rho ^n) 
 \quad \forall \nu \in \mathcal{P}(\{0,1\}^n) .
\end{equation*}
\end{cor}

To prove weak transport cost inequalities for  the binomial distribution $B(n,\rho )$, $\rho \in(0,1)$, the basic idea of \cite{GRST16} is to project the $n$-dimensional transport cost inequalities of Corollary \ref{product-Bernoulli}, from the hypercube $\{0,1\}^n$ onto $I_n:=\{0,1,\dots,n\}$, the state space of $B(n,\rho )$.  Projection arguments are  useful tools  to reach transport-entropy inequalities as explained in the seminal work by Maurey (see Lemma 2 \cite{Mau91}). 
Let  $\mu_{n,\rho }$ denote the  binomial measure on $I_n$, \textit{i.e.}\ 
$\mu_{n,\rho }(k)=  \genfrac{(}{)}{0pt}{}{n}{k}\rho ^k(1-\rho )^{n-k}$ for all $k \in I_n$.
By using the fact that  the image measure of $\mu_\rho ^n$ by the projection $\varphi : \{0,1\}^n \ni (x_1,\dots,x_n) \mapsto \sum_{i=1}^n x_i \in I_n$ is the measure $\mu_{n,\rho }$, Corollary \ref{product-Bernoulli} provides the following transport-entropy inequalities. 
 \begin{cor}\label{pomme} \cite{GRST16}
For  all $\rho  \in (0,1)$ and all integers $n\geq 1$, it holds
\begin{equation*}
\Tb_{\theta_{\rho ,1,n}}(\mu_{n,\rho }|\nu)\leq H(\nu|\mu_{n,\rho }),
\; \text{and}
\;
\Tb_{\theta_{\rho ,0,n}}(\nu | \mu_{n,\rho })\leq H(\nu|\mu_{n,\rho }), 
 \quad \forall \nu \in \mathcal{P}(I_n), 
 \end{equation*}
  where $\theta_{\rho ,t ,n}(h):=n \theta_{\rho ,t }(h/n)$, $h \in \R$, $t\in\{0,1\}$.

\end{cor}


Following \cite{GRST16}, this result implies   weak transport-entropy inequalities for the Poisson probability measure   $p_\lambda$,
with parameter  $\lambda > 0$, defined for all $k\in \N$ by
$p_\lambda(k)=\frac{\lambda^k}{k!}e^{-\lambda}$.  The idea is to use
the weak convergence of the binomial distribution $\mu_{n,\rho_n}$, with $\rho_n:=\lambda/n$, towards the Poisson measure $p_\lambda$. 

Set, for $t \in \{0,1\}$, $h \in \mathbb{R}$,
$
c_{\lambda,t }(h) := \lim_{n\to \infty}  n \theta_{\rho_n,t }\left(\frac hn \right).$
One has \[c_{\lambda,0}(h) =\lambda\,w\left(\frac{h}{\lambda}\right)\1_{h\leq 0}, \quad{ and } \quad 
 c_{\lambda,1}(h) :=\lambda\,w\left(\frac{-h}{\lambda}\right)\1_{h\leq 0},\]
 where $w(h)= (1-h)\log(1-h)+h$ for $h\leq 1$ and $w(h)=+\infty$ if $h>1$.
The result for the Poisson measure   is the following.
 \begin{prop}\label{poisson}\cite{GRST16}
For all $\lambda>0$, it holds
\begin{equation*} 
\Tb_{c_{\lambda,0}}(p_\lambda|\nu)\leq  H(\nu|p_\lambda), 
\quad{ and }\quad
\Tb_{c_{\lambda,1}}(\nu|p_\lambda)\leq H(\nu|p_\lambda),\qquad \forall \nu \in \mathcal{P}(\N).
\end{equation*}
\end{prop}

\begin{rem}

Observe that these transport inequalities  are optimal, \textit{i.e.}\ the constant $1$ cannot be improved. Indeed, \textit{e.g.}\ the second inequality proposition \ref{poisson} is equivalent, thanks to Proposition \ref{prop:bg}, to
$$
\exp\left\{\int \overline{ Q}_{c_{\lambda,0}}f \,dp_\lambda\right\} \int e^{-f} \,dp_\lambda \leq 1\,,
$$
which is an equality for $f(x)=-tx$, $x\in \R$, $t \geq 0$ (the same holds for the first inequality).
\end{rem}

\section{Weak transport-entropy  and  log-Sobolev type inequalities}

In this section, our aim is to give some explicit links between the weak transport-entropy inequalities introduced in Definition \ref{def:te} and functional inequalities of log-Sobolev type. 
Except for the first result below, we are not able to deal with general costs. 
Hence (except for Section \ref{pollution}), we restrict to the specific case (already of interest) of 
$\Tb_\theta$ (introduced in Section \ref{sec:particular cases}). Furthermore, to avoid technicalities, we may restrict to the particular choice $\theta(x)=\|x\|^2$ (for some norm on $\mathbb{R}^m$), even if most of the results below could be extended to more general convex functions (at the price of denser statements and more technical proofs). As an application, using the characterization of $\overline{\T}_2^-$ by means of log-Sobolev type inequalities and results from \cite{AD05}, we may give more examples of measures satisfying such a transport-entropy inequality on the line.

\subsection{Transport-entropy and $(\tau)$-log-Sobolev inequalities} \label{pollution}

In this section, we generalize the notion of $(\tau)$-log-Sobolev inequality introduced in 
\cite{GRS11bis} (see also \cite{GRS13}) and describe some connection to
weak transport-entropy inequalities. 


First we need some notation. 
Given $\lambda >0$ and $\varphi \in \Phi_\gamma(X)$, define
\[
R_c^\lambda \varphi(x) := \inf_{p \in \mathcal{P}_\gamma(X)} \left\{ \int \varphi (y)\,p(dy) + \lambda c(x,p) \right\} , \qquad x \in X .
\]
Observe that $R_c^1=R_c$, where $R_c$ is defined in Theorem \ref{Kantorovich general}.
Following \cite{GRS11bis}, we introduce the $(\tau)$-log-Sobolev inequality as follows. We recall that for any non-negative function $g$, one denotes $\ent_\mu(g) = \int g\log\left(\frac{g}{\int g\,d\mu}\right)\,d\mu.$

\begin{defi}[$(\tau)-\mathbf{LSI}_c(\lambda,C)$] \label{def:taulogosb}
Let  $\gamma \colon \R_+\to \R_+$ be a lower-semicontinuous function satisfying \eqref{gamma},
$c : X \times \mathcal{P}_\gamma(X) \to [0,\infty)$ and $C \in (0,\infty)$.
%
Then
 $\mu \in \mathcal{P}_\gamma(X)$ is said to satisfies the $(\tau)$-log-Sobolev inequality
 with constant $C, \lambda$ and cost $c$ (or in short $(\tau)-\mathbf{LSI}_c(\lambda,C)$)
 if, for all $f$ with $\int f e^f d\mu < \infty$, it holds
 \begin{equation}\label{taulogosb}
 \ent_\mu(e^f) \leq C \int (f-R_c^\lambda f) e^f\,d\mu .
 \end{equation}
\end{defi}


The following result extends \cite[Theorem 2.1]{GRS11bis}.

\begin{prop} \label{ttau}
Let  $\gamma \colon \R_+\to \R_+$ be a lower-semicontinuous function satisfying \eqref{gamma} and
$c : X \times \mathcal{P}_\gamma(X) \to [0,\infty)$ be a cost function. 
If $\mu \in \mathcal{P}_\gamma(X)$ satisfies $\T_c^-(b)$, then it satisfies 
$(\tau)-\mathbf{LSI}_c(\lambda,\frac{1}{1-\lambda b})$ for all $\lambda \in (0,1/b)$.
\end{prop}

\begin{rem} \label{rem:HJ}
In $\mathbb{R}^n$ \cite{GRS11bis}, and more generally in metric spaces \cite{GRS13}, if one considers the usual transport cost $\mathcal{T}_2$ (with cost $\omega(x,y)=d(x,y)^2$), it is proved that the corresponding $\T_2^-(b)$ is actually equivalent to some $(\tau)$-log-Sobolev inequality. In order to get such a result in the setting of the present paper, one would need to develop a general Hamilton-Jacobi theory which is not available at present (see \cite{Shu} for some developments). This is the primary reason for us restricting ourselves to the specific case of the ``bar" cost in the next sections.
\end{rem}

\begin{proof}
Fix a function $f: X \to \mathbb{R}$ with $\int f e^f d\mu < \infty$, $\lambda \in (0,1/C)$ 
and define $d\nu_{f}=\frac{e^f}{\int e^f\,d\mu}\,d\mu$. One has 
\begin{align*}
H(\nu_{f}|\mu) =
\int \log\left( \frac{e^f}{\int e^f\,d\mu}\right)\frac{e^f}{\int e^f\,d\mu}\,d\mu  =
\int f\,d\nu_{f}-\log \int e^f\,d\mu  \leq 
\int f\,d\nu_{f}-\int f\,d\mu,
\end{align*}
where the last inequality comes from Jensen's inequality. 
Consequently, if $\pi(dxdy)=\nu_f(dx)p_x(dy)$ is a probability measure on $X\times X$ with first marginal  $\nu_f$ and second marginal $\mu$,
\begin{align*}
H(\nu_{f}|  \mu)
 \leq 
\iint (f(x)-f(y))\,\pi(dxdy)  =
\int \left(\int (f(x)-f(y))\,p_x(dy) \right) \nu_f(dx) .
\end{align*}
It follows from the definition of $R_c^\lambda$ that 
$-\int f(y)\,p_x(dy) \leq -R_c^\lambda f(x) + \lambda c(x,p_x)$ for all $x \in X$, so using that
$p_x$ is a probability measure,
\begin{align*}
\int (f(x)-f(y))\,p_x(dy) 
 = 
f(x) - \int f(y)\,p_x(dy) \leq f(x) - R_c^\lambda f(x) + \lambda c(x,p_x),
\qquad x \in X .
\end{align*}
Hence,
\begin{equation*}
H(\nu_{f}|  \mu)
\leq 
\int \left(f(x) - R_c^\lambda f(x) \right) \nu_f(dx) + \lambda \int c(x,p_x)\,\nu_f(dx) .
\end{equation*}
Optimizing over all $\pi$ (or equivalently over all $p_x$) with marginals $\nu_f$ and $\mu$, it holds
\begin {align*}
H(\nu_{f}|  \mu)
& \leq 
\int \left(f(x) - R_c^\lambda f(x) \right) \nu_f(dx) + \lambda \mathcal{T}_c(\mu|\nu_f) \\
& 
\leq \frac{1}{\int e^f\,d\mu} \int \left(f - R_c^\lambda f \right) e^fd\mu + 
\lambda b H(\nu_{f}|  \mu) .
\end{align*}
The thesis follows by noticing that $\left(\int e^f\,d\mu\right) H(\nu_{f}|  \mu) = \ent_\mu \left(e^f\right)$.
\end{proof}

\subsection{Weak transport-entropy inequalities $\mathbf{\Tb}_2^\pm$}

In this section we give different equivalent forms of $\mathbf{\overline{T}}^\pm_2$ in terms of the classical log-Sobolev-type inequality of Gross \cite{G75} \emph{restricted} to convex/concave functions, to the $(\tau)$-log-Sobolev inequality \eqref{taulogosb}
and to the hypercontractivity of the (classical) Hamilton-Jacobi semi-group, also restricted to some class of functions.

Throughout this section, we consider the cost 
\[
c(x,p)=\frac{1}{2} \left\|x - \int y\,p(dy) \right\|^2,\qquad x \in \mathbb{R}^m, p \in \mathcal{P}_1(\mathbb{R}^m),
\]  
where $\|\cdot\|$ is a norm on $\mathbb{R}^m$ whose dual norm we denote by $\| \cdot \|_*$. We recall that $\|x\|_* = \max_{y \in \R^m, \|y\|=1} x\cdot y.$
Recall the definition of $\Tb_2$ from Section \ref{sec:particular cases} and the  $(\tau)$-log-Sobolev inequality \eqref{taulogosb} defined with such a cost. As usual,  $\| f \|_p:=(\int |f|^p\,d\mu)^\frac{1}{p}$, $p \in \mathbb{R}^*$
(including negative real numbers) and $\| f \|_0:=\exp\{\int \log |f|\,d\mu\}$
whenever this makes sense. Also, given $\varphi \colon \mathbb{R}^m \to \mathbb{R}$, $t>0$, let
\begin{equation}\label{eq:qtpt}
Q_{t}\varphi(x) : = \inf_{y \in \mathbb{R}^m} \left\{ \varphi(y) + \frac{1}{2t} \left \|x - y \right\|^2 \right\} ,
\quad x \in \mathbb{R}^m ,
\end{equation}
$$
P_{t}\varphi(x) : = \sup_{y \in \mathbb{R}^m} \left\{ \varphi(y) - \frac{1}{2t} \left \|x - y \right\|^2 \right\},
\quad x \in \mathbb{R}^m.
$$
We will make use of the following observation (see Theorem \ref{Kantorovich T bar}): for any $\varphi \colon \mathbb{R}^m \to \mathbb{R}$ convex, Lipschitz and bounded from below, it holds
\[
Q_1 \varphi = R_c \varphi = \inf_{p \in \mathcal{P}_1(X)}\left\{ \int \varphi(y)\,p(dy) + c(x,p) \right\}.
\]

In the result below, we assume that $\|\,\cdot\,\|_*$ is \emph{strictly convex}, i.e. it is such that 
\begin{equation}\label{strict convexity}
 (x\neq y \text{ with } \|x\|_*=\|y\|_*=1) \Rightarrow \|(1-t)x+ty\|_*<1.
\end{equation}
This assumption is made to ensure that the operation $f\mapsto Q_tf$ transforms a convex function into a $\mathcal{C}^1$-smooth convex function (this well known property is recalled in Lemma \ref{regul} below). The proof could certainly be adapted without this assumption, but we dont want to enter into these technical complications.

\begin{rem}
It is well known that the strict convexity of the dual norm $\|\,\cdot\,\|_*$ is equivalent to the $\mathcal{C}^1$-smoothness of the initial norm $\|\,\cdot\,\|$ on $\R^m\setminus \{0\}.$ These equivalent conditions are fulfilled for instance by the classical $p$-norms : $\|x\|_p=[\sum_{i=1}^m |x_i|^p]^{1/p}$, $x\in \R^m$, for $1<p<+\infty.$
\end{rem}

\begin{thm} \label{th:bgl-}
Suppose that $\|\,\cdot\,\|_*$ is a strictly convex norm and let $\mu \in \mathcal{P}_1(\mathbb{R}^m)$. 
Then the following are equivalent:
\begin{itemize}
\item [(i)] there exists $b>0$ such that $\overline{\T}_2^-(b)$ holds;
\item[(ii)] there exists $\lambda, C >0$ such that $(\tau)-\mathbf{LSI}_c(\lambda,C)$ holds;
\item[(iii)] there exists $\rho  >0$ such that for all $\mathcal{C}^1$-smooth function $\varphi \colon \mathbb{R}^m \to \mathbb{R}$ 
convex, Lipschitz and bounded from below, it holds
\begin{equation} \label{eq:lsconvex}
\ent_\mu (e^\varphi) \leq \frac{1}{2\rho } \int \|\nabla \varphi\|_*^2 e^\varphi\,d\mu .
\end{equation}
\item[(iv)] There exists $\rho '>0$ such that for every $t>0$, every $a \geq 0$ and every $\varphi \colon \mathbb{R}^m \to \mathbb{R}$ convex, Lipschitz and bounded from below, it holds
\begin{equation} \label{eq:hyperconvex}
\| e^{Q_{t} \varphi} \|_{a + \rho ' t} \leq \| e^{\varphi} \|_{a} .
\end{equation}
\end{itemize}
Moreover 
\begin{itemize}
\item[] $(i) \Rightarrow (ii)$ for all $\lambda \in (0,1/b)$ and with $C=1/(1-b\lambda)$;
\item[] $(ii) \Rightarrow (iii)$ with $\rho  = \frac{\lambda}{C}$;
\item[] $(iii) \Rightarrow (iv)$ with $\rho '=\rho $;
\item[] $(iv) \Rightarrow (i)$ with $b=\frac{1}{\rho '}$.
\end{itemize}
\end{thm}

\begin{rem}
The implication $(ii) \Rightarrow (i)$ is a variant of a well known result due to Otto and Villani \cite{OV00} showing that the logarithmic Sobolev inequality implies the classical transport-entropy inequality $\T_2$. Here we will make use of the arguments developed in \cite{BGL01}. On the other hand, in the classical setting, the equivalence $(i) \Longleftrightarrow (ii)$ was studied and developed in 
\cite{GRS11bis,GRS13,GRS14}.

Observe that the relations between the various constants are almost optimal. Indeed, starting from $\overline{\T}_2^-(b)$, we deduce from $(ii) \Rightarrow (iii)$ that
the log-Sobolev inequality \eqref{eq:lsconvex} holds with 
$\rho =\sup_{\lambda \in (0,1/b) } \lambda /C=\sup_{\lambda \in (0,1/b)} \lambda (1-b\lambda) = \frac{1}{4b}$ (the maximum is reached at $\lambda=1/(2b)$). From this we deduce $(iv)$ with $\rho '=1/(4b)$ which gives back $\overline{\T}_2^-(4b)$, and in all we are off only by a factor $4$.

We may make use of the above result to obtain example of measures satisfying $\overline{\T}_2^-(b)$ in Section \ref{sec:T_2 on the line}. Indeed, the ``convex" log-Sobolev inequality \eqref{eq:lsconvex} was studied in the literature \cite{AD05}.
\end{rem}

We will use the following classical smoothing property of the infimum convolution operator.
\begin{lem}\label{regul}Let $\|\,\cdot\,\|$ be a norm on $\R^m$ whose dual norm is strictly convex. If $\varphi : \R^m \to \R$ is a convex function, then for all $t>0$, the function $Q_t\varphi$ defined by
\[
Q_t\varphi(x)=\inf_{y\in \R^m}\left\{ \varphi(y) + \frac{1}{2t}\|x-y\|^2\right\},\qquad x\in \R^m.
\] is also convex and $\mathcal{C}^1$-smooth on $\R^m.$
\end{lem}
\begin{proof}
The fact that $Q_t\varphi$ is convex is well-known and easy to check. Consider the Fenchel-Legendre transform of $Q_t\varphi$ defined by 
$$
(Q_t\varphi)^*(x) = \sup_{y \in \R^m}\{x\cdot y - Q_t\varphi(y)\}, \qquad x \in \mathbb{R}^m .
$$ 
A simple calculation shows that $(Q_t\varphi)^*(x) = \varphi^*(x) + \frac{1}{2}\|x\|_*^2$, for all $x\in \R^m.$ By assumption, $\|\,\cdot\,\|_*$ satisfies \eqref{strict convexity}. This easily implies that (and is in fact equivalent to) the convex function $x\mapsto \|x\|_*^2$ is strictly convex (in the usual sense : if $x\neq y$, then $\|(1-t)x+ty\|_*^2 < (1-t)\|x\|_*^2 + t \|y\|_*^2$, for all $t\in (0,1)$). Therefore, the function $x\mapsto (Q_t\varphi)^*(x)$ is strictly convex on $\R^m.$ A classical result in Fenchel-Legendre duality (see e.g. \cite[Theorem E.4.1.1]{HUL01}) then implies that $(Q_t\varphi)^{**}=Q_t\varphi$ is $\mathcal{C}^1$-smooth on $\R^m.$
\end{proof}

\begin{proof}[Proof of Theorem \ref{th:bgl-}]
That $(i)$ implies $(ii)$ is given in Proposition \ref{ttau}.

To prove that $(ii)$ implies $(iii)$, fix  $\varphi \colon \mathbb{R}^m \to \mathbb{R}$ a $\mathcal{C}^1$-smooth function which is convex, Lipschitz and bounded from below. Then, by convexity, for all $x,y \in \mathbb{R}^m$, it holds
$$
\varphi(x) - \varphi(y) \leq  \nabla \varphi(x) \cdot (x-y) \,.
$$
where $u\cdot v$ denotes the scalar product of $u,v \in \mathbb{R}^m$.
Hence, given $\lambda>0$ and $x \in \mathbb{R}^m$, by the Cauchy-Schwarz inequality $u \cdot v \leq \frac{1}{2\lambda} \|u\|_*^2 + \frac{\lambda}{2} \|v\|^2$, $u, v \in \mathbb{R}^m$, we have
\begin{align*}
\varphi(x)  - R_c^\lambda \varphi(x) & = 
\sup_{p \in \mathcal{P}_1(\mathbb{R}^m)} \left\{  \int [\varphi(x) - \varphi(y)] \,p(dy) - \frac{\lambda}{2} \|x-\int y \,p(dy) \|^2 \right\} \\
& \leq
\sup_{p \in \mathcal{P}_1(\mathbb{R}^m)} \left\{  \int \nabla \varphi(x) \cdot (x-y) \,p(dy) - \frac{\lambda}{2} \|x-\int y \,p(dy) \|^2 \right\} \\
& = 
\sup_{p \in \mathcal{P}_1(\mathbb{R}^m)} \left\{  \nabla \varphi(x) \cdot (x-\int y  \,p(dy)) - \frac{\lambda}{2} \|x-\int y \,p(dy) \|^2 \right\}\\
& \leq 
\frac{1}{2 \lambda} \| \nabla \varphi(x) \|_*^2 \,.
\end{align*}
The expected result follows.

To prove that $(iii)$ implies $(iv)$, we follow the now classical argument from \cite{BGL01} based on the Hamilton-Jacobi equation satisfied by $(t,x)\mapsto Q_t\varphi(x).$ Since we do not assume that $\mu$ is absolutely continuous with respect to Lebesgue measure (one of our main motivations is to study transport inequalities for \emph{discrete} measures), there are some technical difficulties to clarify in order to adapt the proof of \cite[Theorem 2.1]{BGL01} to our framework. First, as shown in \cite{GRS14} or \cite{AGS14}, the following Hamilton-Jacobi equation holds for \emph{all} $t>0$ and $x\in \R^m$ :
\begin{equation}\label{HJ}
\frac{d^+}{dt} Q_t\varphi (x) = - \frac{1}{2} |\nabla^-Q_t\varphi|^2(x),
\end{equation}
where, $d^+/dt$ stands for the right derivative, and by definition $|\nabla^-f|(x)$ is a notation for the \emph{local slope} of a function $f$ at a point $x$, defined by  
\[
|\nabla^-f|(x)= \limsup_{y\to x} \frac{[f(y)-f(x)]_-}{\|y-x\|}.
\] 
Here, since $\varphi$ is \emph{convex}, the regularization property of the inf-convol\-ution operator $Q_t$ given in Lemma \ref{regul} implies that for all $t>0$, the function $x\mapsto Q_tf(x)$ is actually $\mathcal{C}^1$-smooth on $\R^m.$ It is then easily checked that $|\nabla^- Q_t\varphi|(x) = \|\nabla Q_t\varphi(x)\|_*.$ Moreover, according to Lemma \ref{regul} again, if $\varphi \colon \mathbb{R}^m \to \mathbb{R}$ is convex, then so does $Q_t \varphi$. Therefore, \eqref{eq:lsconvex} can be applied to the function $Q_t\varphi$ for all $t>0.$ To complete the proof of the implication, we leave it to the reader to follow the proof of \cite[Theorem 2.1]{BGL01} (see also \cite[Theorem 1.11]{GRS14}).

Finally we prove that $(iv)$ implies $(i)$.
We observe that, at $t=1$ and $a=0$, \eqref{eq:hyperconvex} means precisely that,
$$
\int e^{\rho ' Q_1 \varphi} \,d\mu \leq e^{\rho ' \int \varphi \,d\mu} .
$$
This is equivalent to $\mathbf{\overline{T}}^-_2(1/\rho ')$, thanks to Proposition \ref{prop:bg} and to the fact that, as recalled above, $Q_1 \varphi = R_c \varphi = \inf\limits_{p \in \mathcal{P}_1(X)}\left\{ \int \varphi(y)\,p(dy) + c(x,p) \right\}$, for any 
$\varphi \colon \mathbb{R}^m \to \mathbb{R}$ convex, Lipschitz and bounded from below. This completes the proof.
\end{proof}

In order to give a series of equivalent formulations of $\overline{\T}_2^+(b)$, we need to introduce the notion of $c$-convexity
(see e.g.\ \cite{Vil09}).
We recall that if $c:X\times X$ is some cost function on a space $X$, a function $f:X\to \R\cup\{\pm \infty\}$ is said to be $c$-convex, if there exists some function $g:X \to \R\cup\{\pm \infty\}$ such that
\[
f(x) =\sup_{y\in X}\{g(y) - c(x,y)\},\qquad \forall x\in X.
\]
In what follows, we will use this notion with $c(x,y) = \frac{\lambda}{2} \|x-y\|^2$, $x,y\in \R^m$, where $\lambda>0$ and $\|\,\cdot\,\|$ is some norm on $X=\R^m$, such that $\|\,\cdot\,\|_*$ is a strictly convex norm in the sense of \eqref{strict convexity}.
In other words, a function $f \colon \mathbb{R}^m \to \R\cup\{\pm \infty\}$ is $\frac{\lambda}{2}  \|\,\cdot\,\|^2$-convex, if there exists $g \colon \mathbb{R}^m \to \R\cup\{\pm \infty\}$ such that $f=P_{\frac{1}{\lambda}}  g$ (recall the definition of $P_t$ from \ref{eq:qtpt}). In \cite[Proposition 2.2]{GRS14}, for example,  it is proved that $f$ is $\frac{\lambda}{2} \|\,\cdot\,\|^2$-convex if and only if $f = P_{\frac{1}{\lambda}} Q_{\frac{1}{\lambda}} f$. Furthermore, if $f$ is of class $\mathcal{C}^2$ and $\|\,\cdot\,\|=|\,\cdot\,|$ is the Euclidean norm, then
$f$ is $\frac{\lambda}{2}  |\,\cdot\,|^2$-convex if and only if $\mathrm{Hess} f \geq -\lambda \mathrm{Id}$ (as a matrix), where $\mathrm{Hess}$ denotes the Hessian (see e.g. \cite[Proposition 2.3]{GRS14}).

To avoid the use of too heavy a terminology, we will denote by $\mathcal{F}_\lambda(\R^m)$, $\lambda>0$, the class of all functions $f : \R^m\to \R$ that are \emph{concave}, \emph{Lipschitz}, \emph{bounded from above} and $\frac{\lambda}{2}\|\,\cdot\,\|^2$-convex.

\begin{rem}
According to Lemma \ref{regul}, if $g$ is concave on $\R^m$ and $\lambda>0$, then $Q_{1/\lambda}(-g)$ is convex and $\mathcal{C}^1$-smooth.
In particular, $f= - Q_{1/\lambda}(-g)$ is concave and $\mathcal{C}^1$-smooth. But $f= -Q_{1/\lambda}(-g) = P_{1/\lambda} (g)$ and thus $f$ is also $\frac{\lambda}{2}\|\,\cdot\,\|^2$-convex. Furthermore, if $g$ is assumed to be Lipschitz and bounded from above, then $f$ is also Lipschitz and bounded from above. This shows that the class $\mathcal{F}_\lambda (\R^m)\cap \mathcal{C}^1(\R^m)$ is not empty.
\end{rem}

\begin{thm} \label{th:bgl+}
Suppose that $\|\,\cdot\,\|_*$ is a strictly convex norm and let $\mu \in \mathcal{P}_1(\mathbb{R}^m)$. 
Then the following are equivalent:
\begin{itemize}
\item [(i)] there exists $b>0$ such that $\overline{\T}_2^+(b)$ holds;
\item[(ii)] there exist $\lambda, C >0$ such that for all $\varphi\in \mathcal{F}_\lambda(\R^m)$, it holds
\begin{equation} \label{eq:taulsob+}
\ent_\mu(e^\varphi) \leq C \int (\varphi - Q_{1/\lambda} \varphi)e^\varphi\,d\mu ;
\end{equation}
\item[(iii)] there exist $\rho , \lambda' >0$ such that for all $\mathcal{C}^1$-smooth function $\varphi\in \mathcal{F}_{\lambda'}(\R^m)$, it holds
\begin{equation} \label{eq:lsconvex2}
\ent_\mu (e^\varphi) \leq \frac{1}{2\rho } \int \|\nabla \varphi\|_*^2 e^\varphi\,d\mu .
\end{equation}
\end{itemize}
Moreover 
\begin{itemize}
\item[] $(i) \Rightarrow (ii)$ for all $\lambda \in (0,1/b)$ and with $C=1/(1-b\lambda)$;
\item[] $(ii) \Rightarrow (iii)$ for all $\lambda' \in (0, \lambda)$ and with $\rho  = \frac{\lambda- \lambda'}{C}$;
\item[] $(iii) \Rightarrow (i)$ with $b=\frac{\rho  + \lambda'}{\rho  \lambda'}$.
\end{itemize}
\end{thm}

\begin{rem}
Also, Equation \eqref{eq:taulsob+} is very close to (yet different from) the $(\tau)$-log-Sobolev inequality  \eqref{taulogosb}.
The difference is coming from the fact that, for concave functions, $R_c f \neq Q f$, while equality holds for convex functions.

In particular, we emphasize the fact that $\overline{\T}_2^-(b)$ encompasses information about convex functions, while $\overline{\T}_2^+(b)$ about concave functions. 

Finally, we observe that the constants in the various implications are almost optimal. Indeed, starting from
$\overline{\T}_2^+(b)$, we end up with $\overline{\T}_2^+(b')$, with
$b'= \frac{(\lambda- \lambda')(1-b\lambda)+ \lambda'}{\lambda'(\lambda-\lambda')(1-b\lambda)}$
with $\lambda \in (0,1/b)$ and $\lambda' \in (0,\lambda)$. Choosing $\lambda = 1/(2b)$ and $\lambda'=1/(4b)$ one gets $b'=12b$ and we are off  by a factor 12, at the most.
\end{rem}

\begin{proof}
To prove that $(i)$ implies $(ii)$, we follow the argument of the proof of Proposition \ref{ttau}. Consider a concave function $f$, Lipschitz and bounded above, $\lambda \in (0,1/b)$ 
and define for simplicity $s=1/\lambda$ and $d\nu_{f}=\frac{\exp\{P_{s} f\}}{\int \exp\{P_{s} f\}\,d\mu}\,d\mu$. By Jensen's Inequality we have
\begin{align*}
H(\nu_{f}|\mu)
& =
\int \log\left( \frac{e^{P_{s} f}}{\int e^{P_{s} f}\,d\mu}\right)\frac{e^{P_{s} f}}{\int e^{P_{s} f}\,d\mu}\,d\mu
=
\int {P_{s} f}\,d\nu_{f}-\log \int e^{P_{s} f}\,d\mu  \\
& \leq 
\int P_{s} f\,d\nu_{f}-\int {P_{s} f}\,d\mu \\
& =
\int [{P_{s} f} - f]\,d\nu_{f} -\int {P_{s} f}\,d\mu  + \int f\,d\nu_{f}  \\
& \leq 
\int [{P_{s} f} - f] \,d\nu_{f} + \lambda \Tb_2(\nu_f|\mu)
.
\end{align*}
where in the last line we used the homogeneity of the transport cost (as a function of the cost (recall that $s=1/\lambda$)) and
the duality theorem (Corollary \ref{Kantorovich T bar}) to ensure that (since $Q_1 (-\varphi) = - P_1 \varphi$)
\begin{align*}
\Tb_2(\nu_f|\mu) 
& = 
\sup\left\{ \int Q_1 \varphi\,d\mu - \int \varphi\,d\nu_f ; \right. \varphi \text{ convex}, \text{ Lipschitz}, \text{bounded from below}\Big\} \\
& =
\sup\left\{- \int P_1 \varphi\,d\mu + \int \varphi\, d\nu_f ; \right.  \varphi \text{ concave}, \text{ Lipschitz}, \text{bounded from above}\Big\} .
\end{align*}
Applying $\overline{\T}_2^+(b)$ and rearranging the terms, we end up with the following inequality (since
$\int \exp\{P_{s} f\}\,d\mu H(\nu_{f}|  \mu) = \ent_\mu \left(\exp\{P_{s} f\}\right)$):
$$
\ent_\mu \left(e^{P_{s} f}\right) 
\leq \frac{1}{1 - \lambda b} \int [{P_{s} f} - f] e^{P_{s} f}\,d\mu\,,
$$
which  holds for any $f$ concave, Lipschitz and bounded above, and for any $\lambda \in (0,1/b)$ and $s =1/\lambda$. Now, our aim is to get rid of $P_s f$. To that purpose, we observe that, since $f$ is concave,
Lipschitz and bounded above, $Q_s f$ is also concave, Lipschitz and bounded above\footnote{These facts follow from the fact that $Q_sf(x)=\inf_y \{f(x-y) + \frac{s}{2}\|y\|^2\}$. Hence $Q_s f$ is concave as infimum of concave functions. On the other hand, $x \mapsto f(x-y) + \frac{s}{2}\|y\|^2$ are uniformly (in $y$) Lipschitz functions so that $Q_s f$ is again Lipschitz as infimum of Lipschitz functions. Finally, $Q_s f \leq f$ and therefore is bounded above.}  (for any $s \geq 0$), so that, if we assume in addition that $f$ is $\frac{\lambda}{2} \|\,\cdot\,\|^2$-convex, applying the latter to $Q_sf$ and using that $P_s Q_s f=f$, we finally get the desired result of Item $(ii)$.

Now we prove that $(ii)$ implies $(iii)$. Assume Item $(ii)$ and consider
a function $f\in \mathcal{F}_{\lambda'}(\R^m)$, with $\lambda' \in (0, \lambda)$. Our aim is to make use of the $\frac{\lambda'}{2} \|\,\cdot\,\|^2$-convexity property of $f$ to bound $f - Q_{1/\lambda} f$ from above by $\| \nabla f\|_*^2$; we may follow \cite{GRS14}.

Since $f$ is $\frac{\lambda'}{2} \|\,\cdot\,\|^2$-convex, it satisfies $P_sQ_sf=f$, where for simplicity $s=1/\lambda'$ (see e.g. \cite[Proposition 2.2]{GRS14}). Define $m(x)=\left\{ \bar{y} \in \mathbb{R}^m : f(x)=g(\bar y) - \frac{\lambda'}{2}\| x-\bar{y} \|^2\right\}$, \textit{i.e.}\ the set of points where the supremum is reached, which is  is non-empty by simple compactness arguments (see \cite[Lemma 2.6]{GRS14}). Given $\bar{y} \in m(x)$, we have for all $z \in \mathbb{R}^m$\,,
\begin{equation} \label{start}
f(x) 
= Q_s f(\bar y) - \frac{\lambda'}{2}\| x-\bar{y} \|^2
\leq 
f(z) + \frac{\lambda'}{2} \left( \| z -\bar{y} \|^2- \| x-\bar{y} \|^2 \right)\,.
\end{equation}
Since $f$ is concave and $\mathcal{C}^1$-smooth, it holds
\[
f(z) \leq f(x) + \nabla f(x)\cdot(z-x),\qquad \forall z\in \R^m.
\]
Inserting this inequality in \eqref{start}, one gets
\[
0\leq \nabla f(x)\cdot (z-x) +\frac{\lambda'}{2} (\|z-\bar{y}\|^2 - \|x-\bar{y}\|^2),\qquad \forall z\in \R^m.
\]
Applying this to $z_t=(1-t)x + t\bar{y}$, with $t\in (0,1)$, one obtains
\[
0\leq t\nabla f(x)\cdot (\bar{y}-x) +\frac{\lambda'}{2} ((1-t)^2-1) \|x-\bar{y}\|^2.
\]
Dividing by $t$ and letting $t \to 0$, one ends up with the inequality
\[
\lambda' \|x-\bar{y}\|^2 \leq \nabla f(x) \cdot (\bar{y} - x) \leq \|\nabla f(x)\|_*\|x-\bar{y}\|.
\]
According to \eqref{start}, the triangle inequality,  and the inequality $\|x-\bar{y}\| \leq \frac{1}{\lambda'} \|\nabla f(x)\|_*$, one gets
\begin{align*}
f(x) &\leq f(z) + \frac{\lambda'}{2} \left( \| z -x\|^2 +2\|z-x\|\|x-\bar{y} \| \right)\\
& \leq f(z) + \frac{\lambda'}{2} \left( \| z -x\|^2 +2\|z-x\|\frac{\|\nabla f(x)\|_*}{\lambda'} \right)\\
& \leq f(z) + \frac{\lambda}{2} \| z -x\|^2 + \left(\|z-x\|\|\nabla f(x)\|_* -\frac{\lambda - \lambda'}{2} \|z-x\|^2\right)\\
& \leq f(z) + \frac{\lambda}{2} \| z -x\|^2 + \frac{1}{2(\lambda-\lambda')}\|\nabla f(x)\|_*^2.
\end{align*}
Optimizing over $z\in \R^m$, one gets the inequality
\[
f(x)-Q_{1/\lambda}f(x) \leq \frac{1}{2(\lambda-\lambda')}\|\nabla f(x)\|_*^2\,,
\]
which inserted into \eqref{eq:taulsob+} yields  \eqref{eq:lsconvex2}.

It remains to prove that $(iii)$ implies $(i)$. To that purpose, let $\ell(t):=-\rho (1-t)$, $t \in (0,1)$ (observe that $\ell(t) \leq 0$), set $s=-\ell(t)/\lambda'$, and consider a convex, Lipschitz and bounded below function $f \colon \mathbb{R}^m \to \mathbb{R}$. We shall apply the log-Sobolev inequality to $\varphi = \ell(t) Q_t f$ for a given $t \in (0,1)$. We need first to verify that $\varphi$ is concave, Lipschitz, bounded above and $\lambda'c$-convex. Since $f$ is convex, $Q_t f$ is convex and so, since $\ell(t) \leq 0$, $\varphi$ is concave. On the other hand, since $f$ is Lipschitz, so does $\varphi$. Also, $f$ being bounded below, $Q_t f \geq \inf f$ and $\ell(t) \leq 0$, we have  $\varphi = \ell(t) Q_t f \leq \ell(t) \inf f$ which proves that $\varphi$ is bounded above. Finally,
since $Q_t$ is a semi-group and since in general $Q_u (g) = - P_u (-g)$, we have for all $t \in (\frac{\rho }{\rho  + \lambda'},1)$ (to ensure that $s \leq t$)\,,
\begin{align*}
\varphi 
& = 
\ell(t) Q_s (Q_{t-s} f) = -\ell(t) P_s ( - Q_{t-s} f) = P_{\frac{s}{-\ell(t)}} (\ell(t) Q_{t-s} f) \\
& =
P_{\frac{1}{\lambda'}} ( \ell(t) Q_{t-s} f) ,
\end{align*}
hence $\varphi$ is $\lambda' c$-convex. In turn, applying the log-Sobolev inequality to $\varphi$ (which is $\mathcal{C}^1$-smooth according to Lemma \ref{regul}), 
we end up with the following inequality that we shall use later on:
\begin{align*}
\int \ell(t) Q_t f e^{\ell(t) Q_t f} d\mu - H(t) \log H(t) 
& = 
\ent_\mu(e^{\ell(t) Q_t f})\leq 
\frac{\ell(t)^2}{2 \rho } \int \|\nabla Q_t f \|^2_* e^{\ell(t) Q_t f} d\mu\,,
\end{align*}
where $H(t):=\int e^{\ell(t) Q_t f} d\mu$. 
Hence, by the Hamilton-Jacobi equation \eqref{HJ},
\begin{align*}
\frac{d^+}{dt} & \left( \frac{1}{\ell(t)}  \log H(t) \right) 
 =
\frac{1}{\ell(t)^2 H(t)}\left( - \ell'(t)H(t) \log H(t) + \ell(t) H'(t) \right) \\
& =
\frac{1}{\ell(t)^2 H(t)}\left( \ell'(t)\ent_\mu(e^{\ell(t) Q_t f}) + \ell(t)^2 \int \frac{\partial Q_tf}{\partial t} e^{\ell(t) Q_t f}\, d\mu \right) \\
& =
\frac{\ell'(t)}{\ell(t)^2 H(t)}\left( \ent_\mu(e^{\ell(t) Q_t f}) + \frac{\ell(t)^2}{2 \ell'(t)} \int \|\nabla Q_tf \|_*^2 e^{\ell(t) Q_t f}\,d\mu \right) \\
& \leq
\frac{\ell'(t)}{2 H(t)}  \left( \frac{1}{\rho } - \frac{1}{\ell'(t)} \right) \int \|\nabla Q_tf \|_*^2 e^{\ell(t) Q_t f}\, d\mu = 0\,,
\end{align*}
since $\ell'(t)=\rho $. Therefore the function $t  \mapsto \| e^{Q_tf}\|_{\ell(t)}$
is non-increasing on $(\frac{\rho }{\rho  + \lambda'},1)$. In particular, in the limit, we get
$$
\| e^{Q_1f}\|_{\ell(1)} \leq \left\| e^{Q_\frac{\rho }{\rho  + \lambda'}f}\right\|_{\ell(\frac{\rho }{\rho  + \lambda'})}
$$
that we can rephrase as
$$
e^{\int Q_1 f d\mu} \left( \int e^{-\frac{\rho  \lambda'}{\rho  + \lambda'} Q_\frac{\rho }{\rho  + \lambda'}f} d\mu \right)^\frac{\rho  + \lambda'}{\rho  \lambda'} \leq 1 .
$$
Now, since $Q_uf \leq f$, we conclude that
$$
e^{\int Q_1 f d\mu} \left( \int e^{-\frac{\rho  \lambda'}{\rho  + \lambda'} f} d\mu \right)^\frac{\rho  + \lambda'}{\rho  \lambda'} \leq 1\,,
$$
which implies $\overline{\T}_2^+(\frac{\rho  + \lambda'}{\rho  \lambda'})$ by Proposition \ref{prop:bg} and
Corollary \ref{Kantorovich T bar}. This completes the proof.
\end{proof}


\subsection{Sufficient condition for $\overline{\T}^-_2$ on the line}\label{sec:T_2 on the line}

In this short section, we would like to take advantage of some known results from \cite{AD05} to give a sufficient condition for the transport-entropy inequality $\overline \T^-_2$ to hold on the line.

Our starting point is the following result.

\begin{thm}[\cite{AD05}] \label{adamczak}
Let $\mu$ be a symmetric probability measure on the line. Assume that there exists
$c>0$ and $\alpha <1$ such that for all $x \geq 0$, $\mu([x+\frac{c}{x},\infty)) \leq \alpha \mu([x,\infty))$. Then, there exists $C(c,\alpha) \in (0,\infty)$ such that for every smooth, convex function $\varphi \colon \mathbb{R} \to \mathbb{R}$, it holds
$$
\ent_\mu(e^\varphi) \leq C(c,\alpha) \int {\varphi'}^2 e^\varphi d\mu .
$$
\end{thm}

Observe that we assumed symmetry for simplicity. It is not essential and a similar result holds for non-symmetric measures.

\begin{cor}
Let $\mu$ be a symmetric probability measure on the line. Assume that there exists
$c>0$ and $\alpha <1$ such that for all $x \geq 0$, $\mu([x+\frac{c}{x},\infty)) \leq \alpha \mu([x,\infty))$. Then, there exists $C=C(c,\alpha) \in (0,\infty)$ such that $\overline \T^-_2(C)$ holds.
\end{cor}

\begin{proof}
Theorem \ref{adamczak} guarantees that Item $(iii)$ of Theorem \ref{th:bgl-} holds, with $1/(2\rho )=C(c,\alpha)$ (Choose $\| \cdot \| = | \cdot |$, where $| \cdot |$ is the absolute value, so that
$\| \cdot \|_* = | \cdot |$). The desired result follows from Theorem \ref{th:bgl-}.
\end{proof}

We refer to \cite{GRSST15} for a complete characterization of the inequalities $\overline{\T}^\pm_2$ (and other $\overline{\T}$ inequalities) on the line. As proved there in \cite[Theorem 1.2]{GRSST15}, a probability measure $\mu$ satisfies $\overline{\T}^{-}(C)$ and $\overline{\T}^{+}(C)$ for some $C$ is and only if there exists some $D>0$ such that the monotone increasing rearrangement map $U$ transporting the symmetric exponential probability measure $\nu(dx)=\frac{1}{2} e^{-|x|}\,dx$ on $\mu$ satisfies the following growth condition:
\[
\sup_{x} |U(x+u)-U(x)| \leq \frac{1}{D}\sqrt{u+1},\qquad \forall u>0.
\]
See \cite{GRSST15} for an explicit relation between the constants $C$ and $D$, and also for a more general statements.



\section{Generalization of Kantorovich duality}\label{sec:duality}

\subsection{Notations}\label{Notation Kantorovich} First let us recall and complete the notations introduced in Section \ref{Notation}.  
Throughout this section, $(X,d)$ is a complete separable metric space. The space of all Borel probability measures on $X$ is denoted by $\mathcal{P}(X)$ and the space of all Borel signed measures by $\mathcal{M}(X)$.

If $\gamma:\R_+ \to \R_+$ is a lower-semicontinuous function satisfying \eqref{gamma}, we set
\[
\mathcal{M}_\gamma(X):= \left\{\mu \in \mathcal{M}(X) ; \int \gamma(d(x,x_o))\,|\mu|(dx)<\infty\right\}
\] 
for some (hence all) $x_o \in X$. 

We equip $\mathcal{M}_\gamma(X)$ with the coarsest topology that makes continuous the linear functionals $\mu \mapsto \int \varphi\,d\mu$, $\varphi \in \Phi_\gamma(X)$, where we recall that $\Phi_\gamma(X)$ denotes the set of continuous functions $\varphi:X\to \R$ satisfying the growth condition \eqref{eq:growth condition}.
This topology is denoted by $\sigma(\mathcal{M}_\gamma(X))$. To be more specific, a basis for this topology is given by all finite intersections of sets of the form
\begin{equation}\label{eq:ouvert}
U_{\varphi, a, \varepsilon}:=\left\{m\in \mathcal{M}_\gamma(X); \left|\int \varphi\,dm -a\right|<\varepsilon\right\}, \; \varphi\in \Phi_\gamma(X), \, a \in \R, \, \varepsilon>0.
\end{equation}

The set $\mathcal{P}_\gamma(X):= \mathcal{P}(X)\cap \mathcal{M}_\gamma(X)$ is equipped with the trace topology denoted by $\sigma(\mathcal{P}_\gamma(X))$. Let us remark that if $\gamma$ is bounded, then $\mathcal{P}_\gamma(X) = \mathcal{P}(X)$ and the topology $\sigma(\mathcal{P}_\gamma(X))$ is the usual weak topology on $\mathcal{P}(X)$.  
 
We define similarly the spaces $\mathcal{P}_\gamma(X\times X)\subset\mathcal{M}_\gamma(X\times X)$ and equip them with the topologies $\sigma(\mathcal{M}_\gamma(X \times X))$ and $\sigma(\mathcal{P}_\gamma(X\times X))$ defined with the class $\Phi_\gamma(X\times X)$ of continuous functions $\varphi:X \times X \to \R$ such that there exist $a,b\geq 0$ and $x_o\in  X$ such that $|\varphi(x,y)|\leq a+b(\gamma(d(x_o,x))+\gamma(d(x_o,y)))$ for all $x,y\in X$. 

Finally, we recall that $\Phi_{\gamma,b}(X)$ is the set of the elements of $\Phi_\gamma(X)$ that are bounded from below.

Before stating our main result, we need to introduce some technical assumptions and comment on them.
Below we denote by $\pi(dxdy) = p_x(dy)\pi_1(dx)$ the disintegration of a probability measure $\pi$ on $X\times X$ with respect to its first marginal $\pi_1$.
\begin{defi}[Conditions $(C)$, $(C')$, $(C'')$]
Given $(X,d)$ a complete separable metric space  and 
 $c \colon X\times \mathcal{P}_\gamma(X)\to [0,\infty]$  a cost function associated to some lower-semicontinuous function $\gamma:\R_+\to \R_+$ satisfying \eqref{gamma}, we say the condition $(C)$ holds 
 if
 \begin{itemize}
\item[$(C_1)$] For all $\mu \in \mathcal{P}_\gamma(X)$, the function $\pi \mapsto I_c[\pi]:=\int c(x,p_x)\,\pi_1(dx)$ is lower-semicontinuous on the set 
\[
\Pi(\mu,\,\cdot\,) := \{ \pi \in \mathcal{P}_\gamma(X\times X); \pi(dx\times X) = \mu(dx)\}.
\] 
In other words, for all $s\geq0$, the set $\{\pi \in \Pi(\mu,\,\cdot\,); I_c[\pi]\leq s\}$ is closed for the topology $\sigma(\mathcal{P}_\gamma(X\times X))$.
\item[$(C_2)$] The function $p\mapsto c(x,p)$ is convex for all $x\in X$.
\item[$(C_3)$]  The function $(x,p) \mapsto c(x,p)$ is continuous with respect to the product topology. 
\item[$(C_4)$] The cost $c$ is such that if $\mu \in \mathcal{P}_\gamma(X)$ and $(p_x)_{x\in X}$ are measurable probability kernels such that $p_x \in \mathcal{P}_\gamma(X)$ for all $x\in X$ and $\int c(x,p_x)\,\mu(dx) <\infty$, then $\nu = \mu p \in \mathcal{P}_\gamma(X)$.
\end{itemize}
Similarly we say that condition $(C')$ holds if $(C_1), (C_2)$, $(C_4)$ hold together with 
\begin{itemize}
\item[$(C_3')$] $(X,d)$ is compact and the function $(x,p) \mapsto c(x,p)$ is   lower-semicontinuous with respect to the product topology,  
\end{itemize}
and that condition $(C'')$ holds if $(C_2)$, $(C_4)$ hold together with 
\begin{itemize}
\item[$(C_3'')$]  $X$ is a countable set of isolated points and for all $x\in X$, the function $p \mapsto c(x,p)$ is lower-semicontinuous.
\end{itemize}
\end{defi}

The above conditions are technical. However, Condition $(C_2)$ is the least we can hope for. 

As for applications, the main difficulty is coming from Condition $(C_1)$. 
Let us make some comments about this assumption. First specializing to $\mu=\delta_x$, condition $(C_1)$ implies that for all $x\in X$, the function $p\mapsto c(x,p)$ is lower semicontinuous on $\mathcal{P}_\gamma(X)$. In the discrete setting, the converse is also true : as shown in the proof of Theorem \ref{Kantorovich general}, Condition $(C_3'')$ implies Condition $(C_1)$ (this is why the latter does not appear in Condition $(C'')$). For more general spaces, we do not know if Condition $(C_1)$ is strictly stronger than lower-semicontinuity of the cost function $c$. Nevertheless, we have the following rather general abstract result whose proof is postponed to Section \ref{sec:propC1}. In particular, such a result applies to the transport costs $\Tw$, $\Tb$ and $\Th$ introduced in Section \ref{sec:particular cases}.

\begin{prop}\label{prop:C1} 
Let $(X,d)$ be complete separable metric space.  
Let  $(\varphi_k)_{k\in \N}$ be a sequence of elements of $\Phi_\gamma(X\times X)$ {\rm (}with $\gamma:\R_+\to \R_+$ satisfying \eqref{gamma}{\rm )} such that $\varphi_0 \equiv 0$. Assume that the cost function $c:X\times \mathcal{P}_\gamma(X)\to[0,\infty]$ is defined by 
\begin{equation}\label{coco}
c(x,p) = \sup_{k\in \N} \int \varphi_k(x,y)\,p(dy),\qquad \forall x\in X,\qquad \forall p\in \mathcal{P}_\gamma(X) .
\end{equation}
Then Conditions $(C_1)$ and $(C_2)$ hold and $c:X \times \mathcal{P}_\gamma(X)\to [0,\infty]$ is lower-semicontinuous with respect to the product topology.
\end{prop}

We are now in a position to state the main result of this section: a generalization of the Kantorovich duality theorem.

\begin{thm}\label{Kantorovich general} 
Let $(X,d)$ be a complete separable metric space. Let 
 $c \colon X\times \mathcal{P}_\gamma(X)\to [0,\infty]$  be a cost function associated to some lower-semicontinuous function $\gamma \colon \R_+\to \R_+$ satisfying \eqref{gamma}. Assume that 
 condition $(C), (C')$ or $(C'')$ holds. 
Then, for all $\mu,\nu \in \mathcal{P}_\gamma(X)$, the following duality formula holds:
\[
\mathcal{T}_c(\nu|\mu)=\sup_{\varphi \in \Phi_{\gamma,b}(X)}\left\{ \int R_c\varphi(x)\,\mu(dx)-\int \varphi(y)\,\nu(dy)\right\},
\]
where 
\[
R_c\varphi(x):=\inf_{p\in \mathcal{P}_\gamma(X)}\left\{ \int \varphi(y)\,p(dy) + c(x,p)\right\},\quad  x\in X,\quad  \varphi\in \Phi_{\gamma,b}(X).
\]
%
\end{thm}
\begin{rem}
Note that since $c\geq 0$, $R_c\varphi$ is bounded from below as soon as $\varphi$ is bounded from below. Therefore, $\int R_c\varphi(x)\,\mu(dx)$ is always well defined in $(-\infty,\infty].$ Note also, that $R_c\varphi$ is always measurable. This is clear under Condition $(C_3)$, since in this case $R_c\varphi$ is lower-semicontinuous as an infimum of continuous functions. Under Condition $(C_3')$, it is not difficult to check that $R_c\varphi$ remains lower-semicontinuous, using the fact that $\mathcal{P}_\gamma(X)$ is compact.
\end{rem}
%
The proof of Theorem \ref{Kantorovich general} uses classical tools from convex analysis that we recall in a separate subsection (see Section~\ref{sec:duality_proof} below), and then apply them to our specific setting. We refer to Mikami \cite{Mik06}, L\'eonard \cite{Leo11}, Tan-Touzi \cite{TT13} for similar strategies.

\subsection{Fenchel-Legendre duality}
The main tool used in the proof of Theorem \ref{Kantorovich general} is the following Fenchel-Legendre duality theorem (see for instance \cite[Theorem 2.3.3]{Zal02}).
\begin{thm}[Fenchel-Legendre duality theorem]\label{Fenchel-Legendre}
Let $E$ be a Hausdorff locally convex topological vector space and $E'$ its topological dual space. For any lower semicontinuous convex function $F \colon E\to ]-\infty,\infty]$, it holds
\[
F(x)=\sup_{\ell \in E'}\{ \ell(x) - F^*(\ell)\},\qquad  x\in E,
\]
where the Fenchel-Legendre transform $F^*$ of $F$ is defined by 
\[
F^*(\ell)=\sup_{x\in E}\{\ell(x) -F(x)\},\qquad  \ell \in E'.
\]
\end{thm}

To apply Theorem \ref{Fenchel-Legendre} in our framework, one needs to identify the topological dual space of $\left(\mathcal{M}_\gamma(X), \sigma(\mathcal{M}_\gamma(X))\right)$ equipped with the top\-ol\-ogy defined in Section \ref{Notation Kantorovich}. More precisely, the next lemma will enable us to identify the dual space $\left(\mathcal{M}_\gamma(X), \sigma(\mathcal{M}_\gamma(X))\right)'$ to the set $\Phi_\gamma(X)$.

\begin{lem}\label{lem:identification}
A linear form $\ell \colon  \mathcal{M}_\gamma(X)\to \R$ is continuous with respect to the topology $\sigma(\mathcal{M}_\gamma(X))$ if and only if there exists $\varphi \in \Phi_\gamma(X)$ such that 
\[
\ell(m) = \int \varphi\,dm,\qquad \forall m\in \mathcal{M}_\gamma(X).
\]
\end{lem}

The proof of this lemma appears, for instance, in the book by Deus\-chel and Stroock \cite{DS89}. We recall it here for the sake of completeness.
\proof[Proof of Lemma \ref{lem:identification}]
The fact that linear functionals of the form $m\mapsto \int \varphi\,dm$, $\varphi \in \Phi_\gamma$ are continuous comes from the very definition of the topology $\sigma(\mathcal{M}_\gamma(X)).$ Conversely, let $\ell$ be a continuous linear functional and let us show that $\ell$ is of the preceding form. Define $\varphi(x)=\ell(\delta_x)$, $x\in X$ (where $\delta_x$ is the Dirac mass at $x$). First we will show that $\varphi$ belongs to $\Phi_\gamma(X)$. The map $X \ni x \mapsto \delta_x \in \mathcal{M}_\gamma(X)$ is  continuous. Namely, for all $\varphi_1,\ldots,\varphi_n \in \Phi_\gamma$, it holds 
$
\{x \in X ; \delta_x \in \cap_{i=1}^nU_{\varphi_i,a_i,\varepsilon_i}\}=\{x\in X ; |\varphi_i(x)-a_i|<\varepsilon_i, \forall i \leq n\}$,
(where $U_{\varphi_i,a_i,\varepsilon_i}$ is defined by \eqref{eq:ouvert}) and this set is open, which proves that $x\mapsto \delta_x$ is continuous on $X$. As a result $\varphi$ is continuous. It remains to  prove that $\varphi$ satisfies the growth condition \eqref{eq:growth condition}. Since $\ell$ is continuous, the set $O:=\{ m\in \mathcal{M}_\gamma(X);|\ell(m)|<1\}$ is open and contains $0$. By the definition of the topology  $\sigma(\mathcal{M}_\gamma(X))$, there exist an integer $n$, $\varphi_1,\ldots,\varphi_n \in \Phi_\gamma$, $a_1,\ldots,a_n \in \mathbb{R}$ and $\varepsilon_1,\ldots,\varepsilon_n >0$ such that $O$ contains $\cap_{i=1}^nU_{\varphi_i,a_i,\varepsilon_i}$ and $0\in \cap_{i=1}^nU_{\varphi_i,a_i,\varepsilon_i}$. As a result,
\[0\in \bigcap_{i=1}^nU_{\varphi_i,a_i,\varepsilon_i}\qquad \Rightarrow \qquad A:=\max_{i\in \{1,\ldots,n\}}\left|\frac{a_i}{\varepsilon_i} \right|<1,\]
and (given $m \in \mathcal{M}_\gamma(X)$)
\[
\sum_{i=1}^n \left|\int \frac{\varphi_i}{\varepsilon_i}\,dm\right|<1-A\qquad \Rightarrow \qquad m \in O .
\]
Thus, since $m/\ell(m)\notin O$,
\[
|\ell(m)| \leq \frac1{1-A} \sum_{i=1}^n \left|\int \frac{\varphi_i}{\varepsilon_i}\,dm\right|,\qquad \forall m\in \mathcal{M}_\gamma(X).
\]
Applying this inequality to $m=\delta_x$ and using the growth conditions \eqref{eq:growth condition} satisfied by the $\varphi_i's$, one sees that $\varphi$ verifies \eqref{eq:growth condition}.

Finally, let us show that $\ell(m) = \int \varphi\,dm$, for all $m\in \mathcal{M}_\gamma(X)$. If $m$ is a linear combination of Dirac measures, then this identity is clearly satisfied. Since any measure $m$ can be approached in the topology $\sigma(\mathcal{M}_\gamma(X))$ by some sequence $m_n$ of measures with finite support, the equality $\ell(m) = \int \varphi\,dm$ extends to any $m \in \mathcal{M}_\gamma(X).$  
\endproof

During the proof of Theorem \ref{Kantorovich general}, we will also use the following easy extension of Prokhorov's theorem.
\begin{thm}\label{thm:Prokhorov}
A set $A \subset \mathcal{P}_\gamma(X)$ is relatively compact for the topology $\sigma(\mathcal{M}_\gamma(X))$ if and only if for all $\varepsilon>0$, there exists a compact set $K_\varepsilon \subset X$ such that 
\[
\int_{X\setminus K_\varepsilon} (1+\gamma(d(x_o,x)))\,\nu(dx) \leq \varepsilon,\qquad \forall \nu \in A,
\]
where $x_o$ is some arbitrary fixed point.
\end{thm}

\subsection{Proof of Theorem \ref{Kantorovich general} (Duality)}\label{sec:duality_proof}
\proof[Proof of Theorem \ref{Kantorovich general}]
Fix $\mu\in \mathcal{P}_\gamma(X)$ and let us consider the function $F$ defined on $\mathcal{M}_\gamma(X)$ by 
\[
F(m)=\mathcal{T}_c(m|\mu), \text{ if } m\in \mathcal{P}_\gamma(X) \qquad\text{and}\qquad F(m)=+\infty \text{ otherwise}.
\]

Let us show that the function $F$ satisfies the assumptions of Theorem \ref{Fenchel-Legendre}.


First we will prove that $F$ is convex on $\mathcal{M}_\gamma(X)$. According to the definition of $F$, it is clearly enough to prove the convexity of $F$ over (the convex set) $\mathcal{P}_\gamma(X)$. Take $\nu_0,\nu_1 \in \mathcal{P}_\gamma(X)$  and $\pi_i \in \Pi(\mu,\nu_i)$ $i=0,1$ with disintegration kernels $(p^0_x)_{x\in X},(p^1_x)_{x\in X}$.  Then for all $t\in [0,1]$, $\pi_t:=(1-t)\pi_0 + t\pi_1 \in \Pi(\mu,(1-t)\nu_0+t\nu_1)$ and its disintegration kernel satisfies $p^t_x=(1-t)p^0_x+tp^1_x$, for $\mu$ almost every $x\in X.$ Since the cost function $c$ is convex in its second argument, it holds
\begin{align*}
F((1-t)\nu_0+t\nu_1)\leq I_c[\pi_t] & = \int c(x,p^t_x)\,\mu(dx)\leq (1-t)I_c[\pi_0] + tI_c[\pi_1]\,.
\end{align*}
Optimizing over $\pi_0,\pi_1$ gives $F((1-t)\nu_0+t\nu_1)\leq(1-t)F(\nu_0)+tF(\nu_1)$, which proves the desired convexity property.

Next we will prove that  $F$ is lower-semicontinuous, for the topology $\sigma(\mathcal{M}_\gamma(X))$, on $\mathcal{M}_\gamma(X)$. Let $(m_n)_n$  be a sequence of $\mathcal{M}_\gamma(X)$ converging to some $m$. One needs to show that $F(m)\leq \liminf_{n\to \infty} F(m_n)$. One can assume without loss of generality that $F(m_n)<\infty$ for all $n$. By the definition of $\mathcal{T}_c(\,\cdot\,|\mu)$, for all $n\in \N^*$, there exists $\pi_n \in \Pi(\mu,m_n)$ such that $I_c[\pi_n]-1/n\leq \mathcal{T}_c(m_n|\mu) \leq I_c[\pi_n].$ Since $m_n$ is a converging sequence, the set $\{m_n ; n\in \N^*\}\cup\{\mu\}$ is relatively compact. Therefore, according to Theorem \ref{thm:Prokhorov}, for some arbitrary fixed point $x_o\in X$,  for all $\varepsilon>0$, there exists a compact set $K_\varepsilon \subset X$ such that 
\[
\sup _{n\in \N^*}\int_{X\setminus K_\varepsilon} 1+ \gamma(d(x_o,y))\,m_n(dy) \leq \varepsilon
\]
and
\[
\int_{X\setminus K_\varepsilon} 1+ \gamma(d(x_o,x))\mu(dx) \leq \varepsilon.
\]
Therefore, letting $M:=\sup_{n\in \N^*} \int \gamma(d(x_o,x))\,m_n(dx)<\infty$ and $K_\varepsilon^c:=X\setminus K_\varepsilon$,  it holds
\begin{align*}
&\int_{X\times X \setminus  (K_\varepsilon \times K_\varepsilon)}   
1+\gamma(d(x_o,x))+\gamma(d(x_o,y))\,\pi_n(dxdy)\\ 
& \leq 
\int_{X\times K_\varepsilon^c} 1+\gamma(d(x_o,x))+\gamma(d(x_o,y))\,\pi_n(dxdy)  + 
\int_{K_\varepsilon^c\times X} 1+\gamma(d(x_o,x))+\gamma(d(x_o,y))\,\pi_n(dxdy)\\ 
& \leq 
m_n(K_\varepsilon^c)\int \gamma(d(x_o,x))\,\mu(dx) + \int_{K_\varepsilon^c} 1+\gamma(d(x_o,y))\,m_n(dy) + 
\int_{K_\varepsilon^c} 1+\gamma(d(x_o,x))\,\mu(dx)+  \mu(K_\varepsilon^c) M \\
&  \leq \varepsilon \left(2 + M + \int \gamma(d(x_o,x))\,\mu(dx)\right).
\end{align*}
So according to Theorem \ref{thm:Prokhorov}, it follows that $\{\pi_n; n\in \N^*\}$ is rel\-atively compact. 
Extracting a subsequence if necessary, one can assume without loss of generality that $\pi_n$ converges to some $\pi^*\in \mathcal{P}_\gamma(X\times X)$. This $\pi^*$ has the correct marginals $\mu$ and $m$. Furthermore, denoting by $\ell = \liminf_{n\to \infty} I_c[\pi_n]= \liminf_{n\to \infty} \mathcal{T}_c(m_n|\mu)$, we see that, for all $r>0$, 
\[
\pi_n \in \{\pi \in \mathcal{P}_\gamma(X\times X) ; \pi(dx\times X)=\mu(dx) \text{ and }I_c[\pi] \leq \ell + r\}:= A_{\ell+r}\,,
\] 
for infinitely many $n\in \N^*.$ By assumption $(C_1)$, the set $A_{\ell+r}$ is closed for the topology $\sigma(\mathcal{P}_\gamma(X\times X))$. Therefore, the limit $\pi^*$ also belongs to $A_{\ell + r}$. In other words, 
\[
F(m)=\mathcal{T}_c(m|\mu)\leq I_c[\pi^*] \leq \liminf_{n\to \infty} \mathcal{T}_c(m_n|\mu)+ r, \qquad \forall r>0.
\]
Since $r>0$ is arbitrary, this concludes the proof of the lower-semi\-conti\-nuity of $F.$

According to Lemma \ref{lem:identification}, the topological dual space of $\mathcal{M}_\gamma(X)$ can be identified with the set of linear functionals $m \mapsto \int \varphi\,dm$, where $\varphi\in \Phi_\gamma(X).$ 
Applying Theorem \ref{Fenchel-Legendre} together with Lemma \ref{lem:identification} we conclude that,
for any $m \in \mathcal{P}_\gamma(X)$,
\[
F(m)=\sup_{\varphi \in \Phi_\gamma(X)} \left\{ \int \varphi\,dm - F^*(\varphi)\right\} =\sup_{\varphi \in \Phi_\gamma(X)} \left\{ \int -\varphi\,dm - F^*(-\varphi)\right\} .
\]
Now we show that the last supremum can be restricted to $\Phi_{\gamma,b}(X)$. 
Observe that 
\begin{align*}
F^*(-\varphi) &= \sup_{m \in \mathcal{P}_\gamma(\X)}\left\{ \int -\varphi\,dm - F(m)\right\} \\&= \sup_{k\in \R} \sup_{m \in \mathcal{P}_\gamma(X)} \left\{ \int - (\varphi\vee k)\,dm - F(m) \right\} = \sup_{k\in \R} F^*(-(\varphi \vee k))\,,
\end{align*}
so that for all $\varphi \in \Phi_\gamma(X)$ and $m\in \mathcal{P}_\gamma(X)$\,, we have
\[
\int -\varphi\,dm - F^*(-\varphi) = \lim_{k\to -\infty} \int - (\varphi \vee k)\,dm- F^*(- (\varphi \vee k)) .
\]
Therefore,
\begin{align*}
F(m) =
\sup_{\varphi \in \Phi_{\gamma}(X)} \left\{  \int -\varphi\,dm - F^*(-\varphi)  \right\} \leq 
\sup_{\varphi \in \Phi_{\gamma,b}(X)} \left\{  \int -\varphi\,dm - F^*(-\varphi)  \right\}\,,
\end{align*}
and since the other inequality is obvious, the two quantities are equal.
%
To conclude the proof, it remains to show that
\begin{equation}\label{F^*}
F^*(-\varphi)  = -\int R_c\varphi(x)\,\mu(dx),\qquad \forall \varphi \in \Phi_{\gamma,b}(X).
\end{equation}
For all $\varphi \in \Phi_{\gamma,b}$, it holds
\begin{align*}
F^*(-\varphi)
&= 
\sup_{m \in \mathcal{P}_\gamma(\X)}\left\{ \int -\varphi\,dm - \mathcal{T}_c(m|\mu)\right\} \\
& = 
\sup_{m \in \mathcal{P}_\gamma(\X)}\sup_{\pi \in \Pi(\mu,m)}\left\{ \int -\varphi\,dm - I_c[\pi]\right\} \\
&= 
\sup\left\{ \int \left[\int-\varphi(y)\,p_x(dy) - c(x,p_x)\right]\,\mu(dx); \right. (p_x)_{x\in X} 
\text{ probability kernel such that } \mu p \in \mathcal{P}_\gamma(X)\bigg\}\\
& 
= -\inf \left\{ \int \left[\int\varphi(y)\,p_x(dy) + c(x,p_x)\right]\,\mu(dx); \right.(p_x)_{x\in X} 
\text{ probability kernel such that } \mu p \in \mathcal{P}_\gamma(X)\bigg\}.
\end{align*}
By definition, $R_c\varphi(x) = \inf_{p \in \mathcal{P}_\gamma(X)} \{\int \varphi\,dp + c(x,p)\}$. Therefore, one has 
\[
F^*(-\varphi) \leq - \int R_c\varphi(x)\,\mu(dx).
\]
Let us show the converse inequality. One can assume without loss of generality that $\int R_c\varphi(x)\,\mu(dx) \in (-\infty,\infty).$ For all $\varepsilon>0$ and $x \in X$, consider the set $M_x^\varepsilon$ defined by 
\[
M_x^\varepsilon:=\left\{p \in \mathcal{P}_\gamma(X); \int \varphi\,dp + c(x,p) \leq R_c\varphi(x)+\varepsilon\right\}.
\]
Note that, since $\varphi$ is bounded from below and $c\geq0$, $R_c\varphi(x) >-\infty$, for all $x\in X$, we have that $M_x^\varepsilon$ is non-empty for all $\varepsilon>0.$

Assume that for all $\varepsilon>0$, there exists a \emph{measurable} kernel $X \to \mathcal{P}_\gamma(X) : x \mapsto p_x^{\varepsilon}$ such that for all $x \in X$, $p_x^\varepsilon \in M_x^\varepsilon$.
Then, if $\varphi$ is bounded below by $k$, one sees that $\int c(x,p_x^\varepsilon)\,\mu(dx) \leq -k+\varepsilon + \int R_c\varphi\,d\mu<\infty.$ According to condition $(C_4)$ one concludes that $\nu^{\varepsilon} = \mu p^{\varepsilon} \in \mathcal{P}_\gamma(X).$
So it holds
\[
F^*(-\varphi) \geq -\int \int \varphi(y)\,p_x^{\varepsilon}(dy) + c(x,p_x^\varepsilon)\,\mu(dx)\geq  - \int R_c\varphi(x)\,\mu(dx)-\varepsilon\,,
\]
which gives the desired inequality when $\varepsilon \to 0$.

When the condition $(C_3)$ holds, the kernel $p_x^\varepsilon$ is obtained by applying the elementary measurable selection result of Lemma \ref{lem:measurable sel} below. Indeed, note that the function $H(x,p) = \int \varphi\,dp + c(x,p)$ is continuous (and thus upper-semicontinuous), and that $Y=\mathcal{P}_\gamma(X)$ equipped with the topology $\sigma(\mathcal{P}_\gamma(X))$ is metrizable {(for instance,  by the Kantorovich metric $W_r$ if  $\gamma= \gamma_r$, or the L\'evy-Prokhorov distance for the usual weak-topology if  $\gamma=\gamma_0$)}  and separable (see \cite[Theorem 6.18]{Vil09}, \cite[Proposition 7.20]{BS78}).

Under condition $(C_3')$, the space $X$ is compact and the function $H$ defined above is lower-semicontinuous. The  selection  Lemma \ref{lem:measurable selbis} below 
ensures that there exists a \emph{measurable} kernel $X \to \mathcal{P}_\gamma(X) : x \mapsto p_x$
such that $R_c\varphi(x)=\inf_{p\in \mathcal{P}_\gamma(X)} H(x,p)=H(x,p_x)$. The conclusion easily follows.

Under condition $(C_3'')$,  $X$ is a  countable set of isolated points. So all subsets of $X$ are open (the topology on $X$ is thus the discrete one) and all functions are measurable (and even continuous). Therefore by choosing for each $x$ in $X$, some element $p_x^\varepsilon$ in the non-empty set $M_x^\varepsilon$, we get a measurable kernel $X \to \mathcal{P}_\gamma(X) : x \mapsto p_x^\varepsilon$. The same conclusion follows.

To complete the proof, one needs to justify that Condition $(C_1)$ follows from Condition $(C_3'')$.
Assume that $(X,d)$ is a countable set of isolated points and that for all $x\in X$, the function $p\mapsto c(x,p)$ is lower-semicontinuous and let us show that $\pi \mapsto I_c[\pi]$ is lower semicontinuous on $\Pi(\mu,\,\cdot\,).$ Let $(\pi_n)_n$ be a sequence in $\Pi(\mu,\cdot)$ converging to some $\pi$ for the topology $\sigma(\mathcal{P}_\gamma(X\times X))$. Write $\pi_n(dxdy)=p_{x,n}(dy)\,\mu(dx)$ and denote by $\nu_n$ (resp. $\nu$) the second marginal of $\pi_n$ (resp. $\pi$). The sequence $\nu_n$ converges to $\nu$, therefore it is relatively compact and so according to Theorem \ref{thm:Prokhorov}, for all $\varepsilon>0$, there is some compact $K_\varepsilon \subset X$ (i.e. a finite set) such that $\int_{K_\varepsilon^c} \gamma (d(x_o,y))\,\nu_n(dy) \leq \varepsilon$, where $x_o$ is some fixed point in $X.$ In other words, 
\[
\sum_{y\in K_\varepsilon^c} \sum_{x\in X} \gamma(d(x_o,y)) p_{x,n}(\{y\})\mu(\{x\}) \leq \varepsilon\,.
\]
In particular, for all $x\in X$ in the support of $\mu$, it holds 
$$
\sum_{y\in K_\varepsilon^c} \gamma(d(x_o,y)) p_{x,n}(\{y\}) \leq \varepsilon/\mu(\{x\}), 
$$
and so, according to Theorem \ref{thm:Prokhorov}, $\{p_{x,n} ; n \in \N\}$ is relatively compact. 
Without loss of generality (extracting a subsequence if necessary), one can assume that $I_c[\pi_n]=\int c(x,p_{x,n})\,\mu(dx)$ converges. Since for all $x$ in the support of $\mu$, $\{p_{x,n} ; n\in\N\}$ is relatively compact,  the classical diagonal extraction argument enables us to construct an increasing map $\sigma : \N\to \N$ such that $\tilde{p}_{x,\sigma(n)}$ converges to some $p_x \in \mathcal{P}_\gamma(X)$ as $n \to \infty$, for all $x$ in the support of $\mu.$ Finally, using Fatou's lemma and the lower-semicontinuity of $p\mapsto c(x,p)$, one gets
\begin{align*}
\lim_{n\to \infty} I_c[\pi_n] 
& = \lim_{n\to \infty} \int c(x,p_{x,\sigma(n)})\,\mu(dx) \\
& \geq 
\int \liminf_{n\to\infty}c(x,p_{x,\sigma(n)})\,\mu(dx)  \\
& \geq 
\int c(x,p_{x})\,\mu(dx).
\end{align*}
It remains to show that the last term is equal to $I_c[\pi].$ But if $f:X\times X \to \R$ is bounded (continuous), then by dominated convergence,
\begin{align*}
\int f(x,y)\,\pi(dxdy)
&=
\lim_{n\to \infty}\int f(x,y)\pi_{\sigma(n)}(dxdy) \\
& = 
\lim_{n\to \infty} \int \left(\int f(x,y)\,p_{x,\sigma_n(x)}(dy)\right)\,\mu(dx)\\
& = \int \left(\int f(x,y)\,p_{x}(dy)\right)\,\mu(dx).
\end{align*}
Since this holds for all $f$, one concludes that $p_x(dy)\mu(dx) = \pi(dxdy)$ and so in particular, $\int c(x,p_{x})\,\mu(dx)=I_c[\pi]$\,, which completes the proof.
\endproof

In the proof of Theorem \ref{Kantorovich general} we used the following results, elementary proofs of which can be found in \cite{BS78} (see Proposition 7.34 and Proposition 7.33).

\begin{lem}\label{lem:measurable sel}
Let $X$ be a metrizable space, $Y$ a metrizable and separable space and $H \colon X \times Y \to \R\cup\{+\infty\}$ be an upper-semicontinuous function. Denoting by $\overline{H}(x) = \inf _{y\in Y} H(x,y) \in \R\cup\{\pm \infty\}$, for all $\varepsilon>0$, there exists a measurable function $x\mapsto s^\varepsilon(x)$ such that
\[
H(x,s^{\varepsilon}(x)) \leq\left\{\begin{array}{ll} \overline{H}(x) + \varepsilon
 & \text{ if } \overline{H}(x)>-\infty \\-1/\varepsilon & \text{ if } \overline{H}(x)=-\infty .
 \end{array}\right. \]
\end{lem}

\begin{lem}\label{lem:measurable selbis}
Let $X$ be a metrizable space, $Y$ a compact metrizable  space and $H \colon X \times Y \to \R\cup\{+\infty\}$ be a lower-semicontinuous function. Then  there exists a measurable function $x\mapsto s(x)$ such that for all $x\in X$
\[
H(x,s(x)) =\inf _{y\in Y} H(x,y) . \]
\end{lem}

\subsection{Proofs of Theorems \ref{Kantorovich T tilda}, \ref{Kantorovich T bar} and \ref{Kantorovich T hat}.}

\subsubsection{Proof of the usual Kantorovich duality theorem}

As a warm up, let us begin with the proof of the classical Kantorovich duality that we restate below.
\begin{thm}\label{Kantorovich classical} Let $(X,d)$ be a complete separable metric space.
Assume 
that $\omega:X\times X \to [0,\infty]$ is some lower-semicontinuous cost function.Then it holds,
\begin{equation} \label{kanto}
\mathcal{T}_\omega (\nu,\mu) = \sup_{\varphi \in \mathcal{C}_b(X)} \left\{ \int Q_\omega \varphi(x)\,\mu_*(dx) - \int \varphi(y)\,\nu(dy) \right\},\quad  \mu,\nu \in \mathcal{P}(X),
\end{equation}
where $\mu_*$ denotes the inner measure induced by $\mu$ and 
$$
Q_\omega \varphi(x)=\inf_{y\in X} \left\{ \varphi(y) + \omega(x,y)\right\}, \qquad x\in X, \quad\varphi\in \mathcal{C}_b(X).
$$
\end{thm}

\proof[Proof of Theorem \ref{Kantorovich classical}]
First assume that $\omega \colon X\times X \to [0,\infty)$ is continuous and  bounded from above. Then $c(x,p) = \int \omega(x,y)\,p(dy)$ is convex in $p$ and continuous on $X \times \mathcal{P}(X)$, with $\mathcal{P}(X)$ equipped with the usual weak topology. Moreover $I_c[\pi] = \int \omega(x,y)\,\pi(dxdy)$ and so $\pi \mapsto I_c[\pi]$ is continuous on $\mathcal{P}(X\times X)$. So assumptions ${(C_1), (C_2), (C_3),(C_4)}$ of Theorem \ref{Kantorovich general} are fulfilled with $\mathcal{P}_\gamma(X)=\mathcal{P}(X)$ and $\Phi_{\gamma,b}=\Phi_0$. It follows that
\[
\mathcal{T}_\omega(\nu,\mu) = \sup_{\varphi \in \Phi_0(X)}\left\{ \int R_c\varphi(x)\,\mu(dx) - \int \varphi(y)\,\nu(dy) \right\},
\]
with 
\begin{align*}
R_c\varphi(x) 
 = 
\inf_{p \in \mathcal{P}(X)} \left\{ \int \varphi(y) + \omega(x,y)\,p(dy)\right\} 
 = 
\inf_{y\in X} \left\{  \varphi(y) + \omega(x,y) \right\}
 = 
Q_c\varphi(x),
\end{align*}
which completes the proof in the case of a bounded continuous cost function.
Once Kantorovich duality is established for bounded continuous cost functions, one can apply a rather standard approximation argument to extend the duality to lower-semicontinuous cost functions. This is explained for instance in \cite[Point 3 in the proof of Theorem 1.3]{Vil03}.
\endproof

\subsubsection{Proof of Theorem  \ref{Kantorovich T tilda}}
\proof[Proof of Theorem \ref{Kantorovich T tilda}]
Depending on the assumption on the space and on $\alpha$, one needs to verify that Condition $(C)$, $(C')$ or $(C'')$ of Theorem \ref{Kantorovich general} is satisfied. We distinguish between the different cases.

Assume first that $\alpha \colon \R_+\to \R_+$ is convex and continuous. Then the cost $c(x,p) = \alpha\left( \int \gamma(d(x,y))\,p(dy)\right)$ is clearly convex with respect to $p$ and, by definition of the topology $\sigma(\mathcal{P}_\gamma(X))$, it is continuous on $X \times \mathcal{P}_\gamma(X)$ (equipped with the product topology). So assumptions ${(C_2),(C_3)}$ of Theorem \ref{Kantorovich general} are fulfilled. Condition ${(C_4)}$ follows at once from Jensen's inequality. As for Condition $(C_1)$, let us set $\alpha(t)=+\infty$ for $t<0$, so that $\alpha$ is lower-semicontinuous on $\R$. According to the Fenchel-Legendre duality Theorem \ref{Fenchel-Legendre}, \[\alpha(t) = \sup_{s\geq\alpha'(0)} \{st - \alpha^*(s)\}=\sup_{s\geq0} \{st - \alpha^*(s)\},\] where $\alpha'(0)$ is the non-negative right-derivative of $\alpha$ at point $0$, and $\alpha^*(s) = \sup_{t\geq0} \{ st -\alpha(t) \}$. So 
\begin{align*}
c(x,p) 
 = 
\sup_{s\geq 0} \int s\gamma(d(x,y) )- \alpha^*(s)\,p(dy) 
 = 
\sup_{(s,t) \in \mathrm{epi}(\alpha^*)} \int s \gamma(d(x,y)) -t\,p(dy)  
 = 
\sup_{k\in \N} \int \varphi_k(x,y)\,p(dy),
\end{align*}
with $\varphi_0=0$ and $\varphi_k(x,y) = s_k \gamma(d(x,y))-t_k $, $k\geq 1$ where $(s_k,t_k)_{k\geq1}$ is any dense subset of $\mathrm{epi}(\alpha^*) = \{ (s,t) \in [0,\infty) \times \R ; t \geq \alpha^*(s)\}.$
For all $k\in \N$, $\varphi_k \in  \Phi_\gamma(X\times X)$ and so according to Proposition \ref{prop:C1}, the cost function $c$ verifies $(C_1).$ 

Now assume that  $\alpha:\R\to [0,+\infty]$ is  convex and lower-semi\-con\-ti\-nuous. Then $c$ is also clearly convex with respect to $p$ (hence Condition $(C_2)$ is satisfied). Since $\gamma$ is lower-semicontinuous, there exists an increasing sequence $(\gamma_N)_{N\in \N}$ of Lipschitz continuous functions $\gamma_N:\R_+\to \R_+$ that converges to $\gamma$ (for example $\gamma_N(u)=\inf_{v\in \R}\{\gamma(v)+N|u-v|\}$).  By using  the Fenchel-Legendre duality for $\alpha$ as above and by monotone convergence, one has 
\begin{align*}
c(x,p) & =
 \sup_{(s,t)  \in \mathrm{epi}(\alpha^*)} \sup_{N\in \N}\int s \gamma_N(d(x,y)) -t\,p(dy) \\
 & = \sup_{k\in \N} \int \varphi_k(x,y)\,p(dy),
\end{align*}
with $\varphi_0=0$ and $\varphi_k(x,y) = s_{\ell(k)} \gamma_{N(k)}(d(x,y))-t_{\ell(k)} $, $k\geq 1$, where $(s_l,t_l)_{l\in\N}$ is any dense subset of $\mathrm{epi}(\alpha^*) = \{ (s,t) \in [0,\infty) \times \R ; t \geq \alpha^*(s)\},$ and the map $\N^*\ni k \mapsto (N(k),\ell(k))\in \N\times \N$ is one to one. By Proposition \ref{prop:C1}, the conditions $(C_1)$, and $(C_3')$ are fulfilled when $X$ is compact, and respectively $(C_3'')$  when $X$ is a countable set of isolated points. Condition $(C_4)$ is again a consequence of Jensen's inequality.

The result of the corollary is finally a direct consequence of Theorem \ref{Kantorovich general}.
\endproof

\subsubsection{Proof of Theorem  \ref{Kantorovich T bar}}

\proof[Proof of Theorem \ref{Kantorovich T bar}]\ \\
(1) The proof of the first point is similar to that of Corollary \ref{Kantorovich T tilda}. Namely, if $\theta \colon \R^m\to \R$ is a convex function, assumptions $(C_2), (C_3)$ are satisfied with 
$\gamma=\gamma_{1}$. Since $\theta(x)\geq a \|x\|+b$ for some $a>0$ and $b\in \R$, Condition $(C_4)$ follows easily from Jensen's inequality.
Finally, using Fenchel-Legendre duality for $\theta$, one sees that \[
c(x,p) =\theta\left( x- \int y\,p(dy)\right)= \sup_{(s,t) \in \mathrm{epi}(\theta^*)} \int s\cdot \left(x-y\right) - t\,p(dy),
\]
with $\mathrm{epi}(\theta^*) = \{(s,t) \in \R^m \times \R; \theta^*(s) \leq t\}.$ Taking a dense countable subset $(s_k,t_k)_{k\geq 1}$ of $\mathrm{epi}(\theta^*)$, one concludes that 
$$
c(x,p) = \sup_{k\in \N} \int \varphi_k(x,y)\,p(dy), 
$$
with $\varphi_0=0$ and $\varphi_k(x,y) = s_k (x-y) - t_k$. These functions belong to $\Phi_1(X \times X)$, so according to Proposition \ref{prop:C1}, the cost function $c$ verifies $(C_1)$.

If $\theta \colon \R^m\to (-\infty,+\infty]$ is a lower-semicontinuous convex function, we show similarly that $(C_1)$, $(C_2)$, $(C_4)$ are fulfilled, along with $(C_3')$ when $X$ is compact,  and respectively $(C_3'')$  when $X$ is discrete. 

\noindent (2) Let $\varphi \in \Phi_{1,b}(\R^m)$, it holds for all $x\in \R^m$\,,
\begin{align*}
\overline{Q}_\theta \varphi (x)
& = 
\inf_{p \in \mathcal{P}_1(\R^m)}\left\{ \int \varphi\,dp +\theta\left(x-\int y\,p(dy)\right)\right\}\\
& = 
\inf_{z\in \R^m} \left\{g(z)  +\theta\left(x-z\right) \right\},
\end{align*}
where 
\[
g(z):=\inf \left\{\int \varphi\,dp ; p \in \mathcal{P}_1(\R^m), \int y\,p(dy)=z \right\},\qquad z\in \R^m.
\]
The function $g$ is easily seen to be convex on $\R^m$. This implies that $g\leq \overline{\varphi}$. Let us show that $g\geq \overline{\varphi}$. Since $\varphi$ is bounded from below, there is some $a \in \R$ such that $\varphi(y)\geq a$, for all $y\in \R^m$. Then by the definition of $\overline{\varphi},$ it holds $\overline{\varphi}(y)\geq a$, for all $y\in \R^m.$ Since $\overline{\varphi}\leq \varphi$, it follows that $\overline{\varphi}$ is finite everywhere. As a consequence, one can apply Jensen's inequality: if $p \in \mathcal{P}_1(\R^m)$ is such that $\int y\,p(dy)=z$, then 
\[
\int \varphi(y)\,p(dy) \geq \int \overline{\varphi}(y) \,p(dy) \geq \overline{\varphi}\left( \int y\,p(dy)\right) = \overline{\varphi}(z).
\] 
Optimizing over $p$, one concludes that $g(z)\geq \overline{\varphi}(z)$, for all $z\in \R^m$ and so finally $g=\overline{\varphi}.$ 


\noindent (3) Let $\mu,\nu \in \mathcal{P}_1(\R^m)$ and $\varphi \in \Phi_{1,b}(\R^m)$.
According to Point (2),  since $\overline{\varphi}\leq \varphi$,  it holds
\[
\int \overline{Q}_\theta\varphi\,d\mu -\int \varphi\,d\nu =\int Q_\theta \overline{\varphi}\,d\mu- \int \varphi\,d\nu \leq \int Q_\theta\overline{\varphi}\,d\mu - \int \overline{\varphi}\,d\nu.
\]
The function $\overline{\varphi}$ is convex, bounded from below and, since $\varphi \in \Phi_1(\R^m)$, satisfies $\overline{\varphi}(x)\leq a + b \|x\|$, $x \in \R^m$, for some $a,b\geq 0.$ 
This shows that $\overline{\varphi}\in \Phi_{1,b}(\R^m)$. 
From these considerations, it follows that
\begin{align*}
\Tb_\theta(\nu|\mu)&\leq \sup\left\{ \int Q_\theta\overline{\varphi}\,d\mu - \int \overline{\varphi}\,d\nu; \varphi \in \Phi_{1,b}(\R^m) \right\}\\
& \leq \sup\left\{ \int Q_\theta \psi\,d\mu - \int \psi\,d\nu; \psi \in \Phi_{1,b}(\R^m) \text{ convex}\right\}\\
& \leq \sup\left\{ \int \overline{Q}_\theta \psi\,d\mu - \int \psi\,d\nu; \psi \in \Phi_{1,b}(\R^m) \right\}\\
& = \Tb_\theta(\nu|\mu).
\end{align*}
The third inequality is  a consequence of Point (2), since $\psi=\overline{\psi}$ for all convex functions $\psi\in  \Phi_{1,b}(\R^m)$.
Remarking that a convex function belongs to $\Phi_1(\R^m)$ if and only if it is Lipschitz, the proof of Point (3) is complete.
\endproof

\subsubsection{Proof of Theorem  \ref{Kantorovich T hat}}
We start with an alternative  representation of $c(x,p)$ that will be useful subsequently. We recall that $c : X \times \mathcal{P}_\gamma (X) \to \R_+$ is defined by 
\[
c(x,p) = \int \beta\left( \gamma(d(x,y)) \frac{dp}{d\mu_0}(y)\right)\, \mu_0(dy)\,,\]
if $p\ll \mu_0$ on $X\setminus \{x\}$ and $+\infty$ otherwise, where $\mu_0$ is a reference probability measure and $\beta : \R_+\to [0,\infty]$ is a lower-semicontinuous convex function such that $\beta(0)=0$ and $\beta(x)/x \to \infty$ as $x \to \infty.$ As before $\gamma:\R_+\to\R_+$ is a lower-semicontinuous function satisfying \eqref{gamma}.

\begin{lem}\label{duality cost}
Let $X$ be a metric space being either compact or a countable set of isolated points. The cost function $c$ defined above 
satisfies the following duality identity:
\[
c(x,p)=\sup_{h\in\Phi_0(X), h\geq 0}\left\{\int h(y) \gamma(d(x,y)) \,p(dy) -\int \beta^*(h)(y) \,\mu_0(dy)\right\},
\]
where $\beta^*(y)=\sup_{x\geq 0}\{xy-\beta(x)\}$, $y\in \R$, denotes the Fenchel-Legendre transform of $\beta$.
\end{lem}
\proof
The proof is easily adapted from Theorem B.2  in \cite{LV09}.
\endproof

\proof[Proof of Theorem \ref{Kantorovich T hat}]
First, we observe that Condition $(C_2)$ is a simple consequence of the convexity of $\beta$
and Condition $(C_4)$ of Jensen's inequality.
According to Lemma \ref{duality cost}, it holds
\begin{align}
c(x,p)&=\sup_{h\in\Phi_0(X), h\geq 0}\left\{\int h(y) \gamma(d(x,y)) \,p(dy) -\int \beta^*(h)(y) \,\mu_0(dy)\right\}\\ \nonumber
&=\sup_{h\in\Phi_0(X), h\geq 0}\sup_{N\in \N} \int (h(y) \gamma_N(d(x,y)) - B^*(h) )\,p(dy),
\end{align}
where  $(\gamma_N)_{N\in \N}$ is (as in the proof of Corollary \ref{Kantorovich T tilda}) an increasing sequence of Lipschitz continuous functions converging to $\gamma$ and $B^* (h)= \int \beta^*(h)d\mu_0$. 

For all $h \in \Phi_0(X)$ non-negative and all $N\in \N$, the function $(x,y) \mapsto h(y) \gamma_N(d(x,y))$ is continuous. Therefore, $p \mapsto \int h(y) \gamma(d(x,y)) p(dy)$ is a continuous function on $X\times \mathcal{P}_\gamma(X)$. Being a supremum of continuous functions, $c$ is lower-semicontinuous on $X\times \mathcal{P}_\gamma(X).$ In particular, this shows $(C_3')$ and $(C_3'')$.

Next we will check that Condition $(C_1)$ holds (in the compact case).

Since $(X,d)$ is compact, the space $\Phi_0(X)$ of continuous functions (equipped with the norm $\|\,\cdot\,\|_\infty$) on  $X$ is separable (see \cite[Proposition 7.7]{BS78}). Let $\{h_\ell,\ell\in\N\}$ be a countable dense subset of $\Phi_0(X)$. Since $\beta^*$ is convex and finite on $\R$ it is continuous on $\R$. Therefore, the function $\Phi_0(X) \to \R : h\mapsto B^*(h)$ is continuous. It follows that
\begin{equation}\label{cafe}
c(x,p)=\sup_{k\in \N}\int \varphi_k(x,y) \,p(dy),\qquad \forall x\in X,\qquad p \in \Phi_\gamma(X)\,,
\end{equation}
where  $\varphi_0=0$ and $\varphi_k(x,y)= h_{\ell(k)}(y)\gamma_{N(k)}(d(x,y))-B^*(h_\ell(k))$, $k\geq 1$, and $\N^* \ni k \mapsto (\ell(k),N(k))\in \N\times\N$ is one-to-one. Since, for all $k\in \N$, the function $\varphi_k$ belongs to $\Phi_\gamma(X,X)$, the lower-semicontinuity of $I_c$ follows from  Proposition \ref{prop:C1}. 
%
%
%

Corollary \ref{Kantorovich T hat} now follows from Theorem \ref{Kantorovich general}.
\endproof

\subsection{Proof of Proposition \ref{prop:C1}}\label{sec:propC1}


The proof of Proposition \ref{prop:C1} is adap\-ted from \cite[Theorem 2.34]{AFP00}. 
\proof[Proof of Proposition \ref{prop:C1}]
 The function $p\mapsto c(x,p)$ is convex as a supremum of linear functions.
 
For all $n\in \N$, define $c_n(x,p) := \sup_{k\leq n} \int \varphi_k(x,y)\,p(dy)$. When $n$ goes to $\infty$, $c_n(x,p)$ is a nondecreasing sequence converging to $c$. 
Let $\pi \in \Pi(\mu,\,\cdot\,)$, $\pi(dxdy) = p_x(dy)\mu(dx)$ such that \eqref{coco} holds for $\mu$-almost all $x$. 
Defining $I_{c_n}[\pi] = \int c_n(x,p_x)\,\mu(dx)$, the monotone convergence theorem shows that $I_c[\pi] = \sup_{n\in \N} I_{c_n}[\pi]$. 
Since a supremum of lower-semicontinuous functions is itself lower-semicontinuous, it is enough to prove that $ I_{c_n}$ is lower-semicontinuous at point $\pi$. We will now prove such a property.

For $\mu$-almost all $x$, define $\psi_k(x) = \int \varphi_k(x,y)\,p_x(dy)$, $k\leq n.$
Then it holds
\[
I_{c_n}[\pi] = \int  \sup_{k\leq n} \psi_k(x)\,\mu(dx) = \sup_{(f_k)_{k\leq n}} \int \sum_{k=0}^n f_k(x)\psi_k(x)\,\mu(dx),
\]
where the supremum runs over the set of continuous functions $f_k$ taking values in $[0,1]$ and such that $f_0+\cdots+f_n \leq 1$. Let us admit this claim for a moment and finish the proof of the proposition. 
For all $f_0,\ldots,f_n$ as above, it holds
\[
 \int \sum_{k=0}^n f_k(x)\psi_k(x)\,\mu(dx) = \int \sum_{k=0}^n f_k(x)\varphi_k(x,y) \,\pi(dxdy)\,.
\]
Since $\sum_{k=0}^nf_k \varphi_k \in \Phi_\gamma(X\times X)$, the function $\pi \mapsto \int \sum_{k=0}^nf_k\varphi_k\,d\pi$ is continuous on $\Pi(\mu,\,\cdot\,)$. Since a supremum of continuous functions is lower-semicontinuous, this proves that $ I_{c_n}$ is lower-semicontinuous at point $\pi$\,.

It remains to prove the claim.
Obviously, if $f_0,f_1,\ldots,f_n$  take values in $[0,1]$ and are such that $\sum_{k=0}^n f_k \leq 1$, then it holds
\begin{align*}
\int \sum_{k=0}^n f_k(x)\psi_k(x)\,\mu(dx) 
\leq 
\int \sum_{k=0}^n f_k(x)[\psi_k]_+(x)\,\mu(dx) 
&\leq  
\int \sup_j[\psi_j]_+(x)\sum_{k=0}^n f_k(x)\,\mu(dx)\\
&\leq \int \sup_j[\psi_j]_+(x)\,\mu(dx) = I_{c_n}[\pi],
\end{align*}
where the last equality comes from the fact that $\sup_j[\psi_j]_+ = \sup_{j}\psi_j $ since $\varphi_0=0$ and $\psi_0=0$.  This shows that 
$$
I_{c_n}[\pi] \geq \sup_{(f_k)_{k\leq n}} \int \sum_{k=0}^n f_k(x)\psi_k(x)\,\mu(dx).
$$ 

To prove the converse inequality, let for all $k\leq n$, $A_k := \{x\in X ; [\psi_k]_+ = \sup_{j} [\psi_j]_+(x)\}$, and define recursively $B_0 = A_0$, $B_{k} = A_k\setminus (B_0\cup \cdots\cup B_{k-1})$.
Then it holds
\[
I_{c_n}[\pi] = \sum_{k=0}^n \int_{B_k} [\psi_k]_+(x)\,\mu(dx)\,.
\]
When $(X,d)$ is a discrete space, the functions $f_k=\1_{B_k}$ are continuous and $\sum_{k=0}^n f_k=1$. Since 
$\psi_k$ is non-negative on $A_k$, one has
\[I_{c_n}[\pi] = \sum_{k=0}^n \int f_{k}(x)\psi_k(x)\,\mu(dx),\]
and the claim follows in this case.

Assume now that $(X,d)$ is  complete and separable. For all $k\leq n$, consider the finite Borel measure $\mu_k (dx) = [\psi_k]_+(x)\,\mu(dx)$. Let $\varepsilon>0$ ; since finite Borel measures on a complete separable metric space are inner regular (see for instance \cite[Theorems 3.1 and 3.2]{Par67}), for all $k \leq n$ there is a compact set $C_k \subset B_k$ such that $\mu_k(B_k) \leq \mu_k(C_k) + \varepsilon/(n+1)$. So it holds
\begin{align*}
I_{c_n}[\pi] 
 = 
\sum_{k=0}^n \int_{B_k} [\psi_k]_+(x)\,\mu(dx) 
 \leq 
\sum_{k=0}^n \int_{C_k} [\psi_k]_+(x)\,\mu(dx) + \varepsilon 
 = 
\sum_{k=0}^n \int_{C_k} \psi_k(x)\,\mu(dx) + \varepsilon.
\end{align*}
 The compact sets $C_k$ are pairwise disjoint, so $\delta_o = \min_{i\neq j} d(C_i,C_j) >0.$
Consider the family of continuous functions $f_{k,\delta} : X \to [0,1]$ defined by
\[
f_{k,\delta}(x) = \left[1- \frac{d(x,C_k)}{\delta}\right]_+,\qquad x\in X, \qquad k\leq n,\qquad \delta>0.
\]
When $\delta<\delta_o/2$, for any $x\in X$, at most one of the functions is not zero at $x$ and therefore $\sum_{k=0}^n f_{k,\delta}(x) \leq1.$ Passing to the limit when $\delta \to 0$, we see that 
\[
\sum_{k=0}^n \int f_{k,\delta}(x)\psi_k(x)\,\mu(dx) \to \sum_{k=0}^n \int_{C_k} \psi_k(x)\,\mu(dx).
\]
So if $\delta$ is small enough it holds
\[
I_{c_n}[\pi] \leq \sum_{k=0}^n \int f_{k,\delta}(x)\psi_k(x)\,\mu(dx) + 2\varepsilon.
\]
Taking the supremum over all possible functions $f_k$, and then letting $\varepsilon$ go to $0$, gives the desired inequality 
$$
I_{c_n}[\pi] \leq \sup_{(f_k)_{k\leq n}} \int \sum_{k=0}^n f_k(x)\psi_k(x)\,\mu(dx),
$$
and completes the proof.
\endproof


\appendix

\section{Proof of Theorem \ref{tensorization}}\label{appendix-tensorization}
The proof of the tensorization property for transport-entropy inequalities uses the chain rule formula for the entropy on the one hand, and on the other, a similar property for the transport cost, which we now state in the following lemma of independent interest.

\begin{lem}[Chain rule inequality for the transport cost] \label{chainrule}
Let  $\gamma \colon$ $\R_+ \to \R_+$ be a lower-semicontinuous function satisfying \eqref{gamma},
 $(X_1,d_1)$, $(X_2,d_2)$ be complete separable metric spaces equipped with cost functions $c_i:X_i \times \mathcal{P}_\gamma(X_i) \to [0,\infty]$, $i\in\{1,2\}$ such that $c_i(x_i,\delta_{x_i})=0$  and $p_i \mapsto c_i(x_i,p_i)$ is convex for all $x_i\in X_i$. Define $c \colon X_1 \times  X_2 \times \mathcal{P}_\gamma(X_1\times X_2)\to [0,\infty)$ by 
 $c(x,p) = c_1(x_1,p_1) + c_2(x_2,p_2)$, $x=(x_1,x_2) \in X_1 \times X_2$, $p \in \mathcal{P}_\gamma(X_1 \times X_2)$, where $p_i$ denotes the $i$-th marginal distribution of $p$.
 
Then, for all $\nu, \nu' \in \mathcal{P}_\gamma(X_1 \times X_2)$, 
all $\varepsilon >0$, there exists a kernel $p_1^\varepsilon$ such that 
\[
\mathcal{T}_c (\nu' | \nu) \leq \mathcal{T}_{c_1} (\nu_1' | \nu_1) 
+ \int_{X_1 \times X_1} \!\!\!\!\!\!\!\!\!\!\!\mathcal{T}_{c_2} (\nu_2'(y_1,\,\cdot\,) | \nu_2(x_1,\,\cdot\,)) p_1^\varepsilon(x_1,dy_1)\nu_1(dx_1) + 2 \varepsilon\,,
\]
where $\nu_1$ and $\nu'_1$ are the first marginals of $\nu, \nu'$ respectively; the kernels $x_1 \mapsto \nu_2(x_1,\,\cdot\,)$
and  $y_1 \mapsto {\nu'}_2(y_1,\,\cdot\,)$ are such that
\[
\nu(dx_1dx_2)=\nu_1(dx_1)\nu_2(x_1,dx_2) \; \text{and} \;
\nu'(dy_1dy_2)=\nu_1'(dy_1)\nu'_2(y_1,dy_2) ;
\]
and the kernel $p_1^\varepsilon$, defined so that $\pi_1^\varepsilon(dx_1dy_1):=\nu_1(dx_1)p_1^\varepsilon(x_1,dy_1) \in \Pi(\nu_1,\nu_1')$, 
satisfies 
$\mathcal{T}_{c_1} (\nu_1 | \nu'_1) \geq \int_{X_1 \times X_1} c_1(x_1,p_1^\varepsilon(x_1,\,\cdot\,))\,\nu_1(dx_1) - \varepsilon$.
\end{lem}

\begin{rem}
If one assumes that the cost functions $c_1$ and $c_2$ satisfy assumption $(C_1)$, then the error term $\varepsilon$ can be chosen $0$. Indeed, under assumption $(C_1)$ the function $\pi \mapsto \int c_1(x,p_x)\,\nu_1'(dx_1)$ is lower semicontinuous on the set $\Pi(\nu_1',\nu_1)$ which is easily seen to be compact (using Theorem \ref{thm:Prokhorov} below). Therefore it attains its infimum and so there exists some kernel $p_1$ such that $\mathcal{T}_{c_1}(\nu_1|\nu_1') = \int c_1(x_1,p_1(x_1,\,\cdot\,))\,\nu_1'(dx_1).$ The same applies for cost functions based on the cost $c_2$.
\end{rem}

\proof[Proof of Lemma  \ref{chainrule}]
Fix $\nu, \nu' \in \mathcal{P}_\gamma(X_1 \times X_2)$ and $\varepsilon >0$. Our aim is first to define a  probability kernel $p$ appropriately related to $\nu$ and  $\nu'$.

To that purpose, let 
$p_1$ be a probability kernel (that depends on $\varepsilon$ although not explicitly stated for simplicity) so that $\pi_1(dx_1dy_1):=\nu_1(dx_1)p_1(x_1,dy_1) \in \Pi(\nu_1,\nu_1')$ and 
\begin{equation} \label{1}
\int_{X_1 \times X_1} c_1(x_1,p_1(x_1,\,\cdot\,))\, \nu_1(dx_1) 
\leq 
 \mathcal{T}_{c_1} (\nu_1' | \nu_1)  + \varepsilon  .
\end{equation}
Similarly, for all $x_1, y_1 \in X_1$, let $X_2 \ni x_2 \mapsto q_2^{x_1,y_1}(x_2,\,\cdot\,) \in \mathcal{P}(X_2)$  be a probability kernel (that depends also on $\varepsilon$) satisfying
$\pi_2^{x_1,y_1}(dx_2dy_2) := \nu_2(x_1,dx_2)q_2^{x_1,y_1}(x_2,dy_2)  \in \Pi(\nu_2(x_1,\,\cdot\,),\nu'_2(y_1,\,\cdot\,))$ and
\begin{equation} \label{2}
\int_{X_2 \times X_2} c_2(x_2,q_2^{x_1,y_1}(x_2,\,\cdot\,)) \,\nu_2(x_1,dx_2) 
\leq
\mathcal{T}_{c_2} (\nu_2'(y_1,\,\cdot\,) | {\nu}_2(x_1,\,\cdot\,)) 
+
 \varepsilon \,.
\end{equation}
Then observe that, for all $f : X_1 \times X_2 \to \R$,  it holds:
\begin{align*}
 \int f(y_1,y_2)& p_1(x_1,dy_1)  q_2^{x_1,y_1}(x_2,dy_2) \,\nu(dx_1dx_2) \\
& =
\int f(y_1,y_2) p_1(x_1,dy_1)q_2^{x_1,y_1}(x_2,dy_2) \,\nu_2(x_1,dx_2) \nu_1(dx_1)\\
& =
\int f(y_1,y_2) p_1(x_1,dy_1)\nu'_2(y_1,dy_2) \nu_1(dx_1) \\
& =
\int f(y_1,y_2) \nu'_2(y_1,dy_2) \nu'_1(dy_1) \\
& =
\int f(y) \nu'(dy) .
\end{align*}
Hence, $p(x,dy):=p_1(x_1,dy_1)q_2^{x_1,y_1}(x_2,dy_2)$ is a probability kernel satisfying
$\pi(dxdy) :=p(x,dy) \nu(dx) \in \Pi(\nu,\nu')$.
Let 
$$
p_2(x,\,\cdot\,):=\int_{X_1} p_1(x_1,dy_1)q_2^{x_1,y_1}(x_2,\,\cdot\,) \in \mathcal{P}(X_2)
$$ 
be the second marginal of $p(x_1,\,\cdot\,)$,  observing that $p_1(x,\cdot)$ is its first marginal. 

Finally, using the definition of the transport cost, the definition of the cost and
Jensen's inequality, it holds:
\begin{align*}
\mathcal{T}_c (\nu'| \nu) 
& \leq 
\int_{X_1 \times X_2} \! c(x,p)\nu(dx) \\
& =
\int_{X_1} \! c_1(x_1,p_1(x_1,\,\cdot\,)) \nu_1(dx_1) + \int_{X_1 \times X_2} \! c_2(x_2,p_2(x,\,\cdot\,)) \nu(dx) \\
& \leq 
 \mathcal{T}_{c_1} (\nu_1' | \nu_1)  + \varepsilon + 
\int_{X_1^2 \times X_2} c_2(x_2, q_2^{x_1,y_1}(x_2,\,\cdot\,)) p_1(x_1,dy_1) \nu(dx) \\
& =
 \mathcal{T}_{c_1} (\nu_1' | \nu_1)  + \varepsilon + 
\int_{X_1^2} \left( \int_{X_2} c_2(x_2, q_2^{x_1,y_1}(x_2,\,\cdot\,)) \nu_2(x_1,dx_2) \right) p_1(x_1,dy_1) \nu_1(dx_1) \\
& \leq 
\mathcal{T}_{c_1} (\nu_1' | \nu_1)  + \varepsilon  + 
\int_{X_1^2} \left(  \mathcal{T}_{c_2} (\nu_2'(y_1,\,\cdot\,) | \nu_2(x_1,\,\cdot\,)) + \varepsilon \right) p_1(x_1,dy_1) \nu_1(dx_1)\,,
\end{align*}
where the last two inequalities follow from \eqref{1} and \eqref{2} respectively.
The expected result follows and the proof of the lemma is complete.
\endproof

\proof[Proof of Theorem \ref{tensorization}]
By induction, it is enough to consider the case $n=2$. 
Given $\nu, \nu' \in \mathcal{P}_\gamma(X_1 \times X_2)$, thanks to Lemma \ref{chainrule}, for
all $\varepsilon >0$, there exists a kernel $p_1^\varepsilon$ such that 
\[
\mathcal{T}_c (\nu' | \nu) \leq \mathcal{T}_{c_1} (\nu_1' | \nu_1) 
+ \int_{X_1 \times X_1} \!\!\!\!\!\!\!\!\!\! \mathcal{T}_{c_2} (\nu_2'(y_1,\,\cdot\,) | \nu_2(x_1,\,\cdot\,)) p_1^\varepsilon(x_1,dy_1)\nu_1(dx_1) + 2 \varepsilon\,,
\]
where $\nu, \nu_1', \nu_2,\nu_2'$ are defined in Lemma \ref{chainrule}.
Applying the transport-entropy inequalities that hold for $\mu_1$ and $\mu_2$, we get
\begin{align*}
& \mathcal{T}_c (\nu' | \nu) 
 \leq 
a_1^{(1)} H(\nu_1' | \mu_1) + a_2^{(1)} H(\nu_1 | \mu_1 ) + 2 \varepsilon  \\
&  +
\int_{X_1 \times X_1} \left[ a_1^{(2)} H(\nu_2'(y_1,\,\cdot\,) | \mu_2) + a_2^{(2)} 
H(\nu_2(x_1,\,\cdot\,)|\mu_2) \right] p_1^\varepsilon(x_1,dy_1)\nu_1(dx_1) \\
& \leq
a_1 \left[ H(\nu_1' | \mu_1) + \int_{X_1} H(\nu_2'(y_1,\,\cdot\,) | \mu_2) \nu'_1(dy_1) \right] \\
& \quad + 
a_2 \left[ H(\nu_1 | \mu_1 ) + \int_{X_1} H(\nu_2(x_1,\,\cdot\,) | \mu_2) \nu_1(dx_1) \right] 
+ 2 \varepsilon \\
& =
a_1 H(\nu'|\mu) + a_2 H(\nu|\mu) + 2 \varepsilon\,,
\end{align*}
where we used that $\int_{X_1} p_1^\varepsilon(x_1,dx_1') =1$, 
$\int_{X_1} p_1^\varepsilon(x_1,\,\cdot\,)\nu_1(dx_1) = \nu'_1(\,\cdot\,)$ and the chain rule formula for the entropy (recall that we set $a_1:=\max (a_1^{(1)}, a_1^{(2)})$ and $a_2:=\max (a_2^{(1)}, a_2^{(2)})$). Letting $\varepsilon$ go to zero completes the proof of the theorem.
\endproof

\begin{rem}
Alternatively, following \cite{Sam07}, one could give a proof based on the dual characterization of Proposition \ref{prop:bg}.
\end{rem}


\bibliographystyle{amsplain}
\def\cprime{$'$}
\providecommand{\bysame}{\leavevmode\hbox to3em{\hrulefill}\thinspace}
\providecommand{\MR}{\relax\ifhmode\unskip\space\fi MR }
\providecommand{\MRhref}[2]{%
  \href{http://www.ams.org/mathscinet-getitem?mr=#1}{#2}
}
\providecommand{\href}[2]{#2}

\end{document}